\documentclass[12pt]{amsart}
\usepackage{fullpage,url,amssymb,enumerate,colonequals}

\usepackage{stmaryrd}

\usepackage{todonotes}
\usepackage{adjustbox}

\usepackage[ruled,vlined]{algorithm2e}

\usepackage{color}
\usepackage[most]{tcolorbox}

\usepackage[
        colorlinks, citecolor=darkgreen,
        pdfauthor={}, 
]{hyperref}
\usepackage{cleveref}
\usepackage{comment}
\usepackage{tikz-cd}

\newcommand{\F}{\mathbb{F}}
\newcommand{\D}{\mathfrak{D}}
\newcommand{\Fbar}{{\overline{\F}}}

\newcommand{\Q}{\mathbb{Q}}

\newcommand{\Z}{\mathbb{Z}}

\renewcommand{\SS}{\operatorname{SS}}
\newcommand{\SScy}{{\operatorname{SS}}^{\ast}}
\newcommand{\SSpr}{{\operatorname{SS}^{*}_{\calO}(p)}}

\newcommand{\calO}{\mathcal{O}}

\renewcommand{\a}{\mathfrak{a}}

\newcommand{\frakO}{\mathfrak{O}}

\DeclareSymbolFont{matha}{OML}{txmi}{m}{it}
\DeclareMathSymbol{\varv}{\mathord}{matha}{118}
\DeclareMathSymbol{\varw}{\mathord}{matha}{119}

\DeclareMathOperator{\Aut}{Aut}
\DeclareMathOperator{\Br}{Br}

\DeclareMathOperator{\Cls}{Cls}
\DeclareMathOperator{\typee}{Typ}

\DeclareMathOperator{\disc}{disc}

\DeclareMathOperator{\discrd}{discrd}
\DeclareMathOperator{\End}{End}

\DeclareMathOperator{\Frac}{Frac}

\DeclareMathOperator{\Gal}{Gal}

\DeclareMathOperator{\Hom}{Hom}

\DeclareMathOperator{\im}{im}

\DeclareMathOperator{\lcm}{lcm}

\DeclareMathOperator{\Pic}{Pic}

\DeclareMathOperator{\tr}{tr}
\DeclareMathOperator{\trd}{trd}
\DeclareMathOperator{\nrd}{nrd}

\DeclareMathOperator{\cl}{Cl}
\DeclareMathOperator{\rad}{rad}

\DeclareMathOperator{\sep}{sep}
\DeclareMathOperator{\Idl}{Idl}
\DeclareMathOperator{\PIdl}{PIdl}

\newcommand{\M}{\text{M}}
\newcommand{\id}{\text{Id}}

\newcommand{\GL}{\operatorname{GL}}

\newcommand{\PGL}{\operatorname{PGL}}

\numberwithin{equation}{section}  

\newtheorem{theorem}{Theorem}[section]    
\newtheorem{lemma}[theorem]{Lemma}        
\newtheorem{corollary}[theorem]{Corollary}
\newtheorem{proposition}[theorem]{Proposition}
\newtheorem{example}[theorem]{Example}

\newtheorem{definition}{Definition}[section]  

\newtheorem{problem}[definition]{Problem}

\theoremstyle{remark}
\newtheorem{remark}{Remark}[section]  

\definecolor{darkgreen}{rgb}{0,0.5,0}

\newenvironment{tbox}{\begin{tcolorbox}[ colframe=red, colback= white, breakable]} {\end{tcolorbox}}

\begin{document}

\title{Double-orientations on supersingular isogeny graphs}

\author{Do Eon Cha}

\address{Department of Mathematics, University of Calgary\\
Calgary, AB T2N 1N4, Canada } \email{andrew.cha1@ucalgary.ca}

\author{Imin Chen}

\address{Department of Mathematics, Simon Fraser University\\
Burnaby, BC V5A 1S6, Canada } \email{ichen@sfu.ca}

\date{\today}


\begin{abstract} 
We recall and define various kinds of supersingular $\ell$-isogeny graphs and precise graph isomorphism with a corresponding quaternion $\ell$-ideal graph. In particular, we introduce the notion of double-orientations on supersingular elliptic curves and study the structure of double-oriented supersingular $\ell$-isogeny graphs.
\end{abstract}

\maketitle

\setcounter{tocdepth}{2}

\tableofcontents


\excludecomment{details}
\excludecomment{extra}

\section{Introduction}

Let $p$ be a prime. For primes $\ell \neq p$, the supersingular $\ell$-isogeny graph in characteristic $p$ is the multi-digraph $G(p,\ell)$ whose vertices are isomorphism classes of supersingular elliptic curves over $\Fbar_p$ and edges are the equivalence classes of $\ell$-isogenies between them.

Supersingular isogeny graphs have been an important object for applications in isogeny-based cryptography. With the recent break of SIKE (Supersingular Isogeny Diffie–Hellman Key Exchange) \cite{SIKE-1, SIKE-2, SIKE-3}, a natural question remains as to the hardness of the pure isogeny path problem, i.e., computing isogenies without any knowledge of extra structures such as torsion point information.

\begin{problem} \label{path finding} (Isogeny Path Finding Problem)
For a pair of elliptic curves $E, E'$ in the supersingular isogeny graph $G(p,\ell)$, find a path from $E$ to $E'$, which is represented by a chain of $m=O(\log p)$ isogenies of degree $\ell$.
\end{problem}

\begin{extra}
    Given a path of length $n$  in $G(p,\ell)$ 
\begin{equation*}
    E_0 \overset{\varphi_1} {\longrightarrow} E_1 \overset{\varphi_2} {\longrightarrow} E_2 \cdots E_{n-1} \overset{\varphi_n} {\longrightarrow} E_n,
\end{equation*}
where $j_i$ is the $j$-invariant of $E_i$ for $0\leq i \leq n$, we will denote the path by $(j_0j_1\cdots j_n)$.
\end{extra}

If $E/\overline{\F}_p$ is a supersingular elliptic curve, then its endomorphism algebra $\End^0(E):=\End(E) \otimes \Q \cong B_0$ where $B_0/\Q$ is a quaternion algebra ramified at $p$ and $\infty$. With this in mind, two other related problems are:

\begin{problem} \label{Endomorphism Computation} (Endomorphism Ring Problem) Given a prime $p$ and a supersingular elliptic curve $E/\Fbar_p$, output endomorphisms $\theta_1, \theta_2, \theta_3, \theta_4$ of $E$ that form a $\Z$-basis of the endomorphism ring $\End(E)$. In addition the output basis is required to have representation size polynomial in $\log p$.  
\end{problem}

\begin{extra}
\begin{theorem}
Given a set $S$ of an even number of places of $\Q$, there is one isomorphism class of quaternion algebras over $\Q$ ramified exactly at $S$.
\end{theorem}
\end{extra}

\begin{problem} \label{Max Order} (Maximal Order Problem)
Given a prime $p$, the standard basis for $B_0$, and a supersingular elliptic curve $E/\overline{\F}_p$, output vectors $\beta_1, \beta_2, \beta_3, \beta_4 \in B_0$ that form a $\Z$-basis of a maximal order $\frakO$ in $B_0$ such that $\End(E) \cong \frakO$. In addition, the output basis is required to have representation size polynomial in $\log p$.
\end{problem}

Polynomial time (in $\log p$) reductions between all three problems are given in \cite{EHL+18} with heuristics and in \cite{Ben-equivalent} without heuristics under the validity of the generalized Riemann hypothesis (GRH). In fact, Problem~\ref{Endomorphism Computation} is equivalent to finding one non-scalar endomorphism \cite{one-endo} on the validity of GRH.

Col\`o and Kohel \cite{CK20} introduced an  isogeny-based cryptographic protocol called OSIDH (Oriented Supersingular Isogeny Diffie-Hellman) with the notion of orientations of supersingular elliptic curves. The oriented isogeny graphs have a structure called a volcano \cite{ACL+23, Onu21}; each connected component consists of a cycle called the rim and there are notions of ascending, horizontal, and descending isogenies, so we can walk toward or on the rim. The volcano structure enriches the isogeny path problem where we relate the curves and paths in oriented isogeny graphs to the ones in regular isogeny graphs. 

In this paper, we extend the notion of orientations of supersingular elliptic curves and investigate the properties of the associated graphs. Our main contributions can be summarized as follows:
\begin{extra}
Path finding for isogeny graphs of ordinary elliptic curves is straightforward as they have a special structure called a volcano; each connected component consists of a cycle called the rim and trees branching out from each vertex on the rim. There are notions of ascending, horizontal, and descending isogenies, so we can 'walk' toward or on the rim. These useful properties are due to the fact that the endomorphism ring of an ordinary curve is isomorphic to a quadratic order. To introduce such notions to supersingular isogeny graphs, one can consider an embedding $\iota: K \hookrightarrow \End^0(E)$ of a quadratic subfield $K \subseteq  B_0$ and the order $\calO \subseteq K$ such that $\iota(\calO)=\End(E) \cap \iota(K)$. Such an embedding is called an orientation. When equipped with orientations, oriented supersingular isogeny graphs admit a volcano structure. 

The first result of this paper shows the following relationship between the two problems.
\begin{theorem}
\label{main-reduce}
There is a probabilistic polynomial time reduction from the Isogeny Path Finding Problem to the Maximal Order Problem.
\end{theorem}
Probabilistic means that running the algorithm with access to random bits succeeds with high probability after a few runs.

The above theorem was proven in \cite{EHL+18} under plausible heuristic assumptions, and in \cite{Ben-equivalent} under GRH using a similar but refined algorithm that is easier to analyze. 
\end{extra}

Classical results due to Deuring on the correspondence between objects on the supersingular elliptic curve side---such as curves, isogenies, and endomorphism rings---and those on the quaternion algebra side---such as maximal orders and their left ideals---are comprehensively reviewed in \cite[Chapter 42]{Voi21}. These results are also widely used in the literature on the path-finding problem in both classical and oriented supersingular isogeny graphs, as well as in related problems such as those studied in \cite{EHL+18, KLPT14, LB20, Ben22}. In Section \ref{Equivalence of categories in terms of quaternion algebras}, we provide refined versions of these results: we will discuss the correspondence between cyclic isogenies, primitive ideals, and maximal orders relative to the endomorphism algebra of a chosen reference curve $E_0$.

Our first main contribution is to extend the notion of orientations of supersingular elliptic curves in \cite{CK20, ACL+23, Ben22} to double-orientations. In Section \ref{Orientations and maximal orders}, we define a $K_1, K_2$-orientation of a supersingular elliptic curve over $\Fbar_p$ to be a pair of $K_1$-orientation and $K_2$-orientation for imaginary quadratic fields $K_1$ and $K_2$ such that they have simultaneous embeddings into $B_0$ and their images generate $B_0$. We then give a correspondence between double-oriented curves---curves together with double-orientations on them---and maximal orders in $B_0$. We also introduce new classes of maximal orders and their left ideals in $B_0$ relative to an embedding of an imaginary quadratic field $K$ into $B_0$, and give correspondences with $K$-oriented curves.
 
As our second main contribution, in Section \ref{Supersingular isogeny graphs and quaternion ideal graphs}, we introduce several variants of supersingular isogeny graphs and quaternion ideal graphs and give explicit graph isomorphisms between them. In particular, we define several new graphs: double-oriented isogeny graphs and single- and double-oriented quaternion ideal graphs. We also describe the connection between Bruhat-Tits trees and the connected components of single- and double-oriented isogeny graphs.

Double-oriented isogeny graphs admit a notion of roots, which are analogous to the curves on the rims of single-oriented isogeny graphs. In a double-oriented $\ell$-isogeny graph, the local roots are the curves primitively oriented by a pair of $\ell$-fundamental orders in the imaginary quadratic fields $K_1$ and $K_2$, while the global roots are the curves primitively oriented by the pair of maximal orders in $K_1$ and $K_2$. In Section \ref{embedding numbers of bass orders and counting roots}, we count the number of local and global roots of double-oriented isogeny graphs using Bass orders in the quaternion algebra and their embedding numbers, and we present an explicit example of global roots.

In Section \ref{Structure of double-oriented isogeny graphs}, we study the double-oriented isogeny graphs in the case where the associated imaginary quadratic fields $K_i$ and $K_j$ are perpendicular with respect to the reduced-trace pairing on $B_0$. We analyze their structure using the correspondence between quaternion ideals of reduced norm $\ell$ and proper ideals of $M_2(\F_\ell)$. Our last main contribution is the characterization of the volcanic ridge, i.e., volcano structure except that the rim is replaced by a ridge, of double-oriented isogeny graphs:

Each double-oriented isogeny graph has two underlying single-oriented isogeny graphs. In Proposition \ref{joint-ascend-general}, we show that when we are moving to a local root on the double-oriented isogeny graph, we can simultaneously ascend or move horizontally to the rims on the both of underlying single-oriented isogeny graphs unless we are already at a local root.   

Each connected component of $G_{K_i, K_j}(p,\ell)$ is described by Theorem~\ref{Formula for sturcture 1}.

\section{Acknowledgements}

We would like to thank Renate Scheidler for helpful comments and suggestions.

\section{Preliminaries}
\subsection{Supersingular isogeny graphs}

We now summarize some basic definitions and results from \cite{BCE+19} concerning supersingular isogeny graphs.

We say two isogenies $\varphi: E \rightarrow E', \varphi': E \rightarrow E''$ are equivalent if they have the same kernel. The equivalence relation on isogenies defines equivalence classes of isogenies and the class of an isogeny is invariant under post-composition by an automorphism.

If there exists an $\ell$-isogeny $\varphi : E \rightarrow E'$, aSch23nd we have isomorphisms $\lambda : E \cong E_1$ and $\lambda' : E' \cong E_1'$, then there exists an $\ell$-isogeny $\varphi_1 : E_1 \rightarrow E_1'$ given by $\varphi_1 = \lambda' \varphi \lambda^{-1}$. Furthermore, if $\varphi : E \rightarrow E'$ is equivalent to $\varphi' : E \rightarrow E''$, then $\varphi_1 : E_1 \rightarrow E_1'$ is equivalent to $\varphi_1' : E_1 \rightarrow E_1''$. Thus, the definition of the edges in $G(p,\ell)$ is well-defined.

The classical modular polynomial $\Phi_n(X,Y) \in \Z[X,Y]$ is initially constructed as a relation between the modular functions $j(\tau)$ and $j(n\tau)$ \cite{Cox} and has the following modular interpretation:
\begin{theorem}
\label{modular-equation}
Given a $j \in \overline{\F}_p$, the roots of $\Phi_n(j,Y)$ give the $j$-invariants of elliptic curves over $\overline{\F}_p$ which are $n$-isogenous over $\overline{\F}_p$ to an elliptic curve with $j$-invariant $j$.
\end{theorem}
\begin{proof}
See \cite[Theorem 5.5]{lang}.
\end{proof}
Thus, the modular polynomial $\Phi_\ell(X,Y) \in \Z[X,Y]$ parameterizes pairs of $\ell$-isogenous elliptic curves so that each equivalence class of $\ell$-isogeny from $E(j)$, a curve with the $j$-invariant $j$, corresponds to a root of $\Phi_\ell(j,Y)$. As a result, we have the more concrete description of $G(p,\ell)$ below.
\begin{definition}
\label{defn-j-graph}
Let $\ell \neq p$ be primes. The supersingular $\ell$-isogeny graph in characteristic $p$ is the multi-digraph $G(p,\ell)$ whose vertices are $j$-invariants of supersingular elliptic curves over $\mathbb{F}_{p^2}$ and the number of directed edges from $j$ to $j'$ is equal to the multiplicity of $j'$ as a root of $\Phi_\ell(j,Y)$.
\end{definition}

\begin{remark}
  When there are no extra automorphisms, dual isogenies can be identified to create an undirected graph.
\end{remark}

\begin{remark}
The graph $G(p, \ell)$ has $\ell+1$ outgoing edges from each vertex.
\end{remark}

From the following result, we see that any two supersingular elliptic curves over $\mathbb{F}_{p^2}$ are connected by an isogeny of degree $\ell^m$, where we can take $m$ to be polynomial size in $\log p$.
\begin{theorem} \label{Koh96 Theorem 79}
  The graph $G(p, \ell)$ of $\ell$-isogenies of supersingular elliptic curve is connected. The diameter of the graph is $O(\log p)$, where the constant in the bound is independent of $\ell$.
\end{theorem}
\begin{proof}
See \cite[Corollary 78 and Theorem 79]{Koh96}.
\end{proof}

\subsection{Quadratic orders} \label{Quadratic orders} Here we briefly review and set up notation about quadratic orders.

Let $K$ be a number field of degree $[K:\Q] = n $ and $\calO_K$ its ring of integers. An order in $\calO_K$ is a subring of $\calO_K$ that is a free $\Z$-module of rank equal to $n$. The discriminant $d(\calO)$ of an order $\calO$ is the discriminant of any $\Z$-basis for $\calO$ and it is well-defined up to $\pm 1$.

Let $K$ now be a quadratic number field, that is, $[K:\Q] = 2$. A quadratic order is an order $\calO$ in $\calO_K$ where $K$ is a quadratic field. A non-square integer $d$ is the discriminant of a quadratic order if and only if $d \equiv 0, 1 \pmod 4$. For such $d$, consider the order in $K = \Q(\sqrt{d})$ given by
\begin{equation}
\label{quadratic-order-d}
  \calO_d = \Z[\omega_d] 
\end{equation}
where 
\begin{equation}
     \omega_d = \frac{d + \sqrt{d}}{2}.
\end{equation}
Every quadratic order $\calO$ is of the form $\calO = \calO_d$ for a unique discriminant $d$. The conductor $f(\calO)$ of $\calO = \calO_d$ is given by $f(\calO) = \sqrt{|d/d_K|}$, where $d_K$ is the fundamental discriminant of $\calO_K$. If $\ell \nmid f(\calO)$, then $\calO$ and its discriminant $d$ are called $\ell$-fundamental. 


For later applications to orientations by quadratic orders, it is useful to have a slightly different description of a quadratic order than \eqref{quadratic-order-d}.

\begin{lemma}
\label{order-gen}
The quadratic order $\calO_d$ with conductor $f = \sqrt{|d_K/d|}$ can be described as
\begin{equation} \label{quadratic order generator}
    \calO_d = \begin{cases}
        \Z[\frac{f \sqrt{d_K}}{2}]  & \text{ if } d_K \equiv 0 \pmod 4, \\
        \Z[\frac{1 + f \sqrt{d_K}}{2}] & \text{ if } d_K \equiv 1 \pmod 4 \text{ and } f  \text{ is odd,} \\
        \Z[\frac{f \sqrt{d_K}}{2}] & \text{ if } d_K \equiv 1 \pmod 4 \text{ and } f \text{ is even.}
        \end{cases}
\end{equation}
\end{lemma}
\begin{proof}
  This follows from the description of $\calO_d$ in \eqref{quadratic-order-d}.
\end{proof}

Given an integral $\calO$-ideal $\mathfrak{a}$, its absolute norm is defined to be 
$$N(\mathfrak{a})=\left|\calO / \mathfrak{a}\right|=\gcd(\{N(\alpha) : \alpha \in \a\}),$$
where $N(\alpha):=N_{K/\Q}(\alpha)$ is the usual field norm.

\subsection{Orientations}

We recall definitions of orientations from \cite{ACL+23, Onu21}.

Let $K$ be an imaginary quadratic field such that $K$ does not split at $p$. Then $K$ embeds into $B_0$ \cite[Proposition 14.6.7]{Voi21}. Throughout the paper, we assume that $p$ is not split in imaginary quadratic fields unless specified. 

\begin{definition}
A $K$-oriented elliptic curve is a pair $(E,\iota)$ of an elliptic curve $E/\overline{\F}_p$  and an embedding $\iota : K \hookrightarrow \End^0(E)$. The embedding $\iota$ is called a $K$-orientation on $E/\overline{\F}_p$.
\end{definition}

Given an isogeny $\varphi : E \rightarrow E'$ and a $K$-orientation $\iota$ on $E$, there is an induced $K$-orientation $\varphi_\ast(\iota)$ on $E'$ given by 
$$\varphi_*(\iota)(\alpha) = \frac{1}{\deg \varphi} \varphi \iota(\alpha)  \hat \varphi \ \forall \alpha \in K,$$ 
where $\hat{\varphi}$ is the dual isogeny of $\varphi$. An isogeny of $K$-oriented elliptic curves $\varphi : (E,\iota) \rightarrow (E',\iota')$ is then an isogeny $\varphi : E \rightarrow E'$ with $\iota'(a) = \varphi_*(\iota)$. A $K$-oriented isogeny is an isomorphism (resp. automorphism) if it is an isomorphism (resp. automorphism) of the underlying curves. Two $K$-oriented isogenies $\varphi: (E,\iota) \rightarrow (E', \iota'), \varphi': (E,\iota) \rightarrow (E'', \iota'')$ are equivalent if the underlying isogenies are equivalent.

\begin{remark}
    In the case of an ordinary elliptic curve $E/\overline{\F}_p$, $\End^0(E)$ is an imaginary quadratic field, and a $K$-orientation on $E$ is one of the two isomorphisms $K \cong \End^0(E)$. For a supersingular elliptic curve $E/\overline{\F}_p$, there are infinitely many embeddings $\iota : K \hookrightarrow \End^0(E)$ since $\End^0(E) \cong B_0$ and we can conjugate a subfield isomorphic to $K$ by elements in $B_0^\times$ to obtain different subfields isomorphic to $K$. By the Skolem-Noether theorem \cite[Theorem 7.7.1]{Voi21}, any two embeddings $\iota : K \hookrightarrow \End^0(E)  \cong B_0$ differ by conjugation by an element in $B_0^\times$.
\end{remark}

Let $\calO$ be an order in the quadratic field $K$. A $\calO$-orientation on an elliptic curve $E/\overline{\F}_p$ is a $K$-orientation $\iota$ such that $\iota(\calO) \subseteq \End(E)$. A $\calO$-orientation $\iota$ is primitive if in addition $\iota(\calO) = \iota(K) \cap \End(E)$, that is, $\calO$ is the largest order in $K$ for which $\iota(\calO) \subseteq \End(E)$. 

Let $\iota : K \hookrightarrow \End^0(E) $ be a $K$-orientation. Then there is a unique order $\calO$ in $K$ for which $\iota$ is a primitive $\calO$-orientation on $E/\overline{\F}_p$. For a prime $\ell$, we say a $\calO$-orientation $\iota$ is an $\ell$-primitive $\calO$-orientation if the index $[ \End(E) \cap \iota(K) : \iota(\calO)]$ is coprime to $\ell$.

\begin{definition}
Let $\ell \not= p$ be primes. The $K$-oriented supersingular $\ell$-isogeny graph $G_K(p,\ell)$ in characteristic $p$ is the graph whose vertices are $K$-isomorphism classes of $K$-oriented elliptic curves over $\Fbar_p$ and edges are the equivalence classes of $K$-oriented isogenies between them. 
\end{definition}
Inherited from the supersingular $\ell$-isogeny graphs $G(p,\ell)$, the graph $G_K(p, \ell)$ has $\ell+1$ outgoing edges from each vertex.

An advantage of equipping elliptic curves with orientations is that they admit notions of ascending, horizontal, and descending isogenies analogous to ordinary elliptic curves; given a $K$-oriented isogeny $\varphi:(E, \iota) \rightarrow (E', \iota')$ with a primitive $\calO$-orientation $\iota$ and a primitive $\calO'$-orientation $\iota'$, $\varphi$ is either ascending, horizontal, or descending depending on the following relationship between $\calO$ and $\calO'$: \\
(1) $\varphi$ is ascending if $\calO \subsetneq \calO'$ and $[\calO' : \calO]=\ell$,\\ (2) $\varphi$ is horizontal if $\calO = \calO'$,\\
(3) $\varphi$ is descending if $\calO \supsetneq \calO'$ and $[\calO : \calO']=\ell$.



Furthermore, the volcano structure of $G_K(p,\ell)$ is described by the following properties from \cite[Proposition 4.1]{Onu21}.

\begin{proposition} \label{Onu21 Proposition 4.1}
    Let $(E, \iota)/\cong$ be a vertex in $G_K(p,\ell)$, where $\iota$ is a primitive $\calO$-orientation on $E$. Let $d_K$ be the fundamental discriminant of $K$. If $\ell$ does not divide the conductor of $\calO$, then
    \begin{enumerate}
        \item[(1)] There are no ascending edges from $(E,\iota)$,
        \item[(2)] There are $\left( \frac{d_K}{\ell}\right)+1$ horizontal edges from $(E,\iota)$,
        \item[(3)] The remaining edges from $(E,\iota)$ are descending.
    \end{enumerate}
        If $\ell$ divides the conductor of $\calO$, then
        \begin{enumerate}
            \item[(1)] There is exactly one ascending edge from $(E,\iota)$,
            \item[(2)] The remaining edges from $(E,\iota)$ are descending. 
        \end{enumerate}
Note that $\left(\frac{\cdot}{\ell}\right)$ is the Legendre symbol.
\end{proposition}

\begin{definition}
   For an order $\calO \subseteq K$, we define $\SSpr$ to be the set of isomorphism classes of $K$-oriented supersingular elliptic curves over $\Fbar_p$ which are primitively $\calO$-oriented.
\end{definition}

\begin{remark}
    Recall that we assumed that $p$ is not split in $K$. If we drop that assumption, $\SSpr$ is non-empty if and only if $p$ is not split in $K$ and $p$ does not divide the conductor of $\calO$ \cite[Proposition 3.2]{Onu21}. 
\end{remark}

\begin{details}
    \begin{tbox}
        We do not make assumption that $p$ does not divide the conductor of $\calO$ like $\cite{ACL+23}$ does. 
    \end{tbox}
\end{details}

For a primitively $\calO$-oriented curve $(E, \iota)$, we define a subgroup $E[\a] \subseteq E$ given by an integral ideal $\a \subseteq \calO$ prime to $p$
\begin{equation}
    E[\a]:=\bigcap_{\alpha \in \a}\ker \iota(\alpha),
\end{equation}
which is called $\a$-torsion subgroup of $E$. Accordingly, there exists a (separable) $K$-oriented isogeny $\varphi_\a : (E,\iota) \rightarrow (E_\a, \iota_\a)$ where $E_\a 
\cong E/E[\a]$ and $\iota_\a ={(\varphi_\a)}_\ast(\iota)$. Let $\a\ast(E,\iota)$ denote the $K$-isomorphism class of $(E_\a. \iota_\a)$.   
This type of kernels yield either ascending or horizontal $K$-oriented isogenies:
\begin{proposition} \label{Onu21 Proposition 3.5}
    Let $(E,\iota)$ be a primitively $\calO$-oriented elliptic curve, $\a \subseteq \calO$ an integral ideal prime to $p$. Then a $K$-oriented isogeny with kernel $E[\a]$
    \begin{equation}
        \varphi: (E, \iota) \rightarrow (E_\a, \iota_\a)
    \end{equation}
is horizontal or ascending. Furthermore, if $\a$ is invertible then $\varphi$ is horizontal \cite[Proposition 3.5]{Onu21}. 
\end{proposition}

There is an action of the class group $\cl(\calO)$ of $\calO$ on $\SSpr$, given by
\begin{equation}
    \begin{split}
        \cl(\calO) \times \SSpr & \rightarrow \SSpr \\ 
        (\a, (E,\iota)) & \mapsto \a \ast (E,\iota),
    \end{split}
\end{equation}
which is free, but not necessarily transitive \cite[Proposition 3.3]{Onu21}.

Let $\theta \in \End(E)\setminus \Z$ be an endomorphism. There is an imaginary quadratic field $K \simeq \Q(\theta)$ and hence we have an associated $K$-orientation $\iota: K \hookrightarrow \End^0(E)$ such that $\iota(K)=\Q(\theta)$. 

Let $\calO = \Z[\alpha]$ be an order in $K$ such that $\alpha=\iota^{-1}(\theta)$. Then $\theta$ is called $\ell$-primitive if $\iota$ is $\ell$-primitive $\calO$-orientation. For an integer $N$, $\theta$ is called $N$-suitable if $\alpha$ is of the form $f\omega_K+kN$ for some $k$, where $f$ is the conductor of $\calO$ and $\omega_K$ is the generator for $\calO_K$ given by 
\begin{equation}
    \omega_K = \begin{cases}
        \frac{\sqrt{d_K}}{2} & \text{if} \ d_K \equiv 0\pmod 4, \\
        \frac{ 1 + \sqrt{d_K}}{2} & \text{if} \ d_K\equiv 1\pmod 4.
    \end{cases}
\end{equation}


\subsection{Quaternion Algebras}

Let $B_0$ be a quaternion algebra over $\Q$ ramified at $p$ and $\infty$. A lattice $I \subseteq B_0$ is a finitely generated $\Z$-module such that $I \Q = B_0$. For a lattice $I \subseteq B_0$, denote by
\begin{align}
    \calO_{R}(I) & = \left\{ \alpha \in B_0 : I \alpha \subseteq I \right\} \\
    \calO_{L}(I) & = \left\{ \alpha \in B_0 : \alpha I \subseteq I \right\}
\end{align}  
the right and left orders of $I$, respectively, and 
\begin{equation}
  \text{nrd}(I) = \text{gcd} \left( \left\{ \text{nrd}(\alpha) : \alpha \in I \right\} \right)
\end{equation}
the reduced norm of the lattice $I$. We also take the reduced norm of $I$ to be the fractional ideal of $\Z$ generated by the integer $\nrd(I)$. 

\begin{details}
    \begin{tbox}
\begin{remark}
        Let $R$ be a noetherian domain. Let $M$ be an $R$-lattice in $V$, a vector space over $F$. Then
    \begin{equation} \label{lattice local to global} M=\bigcap_{\mathfrak{p}}M_{(\mathfrak{p})},\end{equation}
    where $M_{(\mathfrak{p})}:=M R_{(\mathfrak{p})} \cong M \otimes_R R_{(\mathfrak{p})}$ for prime ideals $\mathfrak{p}$ in $R$ \cite[Lemma 9.4.6]{Voi21}. As a corollary, 
    \begin{center}
        $M \subseteq N \Leftrightarrow M_{(\mathfrak{p})} \subseteq N_{(\mathfrak{p})}$ for all prime (or maximal) ideals $\mathfrak{p}$ of $R$.
    \end{center}

    Suppose $R$ is a DVR with a unique maximal ideal $\mathfrak{p}$. Then
\begin{equation} \label{bijection between localization and completion}
    \begin{split}
        M & \mapsto M_{\mathfrak{p}}:=M \otimes_{R}R_{\mathfrak{p}} \\
        M_\mathfrak{p}  \cap V &\mapsfrom M_\mathfrak{p}
    \end{split}
\end{equation}
are mutually inverse bijections between the set of $R$-lattices in $V$ and the set of $R_\mathfrak{p}$-lattices in $V_\mathfrak{p}$ \cite[Lemma 9.5.3]{Voi21}.\\

    Now suppose $R$ is a Dedekind domain, e.g., $R=\Z$. Then $R_{(\mathfrak{p})}$ is DVR with a unique maximal ideal $\mathfrak{P}:=\mathfrak{p} R_{(\mathfrak{p})}$. Since $R_{(\mathfrak{p})}$ is already local, the localization of $M_{(\mathfrak{p})}$ at $\mathfrak{P}$ is itself. Also, let $\overline{\mathfrak{P}}:=\mathfrak{p}R_\mathfrak{p}$. Since $V$ is a vector space, $V_{(\mathfrak{p})} = V \otimes_{R} R_{(\mathfrak{p})}=V$.
    
    Since $R$ is Dedekind, every prime ideal $\mathfrak{p}$ is maximal. Hence, \href{https://mathoverflow.net/questions/64399/does-completion-commute-with-localization?_gl=1*naside*_ga*NDM3MDE2MDUyLjE3MjY1NTAyNDY.*_ga_S812YQPLT2*MTc0NDc1MDgxNi40NjEuMC4xNzQ0NzUwODE2LjAuMC4w}{completion commutes with localization}, i.e., $(R_{(\mathfrak{p})})_{\mathfrak{P}}=(R_{\mathfrak{p}})_{(\overline{\mathfrak{P}})}=R_{\mathfrak{p}}$.
    
    (\ref{bijection between localization and completion}) implies 
    \begin{equation}
        \begin{split}
            M_{(\mathfrak{p})} &=(M_{(\mathfrak{p})})_{\mathfrak{P}} \cap V\\
            &= \left(M_{(\mathfrak{p})}\otimes_{R_{(\mathfrak{p})}} (R_{(\mathfrak{p})})_{\mathfrak{P}}\right) \cap V\\
            &=\left(M_{(\mathfrak{p})}\otimes_{R_{(\mathfrak{p})}} R_{\mathfrak{p}}\right) \cap V\\
            & = M_{\mathfrak{p}} \cap V.
        \end{split}
    \end{equation}
    Hence,    
    \begin{equation} \label{lattice completion to global}
    M=\bigcap_{\mathfrak{p}}M_{(\mathfrak{p})}=\bigcap_\mathfrak{p}\left(M_\mathfrak{p}\cap V\right)=\left( \bigcap_{\mathfrak{p}} M_\mathfrak{p}\right) \cap V.
    \end{equation}

It follows that
    \begin{center}
        $M \subseteq N \Leftrightarrow M_{\mathfrak{p}} \subseteq N_{\mathfrak{p}}$ for all prime (or maximal) ideals $\mathfrak{p}$ of $R$.
    \end{center}
\end{remark}
\end{tbox}
\end{details}

\begin{extra}
    \begin{tbox}
Next, we write a formula for the reduced norm of a lattice. Recall that the reduced norm of a lattice $I$ in a semisimple algebra $B$ is the fractional ideal of $F$ generated by $\{\nrd(\alpha) : \alpha \in I\}$. 

Suppose $R=\Z$, $V=F=\Q$. Then $\nrd(I)$ is a lattice in $\Q$. By (\ref{lattice local to global}),
\begin{equation}
    \nrd(I)= \bigcap_{\ell \text{ prime}} \nrd(I)_{(\ell)}=\bigcap_{\ell \text{ prime}} \nrd(I_{(\ell)}).
\end{equation}
as
\begin{equation}
    \begin{split}
        \nrd(I)_{(\ell)} & = \{\ell^{\nu_{\ell}(\nrd(\alpha))}: \alpha \in I\} \Z_{(\ell)} \\
        & = \{\nrd(\alpha_{(\ell)}) : \alpha_{(\ell)} \in I\otimes \Z_{(\ell)}\}\Z_{(\ell)} \\
        & = \nrd(I_{(\ell)}).
    \end{split}
\end{equation}
(cf. \cite[16.3.3]{Voi21}). From (\ref{lattice completion to global}),
\begin{equation}
    \nrd(I)= \left(\bigcap_{\ell} \nrd(I)_\ell\right) \cap \Q=\left(\bigcap_{\ell} \nrd(I_\ell)\right) \cap \Q.
\end{equation}
as $\nrd(I)_\ell$ is the $\Z_\ell$-submodule of $V_\ell=F_\ell= \Q_\ell$ such that
\begin{equation}
    \begin{split}
        \nrd(I)_{\ell}&=\{\ell^{\nu(\nrd(\alpha))} : \alpha \in I\}\Z_{\ell} \\
        & = \{\nrd(\alpha_{\ell}) : \alpha_{\ell} \in I\otimes \Z_{\ell}\}\Z_{\el}\\
        & = \nrd(I_{\ell}).
    \end{split}
\end{equation}

If we take $\nrd(I)=\gcd(\{\nrd(\alpha): \alpha \in I\})$ and $\nrd(I)_\ell$ be a positive generator of the subgroup of $\Q_\ell$ generated by $\nrd(\alpha_\ell)$ for $\alpha_\ell \in I_\ell$, then
\begin{equation}
  \nrd(I_\ell)=\ell^{\nu_\ell(\nrd(I))}
\end{equation}
and
\begin{equation}
    \nrd(I)=\prod_{\ell \text{\ prime}} \nrd(I_{\ell}).
\end{equation}
    
        Accordingly, we take the reduced norm of $I_{(\ell)}$ (resp. $I_{\ell}$) to be a positive generator of $\nrd(I)_{(\ell)}$ (resp. $\nrd(I)_{\ell}$), i.e., $\ell^{\nu_{\ell}(\nrd(I))}$. Hence,
\begin{equation}
    \nrd(I) = \prod_{\ell} \nrd(I_\ell)
\end{equation}
(cf. \cite[16.3.3]{Voi21}).
    \end{tbox}
\end{extra}

The localization of $I$ is the $\Z_{(\ell)}$-module $I_{(\ell)} := I\Z_{(\ell)} \cong I \otimes_\Z \Z_{(\ell)}$ and the completion of $I$ is the $I_\ell := \Z_\ell$-module $I \otimes_\Z \Z_\ell$. We have
\begin{equation}
I=\bigcap_{\mathfrak{\ell}}I_{(\ell)}=\bigcap_\ell\left(I_\ell\cap B_0\right).
\end{equation}
where $\ell$ runs over all primes. Also, 
\begin{equation}
    \overline{I}=\{\overline{\alpha}: \alpha \in I\}
\end{equation}
is a lattice in $B_0$ where the map 
\begin{equation}
    \begin{split}
        \overline{\phantom{A}}: B_0 & \rightarrow B_0 \\
      \alpha=  a+bi+cj+dij & \mapsto \overline{\alpha}=a-bi-cj-dij
    \end{split}
\end{equation}
is a standard involution on $B_0$ for a given standard basis $1, i, j, ij$ for $B_0$. We have the following properties:
\begin{align}
            \overline{IJ}&=\bar{J}\bar{I}\\
        \overline{\mathfrak{O}}&=\mathfrak{O} \\
        \calO_L(\overline{I})&=\calO_R(I) \\
        \calO_R(\overline{I})&=\calO_L(I) 
\end{align}
for lattices $I, J \subseteq B_0$ and an order $\mathfrak{O}\subseteq B_0$.

\begin{details}
    \begin{tbox}
     This is \cite[16.6.6 and Lemma 16.6.7]{Voi21}.
    \end{tbox}
\end{details}

A lattice $I \subseteq B_0$ is integral if $I^2 \subseteq I$ and invertible if there is a lattice $I'$ such that
\begin{align}
    I I' = \calO_L(I) = \calO_R(I') \\
    I' I = \calO_L(I') = \calO_R(I).
\end{align}
If $I$ is integral, then $I \subseteq \calO_L(I) \cap \calO_R(I)$ by \cite[Lemma 16.2.8]{Voi21}. By \cite[16.6.14]{Voi21}, if $I$ is invertible, the (two-sided) inverse $I^{-1}$ of $I$ is given by
     \begin{equation}
         I^{-1}=\frac{1}{\nrd(I)}\overline{I}.
     \end{equation}

For an integral invertible lattice $I$, we have that 
\begin{equation}
  \nrd(I)^2 = [\calO_L(I):I] = [\calO_R(I):I]
\end{equation}
by \cite[Main Theorem 16.1.3]{Voi21}. 

A lattice $I \subseteq B_0$ is principal if
\begin{equation}
    I = \calO_L(I) \alpha = \alpha \calO_R(I)
\end{equation}
for some $\alpha \in B_0^\times$. If this is the case, then $I$ is invertible with inverse given by
\begin{equation}
    I^{-1} = \alpha^{-1} \calO_L(I) = \calO_R(I) \alpha^{-1}.
\end{equation}

The lattice $I$ is locally principal at $\ell$ if $I \otimes_\Z \Z_{(\ell)}$ is principal in $B_0 \otimes \Q_{(\ell)}$.

Let $I, J \subseteq B_0$ be lattices. The left and right colon ideals are defined as
\begin{align}
  (I:J)_L & = \left\{ \alpha \in B_0 : \alpha J \subseteq I \right\}, \\
  (I:J)_R & = \left\{ \alpha \in B_0 : J \alpha \subseteq I \right\}.
\end{align}
If $J$ is invertible, then by \cite[Exercise 16.8.11]{Voi21}
\begin{align}
    (I:J)_L & = I J^{-1}, \\
    (I:J)_R & = J^{-1} I. 
\end{align}

Let $\frakO$ be an order in $B_0$. A left fractional $\frakO$-ideal $I \subseteq B_0$ is a lattice $I \subseteq B_0$ such that $\frakO \subseteq \calO_L(I)$; a similar definition holds for right fractional $\frakO$-ideal $I$.

If $\mathfrak{O}$ is maximal, $I$ is locally principal at $\ell$ for every prime $\ell$ and hence invertible by \cite[Theorem 16.1.3]{Voi21}.

\begin{lemma} \label{induced isomorphism of quaternion algebras}
  Let $\varphi : E \rightarrow E'$ be an isogeny of supersingular elliptic curves over $\overline{\F}_p$, $B = \End^0(E)$, and $B' = \End^0(E')$. Then $\varphi$ induces an isomorphism of $\Q$-algebras
  \begin{align} \label{induced isomorphism of quaternion algebras equation}
      \kappa_\varphi : B' & \rightarrow B \\
     \notag  \phi & \mapsto \frac{1}{\deg \varphi} \hat \varphi \phi \varphi.
  \end{align}
\end{lemma}

We recall the following contravariant equivalence of categories.
\begin{theorem}
\label{ss-correspondence}
Let $E_0/\overline{\F}_p$ be a supersingular elliptic curve with $\End(E_0) \cong \mathfrak{O}_0$ where $\mathfrak{O}_0 \subseteq B_0 = \End(E_0) \otimes \Q$ is a maximal order. Then there is a contravariant equivalence of categories:
\begin{align}
     \operatorname{SS}(p) & = \left\{ \text{supersingular elliptic curves } E/{\overline \F}_p, \text{under isogenies} \right\} \\
 \notag & \longleftrightarrow \left\{ \text{left } \frakO_0 \text{-modules } I, \text{under non-zero } \frakO_0 \text{-module homomorphisms} \right\} 
\end{align}
In the equivalence above, $E$ is sent to $I = \Hom(E,E_0)$, $E'$ to $I' = \Hom(E',E_0)$, and the following hold:
\begin{enumerate}
\item[(a)] up to isomorphism, $I$ may be taken to be an invertible left $\mathfrak{O}_0$-ideal,
\item[(b)] $\Hom(E',E) \longleftrightarrow (I':I)_R = I^{-1} I' = \left\{ \alpha \in B_0 : I \alpha \subseteq I' \right\}$ for integral $I, I'$,
\item[(c)] $\Hom(E,E_0) \longleftrightarrow (I:\mathfrak{O}_0)_R = \mathfrak{O}_0^{-1} I = \left\{ \alpha \in B_0 : \mathfrak{O}_0 \alpha \subseteq I \right\}$ for integral $I$,
\item[(d)] $\End(E) \cong \calO_R(I)$ for the left fractional $\frakO_0$-ideal $I$ corresponding to $E$.
\end{enumerate}
\end{theorem}
\begin{proof}
This is \cite[Theorem 42.3.2]{Voi21}, \cite[Lemma 42.2.22]{Voi21}, and \cite[Lemma 42.2.9]{Voi21}.
\end{proof}

\begin{details}
    \begin{tbox} 
    Let $I, I'$ be non-zero integral left $\End(E_0)$-ideals. The bijection in \cite[Lemma 42.2.22]{Voi21} is
\begin{equation}
    \begin{split}
        \Hom(E_{I'}, E_I) &\rightarrow I^{-1}I' \subseteq \End(E_0) \\
        \psi &\mapsto \frac{1}{\deg(\varphi_I)}\hat{\varphi_{I}}\psi \varphi_{I'}.
    \end{split}
\end{equation}
\end{tbox}
\end{details}

\begin{definition}
  Let $\mathfrak{O}$ be a maximal order in $B_0$ and $I$ be an integral non-zero left $\mathfrak{O}$-ideal. We define $E[I] \subseteq E$ to be the scheme theoretic intersection
  \begin{equation}
      E[I] = \bigcap_{\alpha \in I} E[\alpha]
  \end{equation}
  where $E[\alpha] = \ker \alpha$ as a group scheme over $\overline{\F}_p$. There is an associated isogeny $\varphi_I : E \rightarrow E_I := E/E[I]$.
\end{definition}

\begin{remark}
If $I$ is integral, then the isogeny $\phi_I : E_0 \rightarrow E_0/E_0[I] \cong E_I$ has degree $\deg \phi_I = \text{nrd}(I)$ by \cite[Proposition 42.2.16]{Voi21}. Since we are only considering separable isogenies, accordingly we will only consider the integral left $\mathfrak{O}$-ideals of reduced norm coprime to $p$ unless specified.
\end{remark}

Consider the separable isogenies
\begin{equation}
    \begin{split}
        \varphi_I:E_0 & \rightarrow E_I=E_0/E_0[I]\\
        \varphi_{I'}:E_0 & \rightarrow E_{I'}=E_0/E_0[I'].
    \end{split}
\end{equation}
Then $E_0[I] \subseteq E[I']$ and hence $\varphi_{I'}$ factors as $\varphi_{I'}=\psi \varphi_{I}$ 
\[
\begin{tikzcd}
E_0 \arrow[d, "\varphi_I"'] \arrow[dr, swap, "\varphi_{I'}"']  &  \\
E_I=E_0/E_0[I]  \arrow[r, "\tilde{J}"']& E_{I'}=E_0/E_0[I']
\end{tikzcd}
\]
for some isogeny $\psi: E_I \rightarrow E_{I'}$ such that $\psi = \phi_{\tilde{J}}$ and where the kernel of $\psi$ given by  $E_I[\tilde{J}]$ for some left ideal $\tilde{J}$ of $\End(E_I)$. 

\begin{proposition}\label{EHL+18 Proposition 12} Let $I' \subseteq I$ be left $\End(E_0)$-ideals whose norms are coprime to $p$. Then there exists a separable isogeny $\psi : E_I \rightarrow E_{I'}$ such that $\varphi_{I'}=\psi\varphi_{I}$, and a left ideal $\tilde{J}$ of $\End(E_I)$ such that $E_{I}[\tilde{J}]=\ker(\psi)$ and
\begin{equation}
    \tilde{J} =\{\alpha \in \End(E_I) : \alpha(P)=0 \ \forall P \in \ker\psi\}
\end{equation}
where
\begin{equation}
    J :=\kappa_{\varphi_I}(\tilde{J})=I^{-1}I' \subseteq \End(E_0).
\end{equation}

\end{proposition}
\begin{proof}
  This is \cite[Proposition 12]{EHL+18}.
\end{proof}

If $I, J \subseteq B_0$ are two lattices, we define $[I]_L$ and $[I]_R$ to be the equivalence classes under the relations
\begin{align}
   I & \sim_L J \iff \exists \alpha \in B_0^\times \text{ such that } J = I \alpha, \\
   I & \sim_R J \iff \exists \alpha \in B_0^\times \text{ such that } J = \alpha I,
\end{align}
respectively.
\begin{remark}
An isomorphism class of supersingular elliptic curves $E/\overline{\F}_p$ corresponds to $[I]_L$ for a left fractional $\mathfrak{O}_0$-ideal $I$ in $B_0$. 
\end{remark}

The left and right class sets of an order $\frakO$ in $B_0$ are defined as
\begin{align}
  \text{Cls}_L(\frakO) & = \left\{ [I]_L : \text{invertible left fractional } \frakO \text{-ideals $I$ with } \calO_L(I) = \frakO \right\}, \\
  \text{Cls}_R(\frakO) & = \left\{ [I]_R : \text{invertible right fractional } \frakO \text{-ideals $I$ with } \calO_R(I) = \frakO \right\}.
\end{align}
\begin{remark} 
As $\frakO_0$ is a maximal order in $B_0$, it follows that the set of isomorphism classes of supersingular elliptic curves $E/\overline{\F}_p$ corresponds to the left $\mathfrak{O}_0$-class set $\text{Cls}_{L}(\frakO_0)$. Theorem~\ref{ss-correspondence} gives a functorial version of the classical Deuring correspondence \cite[Corollary 42.3.7]{Voi21}. For a statement that is more closely aligned with the original correspondence in terms of supersingular $j$-invariants in $\F_{p^2}$ and conjugacy classes of maximal orders in $B_0^\times$, see \eqref{classical-deuring}. 
\end{remark}

\begin{remark}
 In \cite[Proposition 42.4.8]{Voi21}, there is another functorial equivalence from the category of supersingular elliptic curves $E/\overline{\F}_p$ under reduced isomorphisms to oriented maximal orders in the quaternion algebra $\End(E) \otimes \Q$ under isomorphisms. In particular, the maximal orders are viewed in the varying quaternion algebra $\End(E) \otimes \Q$ rather than a fixed quaternion algebra $\End(E_0) \otimes \Q$.
\end{remark}

A fractional $\frakO,\frakO'$-ideal is a lattice $I \subseteq B_0$ which is a fractional left $\frakO$-ideal and a fractional right $\frakO'$-ideal.
\begin{definition}
Let $\mathfrak{O}, \mathfrak{O'}$ be orders in $B_0$. An invertible lattice $I$ with $\calO_L(I) = \mathfrak{O}$ and $\calO_R(I) = \mathfrak{O}'$ is called a connecting $\mathfrak{O},\mathfrak{O}'$-ideal.
\end{definition}

For the rest of the paper, we omit the term ``fractional'' and use the term integral when we deal with ideals contained in the relevant subring, e.g.\ integral $\mathfrak{O}$-ideal will mean an ideal of the ring $\mathfrak{O}$ in the usual sense; a $\mathfrak{O}$-ideal will mean a fractional $\mathfrak{O}$-ideal.

\section{Equivalence of categories in terms of quaternion algebras} \label{Equivalence of categories in terms of quaternion algebras}
In this section, we introduce an equivalence relation on supersingular elliptic curves over $\Fbar_p$, whose equivalence classes correspond bijectively to the set of maximal orders in $B_0$. To establish this, we review Deuring’s classical correspondence and refine some of the results.

Throughout this chapter, let $\ell$ be any prime and $\frakO_0 = \End(E_0)$ unless specified. Recall that $\frakO_0$ is a maximal order in $\End^0(E_0)\cong B_0$.

\subsection{Correspondence between curves and ideals}
Fix an initial supersingular elliptic curve $E_0/\Fbar_p$, and consider pairs $(E, \varphi)$ of a supersingular elliptic curve $E$ isogenous to $E_0$ and a separable isogeny $\varphi: E_0 \rightarrow E$. We define an equivalence relation on such pairs by $(E, \varphi) \sim (E', \varphi')$ if $\varphi \sim \varphi'$. Let 
\begin{equation}
  \SS(E_0):= \left\{(E,\varphi)  : \varphi: E_0 \rightarrow E \text{ is separable}\right\}/\sim. 
\end{equation}
\begin{details}
    \begin{tbox}
    zero map is not separable, as the corresponding pullback map does not give a separable extension. So separable means non-zero automatically
\end{tbox}
\end{details}
$\SS(E_0)$ can be interpreted as the set of equivalence classes of separable isogenies from $E_0$, or equivalently, as the set of finite subgroups of $E_0$ corresponding to their kernels. Hence, we have natural bijections
\begin{equation}
\begin{array}{c@{\;}c@{\;}c}
 \SS(E_0)
& \longleftrightarrow & 
\left\{\text{equivalence classes of separable isogenies from } E_0\right\} \\[1ex]
(E,\varphi)/\sim 
& \mapsto & 
\varphi:E_0 \rightarrow E /\sim \\[1ex]
& \longleftrightarrow & \left\{\text{finite subgroups $G\subseteq E_0$ of order coprime to $p$} \right\}\\[1ex]
& \mapsto & \ker\varphi.
\end{array}
\end{equation}

\begin{proposition} \label{Curves to Ideals}
     There is a bijection
    \begin{equation}
\begin{array}{c@{\;}c@{\;}c}
 \SS(E_0) 
& \longleftrightarrow & 
\left\{\text{non-zero left integral} \ \mathfrak{O}_0\text{-ideals of reduced norm coprime to } p \right\} \\[1ex]
(E,\varphi)/\sim 
& \mapsto & 
I_{\varphi}:=\Hom(E, E_0)\varphi.
\end{array}
\end{equation}    
\end{proposition}

\begin{proof}
    Under the equivalence of categories in Theorem \ref{ss-correspondence},
the functor $\Hom(-, E_0)$ sends a supersingular elliptic curve $E$ to an invertible left $\frakO_0$-module
\begin{equation}
    E \mapsto \Hom(E, E_0)
\end{equation}
and an isogeny $\psi: E \rightarrow E'$ to the pullback map 
\begin{equation}
\begin{split}
        \psi: E \rightarrow E' \quad  \mapsto \quad \psi^\ast: \Hom(E', E_0) & \rightarrow \Hom(E, E_0) \\
    \phi & \mapsto \phi \psi,
\end{split}
\end{equation}
 which is a homomorphism of left $\frakO_0$-modules. Hence, a non-zero isogeny $\varphi: E_0 \rightarrow E$ is sent to the homomorphism of left $\frakO_0$-modules
\begin{equation}
    \begin{split}
        \varphi^\ast : \Hom(E, E_0) 
 & \rightarrow \mathfrak{O}_0\\
 \phi &\mapsto \phi \varphi.
    \end{split}
\end{equation}
 This is an isomorphism between left $\frakO_0$-modules $\Hom(E, E_0) \cong I_{\varphi}$, where the latter is a non-zero left integral $\mathfrak{O}_0$-ideal.

\begin{details}
    \begin{tbox}
    Claim 1. $\varphi^\ast$ is an embedding of $\Hom(E, E_0)$ into $\End(E_0)$.

    \begin{proof}[proof of Claim 1]
        We only need to show that $\varphi^\ast$ is injective.      Suppose $\theta\varphi=0$ where $\theta \in \Hom(E,E_0)$. Then $\im(\varphi) \subseteq \ker(\theta)$. Since every non-zero isogeny is surjective, $\im(\varphi)=E$. Hence, $\theta=0$.        
    \end{proof}

    Claim 2. $I_\varphi=\Hom(E,E_0)\varphi$ is a left integral $\mathfrak{O}_0$-ideal.

    \begin{proof}[proof of Claim 2]
        $I_\varphi$ is a left $\mathfrak{O}_0$-ideal in the usual sense in the ring theory; if we take an element of $I_\varphi$ and post-compose it by an endomorphism of $E_0$, then it's still in $I_\varphi$. 

        Then by \cite[Remark 16.2.10]{Voi21}, since $I_\varphi$ is a full $\Z$-lattice, i.e., $I_\varphi\Q=B_0$, $I_\varphi$ is an integral left $\mathfrak{O}_0$-ideal in the sense of \cite[Definition 16.2.9]{Voi21}. 
    \end{proof}
\end{tbox}
\end{details}

Now suppose $(E, \varphi) \sim (E', \varphi')$ for separable isogenies $\varphi: E_0 \rightarrow E$ and $\varphi' : E_0 \rightarrow E'$. There exists an isomorphism $\lambda: E \rightarrow E'$ such that $\varphi' = \lambda \varphi$. It follows that
\begin{align}
     \Hom(E', E_0)\varphi' = \Hom(E',E_0)\lambda  \varphi = \Hom(E, E_0)\varphi
\end{align}
as
\begin{align}
\label{hom-a}
    \Hom(E',E_0)\lambda = \Hom(E,E_0).
\end{align}

\begin{details}
    \begin{tbox}
        The map
        \begin{equation}
            \begin{split}
                \lambda^\ast : \Hom(E', E_0) & \rightarrow \Hom(E, E_0) \\
                \varphi' & \mapsto \varphi' \lambda
            \end{split}
        \end{equation}
        is an isomorphism since $\lambda$ is invertible.
        \begin{enumerate}
            \item it is surjective since given $\varphi: E \rightarrow E_0$, $\varphi\lambda^{-1} \in \Hom(E', E_0)$.
            \item it is injective since $\varphi' \lambda=0$ implies $\varphi'=0$.
        \end{enumerate}
    \end{tbox}
\end{details}

Conversely, given a non-zero integral left $\mathfrak{O}_0$-ideal $I$ of reduced norm coprime to $p$, there exists a separable isogeny $\varphi_I: E_0 \rightarrow E_I$ with $\ker(\varphi_I)=E_0[I]$ such that the pullback map $\varphi_I^\ast$ is an isomorphism of left $\mathfrak{O}_0$-modules $\Hom(E_I, E_0)\cong I$ \cite[Lemma 42.2.7]{Voi21}. Then the inverse map in the bijection is given by
\begin{equation} \label{Curves to Ideals inverse}
    I \mapsto (E_I, \varphi_I).
\end{equation}
\end{proof}

\begin{details}
   \begin{tbox}
        \begin{remark} We may extend the bijection in Proposition \ref{Curves to Ideals} to inseparable isogenies and integral ideals of reduced norm divisible by $p$.

Recall that for an elliptic curve $E/\Fbar_p$ and $q=p^r$, $E^{(q)}$ is defined to be the curve whose homogeneous ideal is $I(E^{(q)})$, i.e., ideal generated by $\{f^{(q)}: f \in I(E)\}$, where $f^{(q)}$ is the polynomial obatained from $f$ by raising each coefficients to the $q$th power. Any isogenies $\varphi:E \rightarrow E'$ of inseparable degree $q$ factor through the $q$th power Frobenius morphism $\pi_q: E \rightarrow E^{(q)}$ of $E$ as
\begin{equation}
    \varphi = \varphi_{\sep}\pi_{q},
\end{equation}
where $\varphi_{\sep}$ is the separable part of $\varphi$ \cite[Corollary II.2.12]{Sil09}. We say two inseparable isogenies  $\varphi: E \rightarrow E', \varphi': E \rightarrow E''$ are equivalent if their separable parts are equivalent and their inseparable degrees are the same.

Let $\varphi: E_0 \rightarrow E$ and $ \varphi': E_0 \rightarrow E'$ be equivalent inseparable isogenies and $q>1$ be their inseparable degree. It follows that $\varphi = \varphi_{\sep} \pi_{q}$ and $\varphi'= \varphi'_{\sep} \pi_{q}$. Then we have
\begin{equation}
    \Hom(E', E_0) \varphi' =\Hom(E', E_0)\varphi'_{\sep} \pi_{q}=\Hom(E,E_0)\varphi_{\sep}\pi_{q}=\Hom(E, E_0)\varphi.
\end{equation}

Conversely, given a non-zero integral left $\mathfrak{O}_0$-ideal $I$ of reduced norm $qn$ with $q=p^r$, $r\geq 1$, and $p \nmid n$, we can write uniquely as $I=\mathfrak{P}^rI'$ where $\mathfrak{P}$ is the unique two-sided $\mathfrak{O}_0$-ideal of reduced norm $p$ and $\nrd(I')=n$ (cf. \cite[42.2.4.]{Voi21}). Then there exists an inseparable isogeny $\varphi_I: E_0 \rightarrow E_I$ with $\ker(\varphi_I)=E_0[I]$ such that the pullback map $\varphi_I^\ast$ is an isomorphism of left $\mathfrak{O}_0$-modules $\Hom(E_I, E_0)\cong I$ \cite[Lemma 42.2.7]{Voi21}. Note that $\varphi_I$ factors as $\varphi_I=\varphi_{I'}\pi_q$ where $\varphi_{I'}$ is the separable part of $\varphi_I$. 
\end{remark}
   \end{tbox}
\end{details}

\subsection{Cyclic isogenies and primitive ideals} We show that the correspondence between separable isogenies and integral ideals in Proposition \ref{Curves to Ideals} restricts to the one between cyclic isogenies and primitive ideals; in the later part of the section, we relate them to maximal orders in $B_0$ through an explicit bijection. 

\begin{definition} Let $\varphi: E \rightarrow E'$ be a separable isogeny of elliptic curves. Then $\varphi$ is called cyclic if its kernel is a cyclic subgroup of $E$.
\end{definition}

By a trivial cyclic isogeny $\varphi: E\rightarrow E'$, we mean $\varphi$ is an isomorphism, i.e., $\deg \varphi=1$, whose kernel $\{0\}$ is trivially cyclic.

\begin{theorem} \label{Gal12 Theorem 25.1.2}
    Any separable isogeny $\varphi: E \rightarrow E'$ of elliptic curves can be written as a chain of isogenies $\varphi= \varphi_1 \cdots\varphi_k[n]$ for some isogenies $\varphi_1, \cdots, \varphi_k$ of prime degree and the largest integer $n$ such that $E[n] \subseteq  \ker\varphi$ where
    \begin{equation}
        \deg \varphi = n^2 \prod_{i=1}^k \deg(\varphi_i).
    \end{equation}
\end{theorem}
\begin{proof}
    This is \cite[Theorem 25.1.2]{Gal12}.
\end{proof}

Let $\varphi: E \rightarrow E'$ be a separable isogeny. Write $\varphi=\varphi_1 \cdots\varphi_k[n]$ for some isogenies $\varphi_1, \cdots, \varphi_k$ of prime degree and the largest integer $n\geq 1$ such that $E[n] \subseteq \ker\varphi$. Then $\varphi$ is cyclic if $n=1$ and $\deg\varphi_i=\ell_i$, $1\leq i\leq k$ are pairwise distinct primes, in which case
    \begin{equation}
        \ker\varphi \cong (\Z/\ell_1\Z) \times \cdots \times (\Z/\ell_k\Z) \cong \Z/(\ell_1\cdots \ell_k)\Z
    \end{equation}
  by the Chinese remainder theorem.

If $\varphi$ is cyclic, then $n=1$ as otherwise $n>1$ and 
 \begin{equation}
     E[n]\cong (\Z/n\Z)^2\subseteq \ker \varphi,
 \end{equation}
which is a contradiction.

\begin{details}
    \begin{tbox}
         $\varphi$ is called primitive if it does not factr through $[n]: E \rightarrow E$ for any integer $n >1$. 

   See \cite{BCE+19} for primitive isogenies and cycles with no backtracking in $G(p,\ell)$
         
It follows that if $\varphi$ is cyclic, then $\varphi$ is primitive. If $\ell_i=\ell$ for all $ 1\leq i \leq k$, then two notions coincide \cite[Section 4]{BCE+19}.
    \end{tbox}
\end{details}

\begin{definition}
  Let $\mathfrak{O}$ be a maximal order in $B_0$ or $B_\ell$. We say an integral left (resp.\ right) $\mathfrak{O}$-ideal $I$ is primitive if for any integer $n \neq \pm1$, we cannot write $I = nI'$ for an integral left (resp.\ right) $\mathfrak{O}$-ideal $I'$.
\end{definition}

By a trivial primitive left (resp.\ right) $\mathfrak{O}$-ideal, we mean integral left (resp.\ right) $\mathfrak{O}$-ideal $I$ of reduced norm $1$, i.e., $I=\mathfrak{O}$.

\begin{details}
    \begin{tbox}
        If $\nrd(I)=1$, then $I=\mathfrak{O}\nrd(I)+\mathfrak{O}\alpha=\mathfrak{O}+\mathfrak{O\alpha}=\mathfrak{O}$ for some $\alpha \in B^\times_0$ by \cite[Exercise 16.6]{Voi21} where the last equality follows since $\mathfrak{O}\alpha \subseteq \mathfrak{O}$.
    \end{tbox}
\end{details}

\begin{lemma} \label{equivalent definition of primitive ideal}
   Let $\mathfrak{O}\subseteq B_0$ be a maximal order and $I$ be an integral left $\mathfrak{O}$-ideal. The following are equivalent:
   \begin{enumerate}
       \item[(i)] $I$ is primitive.
       \item[(ii)] $n^{-1}I$ is not an integral left $\mathfrak{O}$-ideal for any integer $n\neq \pm1$.
       \item[(iii)] $I$ is not contained in any non-trivial two-sided ideal of the form $n\mathfrak{O}$ for any integer $n \neq \pm 1$.
       \item[(iv)] $I$ is primitive as an integral right $\mathcal{O}_R(I)$-ideal.
   \end{enumerate}
   Also, if $\nrd(I)$ is square-free, then $I$ is primitive.
\end{lemma}

\begin{details}
        \begin{tbox}
    (i) is from \cite[26.4.3 and 41.3 p.771]{Voi21}. \\
    (ii) is from \href{https://eprint.iacr.org/2025/042.pdf}{2.5 Integral ideals}: a left (integral) $\mathfrak{O}$-ideal $I$ is cyclic (a.k.a. primitive) if for any $q \in \Z$ the set $\{\frac{x}{q}: x \in I\}$ is not a left (integral) $\mathfrak{O}$-ideal.  \\
    (iii) is from \href{https://arxiv.org/pdf/0808.3833}{Algorithmic Enumeration of Ideal Classes for Quaternion Orders}.
\end{tbox}
\end{details}

\begin{proof}
   The implication $(i) \Leftrightarrow (ii)$ is trivial.

  The implication $(ii) \Leftrightarrow (iii)$ follows from $I \subseteq n \mathfrak{O}$ iff $n^{-1} I \subseteq \mathfrak{O}$: if $n^{-1} I$ is an integral left $\mathfrak{O}$-ideal, then $n^{-1} I \subseteq \mathfrak{O}$; conversely if $I \subseteq n \mathfrak{O}$, then $n^{-1} I \subseteq \mathfrak{O}$ is an integral left $\mathfrak{O}$-ideal.

The implication $(i) \Leftrightarrow (iv)$: let $\mathfrak{O}' = \mathcal{O}_R(I)$.  
Suppose $I$ is a primitive left $\mathfrak{O}$-ideal but not a primitive right $\mathfrak{O}'$-ideal.  
Then, by definition, there exists an integer $n \neq \pm 1$ and an integral right $\mathfrak{O}'$-ideal $I'$ such that $I = n I'$. Since $\mathcal{O}_L(I') = \mathcal{O}_L(I) = \mathfrak{O}$, $I'$ is also an integral left $\mathfrak{O}$-ideal by \cite[Lemma 16.2.8]{Voi21}, contradicting the assumption that $I$ is primitive as a left $\mathfrak{O}$-ideal.  
The converse follows by the same argument, interchanging the roles of $\mathfrak{O}$ and $\mathfrak{O}'$.

  For the last statement, suppose $I$ is not primitive. There is an integer $n>1$ and an integral left $\mathfrak{O}$-ideal $I'$ such that $I=nI'$. Then
      \begin{equation}
        \nrd(I)=\nrd(nI')=n^2 \nrd(I')
    \end{equation}
    is not square-free.
\end{proof}

\begin{remark}
    Let $I$ be a non-primitive integral left $\mathfrak{O}$-ideal and $n>1$ is the largest integer such that $I\subseteq n\mathfrak{O}$. Then $n^{-1}I \subseteq \mathfrak{O}$ is integral and primitive. To see this, suppose $n^{-1}I$ is not primitive. Then there exists an integer $m>1$ such that $n^{-1}I \subseteq m \mathfrak{O}$. This implies $I \subseteq mn \mathfrak{O}$, contradicting the maximality of $n$.
\end{remark}

We may likewise define primitivity for elements of lattices. The notions of the primitivity of elements and ideals are closely related: the primitivity of an integral ideal can be characterized in terms of the existence of primitive elements among its generators (cf. \cite[p.25]{Sch23}). 

\begin{definition}
Let $I$ be a lattice in $B_0$. We say a non-zero element $\alpha \in I$ is primitive if for any integer $n \neq \pm 1$, we cannot write $\alpha=n\beta$ for any $\beta \in I$.
\end{definition}

\begin{remark}
If $n$ is the largest integer such that $n^{-1}\alpha \in I$ for an element $\alpha \in I$, then $n^{-1}\alpha$ is a primitive element of $I$.
\end{remark}

\begin{definition}
Let $\mathfrak{O}$ be a maximal order in $B_0$ and $I$ be a left (resp. right) $\mathfrak{O}$-ideal. We say $\alpha \in B_0^\times$ is a generator of $I$ if $I=\mathfrak{O}\nrd(I)+\mathfrak{O}\alpha$ (resp. $I=\nrd(I)\mathfrak{O}+\alpha\mathfrak{O}$).
\end{definition}

\begin{lemma}
Let $\mathfrak{O}$ be a maximal order in $B_0$. A left (resp. right) $\mathfrak{O}$-ideal $I$ is primitive if and only if there is a generator of $I$ which is a primitive element of $\mathfrak{O}$.
\end{lemma}
\begin{proof}
    We prove this for left $\mathfrak{O}$-ideal $I$.

    Suppose $I$ is primitive. By \cite[Exercise 16.6]{Voi21}, we can write $I=\mathfrak{O}\nrd(I)+\mathfrak{O}\alpha$ for some $\alpha \in B_0^\times$. Since $I \subseteq \mathfrak{O}$, we have $\alpha=0\cdot \nrd(I)+1\cdot \alpha \in I$. Suppose $\alpha$ is not a primitive element of $\mathfrak{O}$. Then there exists $\beta \in \mathfrak{O}$ such that $\alpha=n\beta$ for some integer $n>1$. It follows that we have an integral left $\mathfrak{O}$-ideal $I'=\mathfrak{O}n\nrd(I)+\mathfrak{O}\beta$ such that $I=nI'$. This is a contradiction to our assumption that $I$ is primitive. Therefore, $\alpha$ is a generator of $I$ which is a primitive element of $\mathfrak{O}$.

\begin{details}
    \begin{tbox}
        $I'$ is integral left $\mathfrak{O}$-ideal since the sum of left ideals of a ring is again a left ideal and sum of lattices of an $F$-algebra is again lattice. 
    \end{tbox}
\end{details}

    Conversely, suppose $I=\mathfrak{O}\nrd(I)+\mathfrak{O}\alpha$ for some primitive element $\alpha$ of $\mathfrak{O}$. Suppose $I$ is not primitive. Then there is an integral left $\mathfrak{O}$-ideal $I'$ such that $I=nI'$ for some integer $n>1$. In particular, $\nrd(I)=n^2\nrd(I')$. Since $I'=\mathfrak{O}n^{-1}\nrd(I)+\mathfrak{O}n^{-1}\alpha \subseteq \mathfrak{O}$, we have $n^{-1}\alpha \in \mathfrak{O}$. This is a contradiction to our assumption that $\alpha$ is a primitive element of $\mathfrak{O}$. Therefore, $I$ is primitive.
\end{proof}

Being a primitive ideal is a local property.

\begin{proposition} \label{Primitivity local property} Let $\mathfrak{O}$ be a maximal order in $B_0$ and $I$ be an integral left $\mathfrak{O}$-ideal. $I$ is a primitive left $\mathfrak{O}$-ideal if and only if $I_\ell$ is a primitive left $\mathfrak{O}_\ell$-ideal for any prime $\ell$.
\end{proposition}

\begin{proof}
    Suppose $I$ is not primitive. Then there exists an integer $n \neq \pm1$ such that $I=nI'$ for an integral left $\mathfrak{O}$-ideal $I'$. Let $\ell$ be a prime factor of $n$. Then
    \begin{equation}
        I_\ell = \ell^{\nu_\ell(n)}I_\ell'
    \end{equation}
    where $\nu_\ell(n) \geq 1$ and $I_\ell'$ is an integral left $\mathfrak{O}_\ell$-ideal. Hence, $I_\ell$ is not primitive.

    Conversely, suppose $I_\ell$ is not a primitive $\mathfrak{O}_\ell$-ideal for a prime $\ell$. Then there exist an integer $e \geq 1$ and an integral left $\mathfrak{O}_\ell$-ideal $J$ such that $I_\ell=\ell^e J$. By \cite[Lemma 27.6.8]{Voi21}, there exists an integral left $\mathfrak{O}$-ideal $I'$ such that $I'_\ell=J$ and $I'_{\ell'}=I_{\ell'}$ for any prime $\ell' \neq \ell$. It follows that $I=\ell^eI'$. Hence, $I$ is not primitive.
\end{proof}

Next, we show that there is a correspondence between cyclic isogenies and primitive ideals.

Let $E/\Fbar_p$ be a supersingular elliptic curve, $\ell \neq p$ be a prime, and $\mathfrak{O}=\End(E)$. We have an isomorphism of rings (cf. \cite[Theorem 42.1.9]{Voi21})
\begin{equation}
    \begin{split}
        \mathfrak{O}/\ell\mathfrak{O} & \rightarrow \End (E[\ell]) \\
        \overline{\alpha} & \mapsto \alpha_\ell
    \end{split}
\end{equation}
where $\overline{\alpha}$ denotes the coset of an endomorphism $\alpha \in \mathfrak{O}$ modulo $\ell\mathfrak{O}$ and $\alpha_\ell$ denotes the restriction of $\alpha$ to $E[\ell] \cong (\Z/\ell \Z)^2$. After fixing a basis for $E[\ell]$, $\alpha_\ell$ has a matrix representation in $\GL_2(\F_\ell)$ by the action on the basis;  let $P_1, P_2 \in E[\ell]$ be a basis for $E[\ell]$, and
    \begin{equation}
    \begin{split}
        \alpha(P_1) &= a P_1 + b P_2, \\
        \alpha(P_2) & = cP_1 + dP_2
    \end{split}
\end{equation}
for some $a, b, c, d \in \mathbb{F}_\ell$. We have an isomorphism of rings
\begin{equation} \label{Isomorphism of quotient and matrix ring}
    \begin{split}
        \mathfrak{O}/\ell\mathfrak{O} & \rightarrow M_2(\F_\ell) \\
        \overline{\alpha} & \mapsto  
        M_\alpha:=\begin{pmatrix}
            a & b \\ c & d
        \end{pmatrix}.
    \end{split}
\end{equation}

The characteristic polynomial $c_\alpha(x)$ of $\alpha_\ell$ (or $M_\alpha$) coincides with the polynomial
\begin{equation}
    x^2 - \trd(\alpha) x + \nrd(\alpha)
\end{equation}
modulo $\ell$.

\begin{details}
\begin{tbox}   
$\mathfrak{O}/\ell\mathfrak{O} \rightarrow M_2(\F_\ell)$ preserves addition and multiplication: Suppose $\beta$ is another endomorphism with the action on $E[\ell]$ given by
    \begin{equation*}
    \begin{split}
        \beta(P_1) &= a' P_1 + b' P_2 \\
        \beta(P_2) & = c'P_1 + d'P_2,
    \end{split}
    \end{equation*}
\begin{itemize}
    \item Addition:\\
    The action of $\alpha+\beta$ is
    \begin{equation*}
        \begin{split}
            (\alpha+\beta)(P_1) = (a+a')P_1+(b+b')P_2\\
            (\alpha+\beta)(P_2) = (c+c')P_1+(c+c')P_2.
        \end{split}
    \end{equation*}
    \item Multiplication:\\
    The action of $\beta\circ \alpha$ is 
    \begin{equation*}
        \begin{split}
        (\beta\circ \alpha)(P_1)&=\beta(aP_1+bP_2)=(aa'+bc')P_1+(ab'+bd')P_2\\
        (\beta\circ \alpha)(P_2)&=\beta(cP_1+dP_2)=(a'c+c'd)P_1+(b'c+dd')P_2.
        \end{split}
    \end{equation*}
    and
    \begin{equation*}
        \begin{pmatrix}
            a & b\\
            c & d
        \end{pmatrix}
        \begin{pmatrix}
            a' & b'\\
            c' & d'
        \end{pmatrix}
        =
        \begin{pmatrix}
            aa'+bc' & ab'+bd' \\
            a'c+c'd & b'c+dd'
        \end{pmatrix}.
    \end{equation*}
    \item $1_\ell =\begin{pmatrix}
        1 & 0\\
        0 & 1
    \end{pmatrix}$.
\end{itemize}
\end{tbox}
\end{details}

\begin{proposition} \label{cyclic isogeny and primitive ideal}
    Let $\varphi: E_0 \rightarrow E$ be a separable isogeny. Then $\varphi$ is cyclic if and only if
    $I_{\varphi}$ is primitive.
\end{proposition}

\begin{details}
    \begin{tbox}
      If $\varphi$ is cyclic, by \cite[Proposition 2.1.2]{Ler22}, $I_\varphi$ is primitive.

    See also \cite[Algorithm 19, Algorithm 20]{Ler22} for computing $E_0[I]$ for primitive ideal $I$ and $I(\ker\varphi)$ for cyclic isogeny $\varphi$.

    Given in \href{https://eprint.iacr.org/2025/042.pdf}{2.5 Deuring Correspondence} without proof.
\end{tbox}
\end{details}

\begin{proof}
    Firstly note that $\varphi$ is a trivial cyclic isogeny if and only if $I_{\varphi}$ is a trivial primitive left $\mathfrak{O}$-ideal. Assume now that $\deg\varphi=\nrd(I_{\varphi})\neq 1$.
    
    Suppose $\varphi$ is cyclic. Assume for contradiction that $I_{\varphi}$ is not primitive. Then there is an integer $n>1$ such that $I_{\varphi} \subseteq n\mathfrak{O}_0$ by Lemma \ref{equivalent definition of primitive ideal}. Hence, $E[n] \subseteq E[I_{\varphi}]=\ker\varphi$ and $\varphi$ factors through $[n]$. Therefore, $\varphi$ is not cyclic.

    Conversely, suppose $I_\varphi$ is primitive. Assume for contradiction that $E_0[I_{\varphi}]$ is not cyclic. Then $E_0[I_\varphi]$ contains a subgroup isomorphic to $E_0[\ell] \cong(\Z/\ell\Z)^2$ for some prime $\ell \neq p$. The action of $\mathfrak{O}_0$ on $E_0[\ell]$ defines a map
    \begin{equation}
       \rho_\ell: \mathfrak{O}_0 \rightarrow    \End(E_0[\ell]) \cong \mathfrak{O}_0/\ell\mathfrak{O}_0, 
    \end{equation}
where $\ker\rho_\ell=\ell\mathfrak{O}_0$. Since
    \begin{equation}
        E_0[I_\varphi]= \bigcap_{\alpha \in I_\varphi} \ker \alpha
    \end{equation}
   and we assumed $E_0[\ell] \subseteq E_0[I_{\varphi}]$, any $\alpha \in I_\varphi$ acts trivially on $E_0[\ell]$, i.e., $I_\varphi \subseteq \ker \rho_\ell = \ell \mathfrak{O}_0$.  By Lemma~\ref{equivalent definition of primitive ideal}, $I_\varphi$ is not primitive, a contradiction.
\end{proof}

\begin{details}  
\begin{tbox}
        "Since $I$ is not divisible by any integer $n$, $E[I]$ is cyclic" \href{https://hal.science/hal-04056062v1/file/SQISign_dim4.pdf}{Lemma B.1.1}
\end{tbox}
\end{details}

It follows that by writing a cyclic isogeny $\varphi$ as a chain of isogenies of prime degree, we get a decomposition of $I_{\varphi}$ into a product of primitive ideals of prime reduced norm.

\begin{corollary}
\label{decompose-primitive}
   Let $\varphi$ be a non-trivial cyclic isogeny from $E_0$. Then $I_{\varphi}$ can be decomposed into a product $I_{\varphi}=I_k \cdots I_1$ of primitive ideals $I_1, \cdots, I_k$ of prime reduced norm. 
\end{corollary}

\begin{proof}
    By Theorem \ref{Gal12 Theorem 25.1.2}, we can write $\varphi$ as a chain $\varphi=\varphi_1 \cdots \varphi_k$ of isogenies of prime degree. This corresponds to the product $I_{\varphi}=I_k\cdots I_1$ where $I_1, \cdots, I_k$ are primitive ideals of prime reduced norm such that $E[I_i]=\ker\varphi_{i}$ for $1\leq i \leq k$ by \cite[Propositions 2.1.4, 2.1.5]{Ler22Thesis}. For a direct proof of decomposition into primitive ideals, see \cite[Lemma 24]{Cle25}.
\end{proof}

\subsection{Correspondences between curves, primitive ideals, and maximal orders}
For any separable isogeny $\varphi : E_0 \rightarrow E$, let $\mathfrak{O}_{\varphi}:=\calO_R(I_{\varphi})$. Let $\SScy(E_0) \subseteq \SS(E_0)$ be the set of equivalence classes of pairs $(E,\varphi)$ in $\SS(E_0)$ such that $\varphi$ is cyclic
\begin{equation}
  \SScy(E_0):=  \left\{(E,\varphi) : \varphi: E_0 \rightarrow E \text{ is cyclic}\right\}/\sim. 
\end{equation}
For an equivalence class $(E,\varphi)/\sim$ in $\SScy(E_0)$, we identify $\End(E)$ with the maximal order $\mathfrak{O}_{\varphi}$ in $\End^0(E_0)$. This identification is well-defined up to the equivalence relation on $(E, \varphi)$. Conversely, we will show that every maximal order in $\End^0(E_0)$ arises as $\mathfrak{O}_\varphi$ for a unique equivalence class of such pairs. 

It is also important to note that primitive ideals play a role analogous to that of isogenies from the fixed curve $E_0$: just as each curve $E$ is marked by an isogeny $\varphi: E_0 \rightarrow E$, the ideal $I_\varphi$ connects maximal orders $\mathfrak{O}_0$ and $\mathfrak{O}_{\varphi}$. We will show that $I_\varphi$ is the unique integral connecting $\mathfrak{O}_{0}, \mathfrak{O}_{\varphi}$-ideal of minimal reduced norm if $\varphi$ is cyclic. 

To this end, we begin by reviewing integral connecting ideals of minimal reduced norm between maximal orders in $B_0$.  

\begin{definition}
Let $\mathfrak{O}, \mathfrak{O}'$ be maximal orders of $B_0$ or $B_\ell$. We say an integral connecting $\mathfrak{O}, \mathfrak{O}'$-ideal $I$ has minimal reduced norm if its reduced norm is minimal among integral connecting $\mathfrak{O}, \mathfrak{O}'$-ideals.
\end{definition}
 
\begin{lemma} \label{minimal reduced norm local}
   Let $\ell \neq p$ be a prime. Let $\mathfrak{O},\mathfrak{O}'$ be local maximal orders in $B_\ell$. There is $\alpha \in \mathfrak{O}\setminus \ell \mathfrak{O}$ unique up to left multiplication by $\mathfrak{O}^\times$ such that 
   \begin{equation}
    \mathfrak{O}'=\alpha^{-1}\mathfrak{O}\alpha.
   \end{equation}
   Furthermore, $I=\mathfrak{O}\alpha$ is the unique integral connecting $\mathfrak{O},\mathfrak{O}'$-ideal of minimal reduced norm and
   \begin{equation}
      \nrd(\alpha) = [\mathfrak{O} : \mathfrak{O}\cap \mathfrak{O}'].
   \end{equation}
\end{lemma}
\begin{proof}
    This is \cite[Exercise 17.4(c), (d)]{Voi21}. Also, see the proof of \cite[Lemma 4.2]{LB20}.
\end{proof}

\begin{lemma} \label{minimal reduced norm global}
    Let $\mathfrak{O},\mathfrak{O}'$ be maximal orders in $B_0$. There exists a unique integral connecting $\mathfrak{O},\mathfrak{O}'$-ideal $I$ of minimal reduced norm. Moreover,
    \begin{enumerate}
        \item[(a)] $\nrd(I)=[\mathfrak{O}: \mathfrak{O}\cap \mathfrak{O}']=[\mathfrak{O}': \mathfrak{O}\cap \mathfrak{O}']$.
        \item[(b)] $I=\{\alpha \in B_0: \alpha\mathfrak{O}' \overline{\alpha} \subseteq [\mathfrak{O}:\mathfrak{O} \cap \mathfrak{O}']\mathfrak{O}\}.$
    \end{enumerate}
\end{lemma}
\begin{proof}
    (a) This is \cite[Exercise 17.4]{Voi21} or \cite[Lemma 4.2]{LB20}.
    
    (b) This is \cite[Lemma 8]{KLPT14}.
\end{proof}

An Eichler order is the intersection of two maximal orders and its level coincides with the common index of this intersection inside either maximal order.

\begin{proposition} \label{Index of Eichler order} Let $\varphi:E_0 \rightarrow E$ be a cyclic isogeny. Consider the ring of endomorphisms of $E_0$ preserving $\ker \varphi$
            \begin{equation}
                \End(E_0, \ker \varphi):= \{\alpha \in \End(E_0) : \alpha(\ker\varphi) \subseteq \ker \varphi\}.
            \end{equation}
We have $\End(E_0, \ker \varphi)=\mathfrak{O}_0 \cap \mathfrak{O}_\varphi$ and $[\mathfrak{O}_0 : \mathfrak{O}_0 \cap \mathfrak{O}_\varphi]=\deg \varphi$ where $\mathfrak{O}_0 = \End(E_0)$.
\end{proposition}

\begin{details}
    \begin{tbox}
        $(E,G)$ for $G \subseteq E[N]$ was defined for integer $N$ coprime to $p$ in \cite{Arp24}. Hence, we need the condition that $\varphi$ is separable.
    
        Observe that both $\End(E_0, \ker \varphi)$ and $\mathfrak{O}_0 \cap \mathfrak{O}_\varphi$ are invariant under replacing $\varphi$ by an equivalent isogeny by construction.
    \end{tbox}
\end{details}

\begin{proof}
This is \cite[Proposition 3.4 and Theorem 3.7]{Arp24}. Note that $\mathfrak{O}_0 \cap \mathfrak{O}_\varphi$ is an Eichler order of level $\#\ker\varphi=\deg \varphi$. Hence, its index in $\mathfrak{O}_0$ is $\deg\varphi$.
\end{proof}

\begin{details}
    \begin{tbox}
        Let $B$ be a quaternion algebra over $\Q$ of discriminant $D$. An order in $B$ is maximal if and only if its reduced discriminant is $D$ \cite[Theorem 15.5.5]{Voi21}. An Eichler order $\mathfrak{O}$ in $B$ of level $M$ has $\discrd(\mathfrak{O})=DM$ with $\gcd(D,M)=1$ \cite[Chapter 23.1]{Voi21}. Since for orders $\mathfrak{O}\subseteq \mathfrak{O}'$ in $B$ satisfy
        $$\discrd \mathfrak{O}=[\mathfrak{O}':\mathfrak{O}]\discrd(\mathfrak{O}')$$
        by \cite[Lemma 15.5.1]{Voi21}, if $\mathfrak{O}'$ is maximal and $\mathfrak{O}$ is an Eichler order of level $M$, then $[\mathfrak{O}':\mathfrak{O}]=M$.
    \end{tbox}
\end{details}

\begin{proposition} \label{maximal order to cyclic isogeny}
 For any maximal order $\mathfrak{O}$ in $\End^0(E_0)$, there is a cyclic isogeny $\varphi: E_0 \rightarrow E$ such that $\mathfrak{O}=\mathfrak{O}_{\varphi}$. 
\end{proposition}
\begin{proof}
    By \cite[Proposition 3.8]{Arp24}, there is a cyclic isogeny $\varphi : E_0 \rightarrow E$ such that
    \begin{equation}
        \End(E_0, \ker\varphi)=\mathfrak{O}_0 \cap \mathfrak{O}.
    \end{equation}
    By Proposition \ref{Index of Eichler order}, $\mathfrak{O}_0\cap\mathfrak{O}=\mathfrak{O}_0\cap \mathfrak{O}_{\varphi}$. Write $\mathfrak{O}_{\varphi, \ell}:=\mathfrak{O}_{\varphi}\otimes\mathbb{Z}_\ell$. For each prime $\ell$, 
    \begin{equation}
        \mathfrak{O}_\ell=\mathfrak{O}_{\varphi,\ell}
    \end{equation}
    by \cite[Proposition 23.4.3]{Voi21}. It follows that
    \begin{equation}
        \mathfrak{O}= \bigcap_{\ell}\left(\mathfrak{O}_\ell \cap B_0 \right)=\bigcap_{\ell}\left(\mathfrak{O}_{\varphi,\ell} \cap B_0\right)=\mathfrak{O}_{\varphi}.
    \end{equation}
\end{proof}

\begin{theorem} \label{bjection between curves, ideals, maximal orders}
There are bijections
\begin{equation}
\begin{array}{c@{\;}c@{\;}c}
 \SScy(E_0)  
& \longleftrightarrow & 
\left\{\text{primitive left} \ \mathfrak{O}_0\text{-ideals} \right\} \\[1ex]
(E,\varphi)
& \mapsto & I_{\varphi} \\[1ex]
(E_I, \varphi_I) & \mapsfrom & \substack{ \text{the unique integral connecting} \ \mathfrak{O}_0,\mathfrak{O}\text{-ideal} \\ I \text{of minimal reduced norm}}   \\[1ex]
& \longleftrightarrow & \{\text{maximal orders in } \End^0(E_0)\}  \\[1ex]
& \mapsto & \mathfrak{O}_{\varphi} \\[1ex]
& \mapsfrom & \mathfrak{O}.
\end{array}
\end{equation}
Moreover, for any separable isogeny $\varphi: E_0 \rightarrow E$, the map $\kappa_{\varphi}$ in \eqref{induced isomorphism of quaternion algebras equation} restricts to an isomorphism $\End(E) \cong \mathfrak{O}_\varphi$.
\end{theorem}

\begin{proof}
Given a cyclic isogeny $\varphi:E_0 \rightarrow E$, $I_\varphi$ is a primitive left $\mathfrak{O}_0$-ideal by Proposition \ref{cyclic isogeny and primitive ideal}. Also, since $\nrd(I_{\varphi})=\deg \varphi$, $I_{\varphi}$ is the unique integral connecting $\mathfrak{O}_0, \mathfrak{O}_\varphi$-ideal of minimal reduced norm by Proposition \ref{Index of Eichler order}.

Given a maximal order $\mathfrak{O}$ in $\End^0(E_0)$, there is a unique integral connecting $\mathfrak{O}_0, \mathfrak{O}$-ideal $I$ of reduced norm $[\mathfrak{O}_0:\mathfrak{O}_0 \cap \mathfrak{O}]$ by Lemma \ref{minimal reduced norm global}. By Proposition \ref{maximal order to cyclic isogeny}, there is a cyclic isogeny $\varphi: E_0 \rightarrow E$ such that $\mathfrak{O}=\mathfrak{O}_{\varphi}$. It follows that $I=I_{\varphi}$ is the primitive connecting $\mathfrak{O}_0, \mathfrak{O}_{\varphi}$-ideal by the previous argument. Furthermore, $E=E_I$ and $\varphi=\varphi_I$ for a pair $(E_I, \varphi_I)$ given by the inverse map in \eqref{Curves to Ideals inverse}.

The last statement follows from \cite[Lemma 42.2.9]{Voi21}. 
\end{proof}
\begin{details}
    \begin{tbox} 
    Proof of \cite[Lemma 42.2.9]{Voi21}
         Let $\beta \in \End(E)$. Then $\beta$ acts on $\Hom(E, E_0)$ by pre-composition
             \begin{equation} 
\begin{array}{c@{\;}c@{\;}c}
\End(E) \times \Hom(E, E_0)
& \rightarrow & 
\Hom(E, E_0) \\[1ex]
(\beta, \psi) 
& \mapsto & 
\psi \beta.
\end{array}
\end{equation}
Then the isomorphism $\varphi^\ast: \Hom(E, E_0) \rightarrow I_\varphi$ of left $\mathfrak{O}_0$-modules induces an action of $\beta$ on $I_\varphi$ given by
 \begin{equation} 
\begin{array}{c@{\;}c@{\;}c}
\End(E) \times I_\varphi
& \rightarrow & 
I_\varphi\\[1ex]
(\beta, \psi\varphi) 
& \mapsto & 
\psi \beta \varphi,
\end{array}
\end{equation}
which is the right multiplication by $\kappa_\varphi(\beta)$
as
$$\psi \beta \varphi = (\psi \varphi) (\frac{1}{\deg\varphi}\hat{\varphi}\beta \varphi)=(\psi\varphi)\kappa_\varphi(\beta).$$
Hence, $\kappa_\varphi(\beta) \in \calO_R(I_\varphi)=\frakO_\varphi$.

Next, we show that $\kappa_\varphi$ is injective; suppose $\kappa_\varphi(\beta)=0$. Then 
\begin{equation}
    \hat{\varphi}\beta \varphi=0
\end{equation}
Applying $\hat{\varphi}$ on the left on the both sides of the equation, we get $\beta \varphi=0$. Since $\varphi$ is a non-zero isogeny, which is surjective, this implies $\beta=0$.  

Lastly, since $\kappa_\varphi(\End(E)) \subseteq \mathfrak{O}_\varphi$ and $\End(E)$ is a maximal order, the equality holds.
    \end{tbox}
\end{details}

\begin{proposition}
 \label{primitive and minimal reduced norm ideal}
Let $\mathfrak{O},\mathfrak{O}'$ be maximal orders in $B_0$ and $I$ be an integral connecting $\mathfrak{O},\mathfrak{O}'$-ideal. The following are equivalent:
    \begin{enumerate}
        \item[(i)] $I$ is primitive.
        \item[(ii)] $I$ is the unique integral connecting $\mathfrak{O}, \mathfrak{O}'$-ideal of minimal reduced norm.
    \end{enumerate}
\end{proposition}
\begin{proof}
   This is shown in the proof of Theorem \ref{bjection between curves, ideals, maximal orders}. See also \cite[Lemma 8]{KLPT14} for a direct proof. 
\end{proof}

\begin{lemma} \label{index for isomorphic maximal orders}
Let $\mathfrak{O}$ be a maximal order in $B_0$ and $\mathfrak{O}'=\alpha^{-1}\mathfrak{O}\alpha$ for some $\alpha \in B_0^\times$. Let $I=\mathfrak{O}\alpha$. We have the following:
    \begin{enumerate}
        \item[(a)] $I$ is integral if and only if $\alpha \in \mathfrak{O}$.
        \item[(b)] $I$ is the primitive connecting $\mathfrak{O},\mathfrak{O}'$-ideal if and only if $\alpha \in \mathfrak{O}$ is primitive. In this case, $[\mathfrak{O} : \mathfrak{O} \cap \mathfrak{O}']=\nrd(\alpha)$.
    \end{enumerate}
\end{lemma}
\begin{proof}
    (a) If $I$ is integral, then $\alpha \in I \subseteq  \mathfrak{O}$. Conversely, if $\alpha \in \mathfrak{O}$, then $\mathfrak{O}\alpha \subseteq \mathfrak{O}$.

    (b) Suppose $I$ is primitive. If $\alpha \in \mathfrak{O}$ is not primitive, then there exists $\beta \in \mathfrak{O}$ such that $\alpha=n\beta$ for some integer $n>1$. Let $I'=\mathfrak{O}\beta$. $I'$ is an integral left $\mathfrak{O}$-ideal such that $I=nI'$. This is a contradiction to our assumption that $I$ is primitive. 
    
    Conversely, suppose $\alpha \in \mathfrak{O}$ is primitive. If $I$ is not primitive, then there is an integral left $\mathfrak{O}$-ideal $I'$ such that $I=nI'$ for some integer $n>1$. Hence, $n^{-1}\alpha \in I' \subseteq \mathfrak{O}$. This is a contradiction to our assumption that $\alpha \in \mathfrak{O}$ is primitive. 

    In case $I$ is primitive,
    \begin{equation}
        [\mathfrak{O}: \mathfrak{O}\cap \mathfrak{O}'] = \nrd(I)=\nrd(\alpha).
    \end{equation}
\end{proof}

\begin{extra}
    possible wrong proof:
     By Lemma \ref{minimal reduced norm local}, for each prime $\ell \neq p$, there is $\beta_\ell \in \mathfrak{O}_{\ell} \setminus \ell \mathfrak{O}_{\ell}$ unique up to left multiplication by $\mathfrak{O}_{\ell}^\times$ such that $\mathfrak{O}'_{\ell}=\beta_\ell^{-1}\mathfrak{O}_{\ell}\beta_{\ell}$. Consider $\alpha$ as an element of $B_\ell$, which we denote by $\alpha_\ell$. We also have
    \begin{equation}
\mathfrak{O}'_{\ell}=\alpha_\ell^{-1}\mathfrak{O}_{\ell}\alpha_\ell.
    \end{equation}
    By the uniqueness of $\beta_\ell$, we have $\alpha_\ell=\mu\beta_\ell$ for some $\mu \in \mathfrak{O}_{\ell}^\times$, i.e., $\alpha_\ell \in \mathfrak{O}_{\ell}$. By Lemma \ref{minimal reduced norm local},
    \begin{equation}
       \ell^{\nu_\ell(\nrd(\alpha))}=\nrd(\alpha_\ell)  = \nrd(\beta_\ell) =[\mathfrak{O}_\ell:\mathfrak{O}_\ell \cap \mathfrak{O}'_\ell],
    \end{equation}
    where $\nu_\ell(n)$ denotes the $\ell$-adic valuation of an integer $n$.
    
    Also, since $\mathfrak{O}_p$ is the unique maximal order in $B_p$, $\alpha_p \in \mathfrak{O}_p^\times$ and
    \begin{equation}
        p^{\nu_p(\nrd(\alpha))} = [\mathfrak{O}_p : \mathfrak{O}_p \cap \mathfrak{O}_p']=1.
    \end{equation}
   It follows that  $\alpha \in \mathfrak{O}$ and $I$ is integral where    
    \begin{equation}
        \nrd(I)=\nrd(\alpha) = [\mathfrak{O}: \mathfrak{O}\cap \mathfrak{O'}].
    \end{equation}
        
    Therefore, $I$ is the unique integral connecting $\mathfrak{O},\mathfrak{O}'$-ideal of minimal reduced norm.
\end{extra}

\begin{extra}
     The index $[\mathfrak{O} : \mathfrak{O}\cap \mathfrak{O'}]$ is the size of the finite abelian group $\mathfrak{O}/(\mathfrak{O}\cap\mathfrak{O}')$, and locally we have $\mathfrak{O}_\ell/(\mathfrak{O}_\ell \cap \mathfrak{O}_\ell')$ and the corresponding index $[\mathfrak{O}_\ell : \mathfrak{O}_\ell \cap \mathfrak{O}_\ell']$. Since
    \begin{equation}
        \mathfrak{O}/(\mathfrak{O} \cap \mathfrak{O}') \cong \bigoplus_\ell (\mathfrak{O}_\ell/(\mathfrak{O}_\ell \cap \mathfrak{O}_\ell'))
    \end{equation}
    via the decomposition $\hat{\Z}=\prod_\ell \Z_\ell $.
\end{extra}

\begin{lemma} \label{primitivity of inverse ideal}
Let $I$ be the primitive connecting $\mathfrak{O},\mathfrak{O}'$-ideal of maximal orders $\mathfrak{O}, \mathfrak{O}'$ in $B_0$. Then $\bar{I}$ is the primitive connecting $\mathfrak{O}',\mathfrak{O}$-ideal.
\end{lemma}
\begin{proof}
   If $I$ is the primitive connecting $\mathfrak{O},\mathfrak{O}'$-ideal, then $\bar{I}$ is an integral connecting $\mathfrak{O}',\mathfrak{O}$-ideal of reduced norm
   \begin{equation}
       \nrd(\bar{I})=\nrd(I)=[\mathfrak{O}: \mathfrak{O}\cap\mathfrak{O}'].
   \end{equation}
\end{proof}

\begin{remark} \label{unique cyclic isogeny}
    Let $(E,\varphi) \in \SS(E_0)$ where $\varphi$ is not necessarily cyclic. For the maximal order $\mathfrak{O}=\mathfrak{O}_{\varphi}$, there exists a cyclic isogeny $\varphi': E_0 \rightarrow E'$ such that $\mathfrak{O}=\mathfrak{O}_{\varphi'}$. Since $I_{\varphi}, I_{\varphi'}$ are both integral connecting $\mathfrak{O}_0,\mathfrak{O}$-ideals and $I_{\varphi'}$ is primitive, $I_{\varphi}=nI_{\varphi'}$ for some integer $n\geq 1$. This implies $\varphi \sim \varphi'[n]$. Up to post-composition by an isomorphism, we may assume $E=E'$ and $\varphi=\varphi'[n]$. Then $\varphi$ and $\varphi'$ induce the same orientation: for a given $K$-orientation $\iota_0$ on $E_0$,
    \begin{equation}
        \varphi_\ast(\iota_0)(\alpha) = \frac{1}{\deg \varphi}\varphi \iota_0(\alpha)\hat{\varphi} =  \frac{1}{n^2\deg \varphi'}[n]\varphi' \iota_0(\alpha)\hat{\varphi'}[n]=\varphi'_\ast(\iota_0)(\alpha)
    \end{equation}
    for all $\alpha \in K$.
\end{remark}

\begin{details}
    \begin{tbox}
 It was stated without proof in \href{https://eprint.iacr.org/2025/042.pdf}{2.4 Integral ideals} that any connecting ideal is scalar multiple of connecting ideal of minimal reduced norm.\\

 $\ker(\varphi)= E[I_{\varphi}]=E[nI_{\varphi'}]=\ker(\varphi'[n])$. It was shown in \cite[proof of Lemma 42.2.13]{Voi21} that for $\beta \in B^\times_0$, $\beta E[I\beta]=E[I]$ and $\varphi_{I\beta}=\varphi_I\beta$
\end{tbox}
\end{details}

\begin{extra}
\begin{tbox}
    \textbf{Discussion on inseparable isogenies}

     We want to extend the bijections to inseparable isogenies
    \begin{equation}
\begin{array}{c@{\;}c@{\;}c}
\left\{ \text{equivalence classes of } (E,\varphi) \right\} 
& \longleftrightarrow & 
\left\{\text{non-zero left integral} \ \mathfrak{O}_0\text{-ideals} \right\} \\[1ex]
(E,\varphi)/\sim 
& \mapsto & 
I_{\varphi}:=\Hom(E, E_0)\varphi.
\end{array}
\end{equation}    
 \begin{equation} 
\begin{array}{c@{\;}c@{\;}c}
\{\text{non-zero left integral} \ \mathfrak{O}_0\text{-ideals}\} 
& \longleftrightarrow & 
\{\text{maximal orders in} \End^0(E_0)\} \\[1ex]
I_\varphi 
& \mapsto & 
\mathfrak{O}_\varphi \\
\substack{ \text{the unique integral connecting} \ \mathfrak{O}_0,\mathfrak{O}\text{-ideal} \\ \text{of reduced norm} \ [\mathfrak{O}: \mathfrak{O}_0 \cap \mathfrak{O}]}  &  \mapsfrom & \mathfrak{O}.
\end{array}
\end{equation}

   Suppose $\varphi$ is inseparable and write 
\begin{equation}
    \varphi=\varphi_{\sep}\pi_q
\end{equation}
     where $q=\deg_i(\varphi)=p^r$. Then
     \begin{equation}
         I_{\varphi}=\mathfrak{P}^rI_{\varphi_{\sep}}
     \end{equation}
   It follows that $\mathfrak{O}_{\varphi}=\mathfrak{O}_{\varphi_{\sep}}$.

     Then we have another issue that $I_{\varphi}$ is not the integral connecting ideal of minimal reduced norm as 
     \begin{equation}
         [\mathfrak{O}_0 : \mathfrak{O}_0 \cap \mathfrak{O}_{\varphi}]=[\mathfrak{O}_0 : \mathfrak{O}_0 \cap \mathfrak{O}_{\varphi_{\sep}}]=\deg_s \varphi
     \end{equation}
     and
     \begin{equation}
         \nrd(I_\varphi)=p^r + \deg \varphi_{\sep}=\deg_i \varphi+\deg_s \varphi = \deg\varphi.
     \end{equation}
\end{tbox}
\end{extra}

\subsection{Connecting ideals}
We further investigate the properties of primitive connecting ideals between maximal orders in $B_0$. 

\begin{lemma} \label{colon ideals} Let $\mathfrak{O}_0,\mathfrak{O}, \mathfrak{O}'$ be maximal ideals in $B_0$. Let $I$ be a connecting $\mathfrak{O}_0, \mathfrak{O}$-ideal and $I'$ be a connecting $\mathfrak{O}_0,\mathfrak{O}'$-ideal. Let
    \begin{equation}
        J=(I':I)_R= I^{-1}I'.
    \end{equation}
We have the following:
    \begin{enumerate}
        \item[(a)] $J$ is a connecting $\mathfrak{O}, \mathfrak{O}'$-ideal.
        \item[(b)] $\nrd(J) = \nrd(I')/\nrd(I)$.
        \item[(c)] $I' \subseteq I$ if and only if $J$ is integral.
        \item[(d)] Suppose $I' \subseteq I$. If $I'$ is primitive, then $J$ is primitive.
    \end{enumerate}  
\end{lemma}

\begin{details}
    \begin{tbox}
        In part $(c)$, $I$ and $I'$ integral doesn't necessarily imply $J$ is integral.
    \end{tbox}
\end{details}
   
\begin{proof}          
    (a) We have
    \begin{equation}
        \calO_L(I^{-1})= \calO_{L}((1/\nrd(I))\overline{I})= \calO_L(\overline{I})=\mathcal{O}_R(I)= \mathfrak{O}
    \end{equation}
    and similarly, $\calO_R(I^{-1})=\calO_L(I)=\mathfrak{O}_0$. Hence, 
    \begin{equation}
        \calO_R(I^{-1})=\calO_L(I').
    \end{equation}
   By \cite[Lemma 16.5.11]{Voi21}, 
    \begin{equation}
 \calO_L(I^{-1}I')=\calO_L(I^{-1})=\mathfrak{O}
    \end{equation}
    and similarly, $\calO_{R}(I^{-1}I')=\calO_R(I')=\mathfrak{O}'$. Therefore, $I^{-1}I'$ is a connecting $\mathfrak{O},\mathfrak{O}'$-ideal.

\begin{details}
\begin{tbox}
    Being a connecting ideal requires being an invertible lattice by definition, so $I$ is invertible. By definition, for any $a \in \Q$, 
    \begin{equation}
        \calO_L(aI)=\{\alpha \in B_0: \alpha aI \subseteq aI \}=\{\alpha \in B_0 : \alpha I \subseteq I\}=\calO_L(I).
    \end{equation}
    Similarly, $\calO_R(aI)=\calO_R(I)$. 
\end{tbox}
\end{details}

    (b) By \cite[Lemma 16.3.7]{Voi21}, 
    \begin{equation}
        \nrd(I^{-1}I')=\nrd(I^{-1})\nrd(I').
    \end{equation}
    Also,
    \begin{equation}
        \nrd(I^{-1})= \nrd((1/\nrd(I)) \overline{I})= \frac{1}{(\nrd(I))^2}\nrd(\overline{I})=\frac{1}{\nrd(I)}.
    \end{equation}
    Hence, 
    \begin{equation}
        \nrd(I^{-1}I')=\frac{\nrd(I')}{\nrd(I)}.
    \end{equation}

(c) Suppose $I' \subseteq I$. 
\begin{equation}
    J= I^{-1}I' \subseteq I^{-1}I=\mathfrak{O}.
\end{equation}

Conversely, suppose $J$ is integral. Then $J \subseteq \mathfrak{O}$ and 
\begin{equation}
    I'=\mathfrak{O}_0I'=(II^{-1})I'=IJ \subseteq I\mathfrak{O}=I.
\end{equation}

\begin{details}
    \begin{tbox}
        Use definitions:
        $\calO_R(I)=\{\alpha \in B_0 : I\alpha \subseteq I\}, \calO_L(I)=\{\alpha \in B_0 : \alpha I \subseteq I\}, I^{-1}I=\calO_R(I), II^{-1}=\calO_L(I)$.
    \end{tbox}
\end{details}
    
(d) Suppose $I'\subseteq I$ and $I'$ is primitive. Assume for contradiction that $J$ is a non-primitive integral ideal. Then there is an integral connecting $\mathfrak{O}, \mathfrak{O}'$-ideal $J'$ such that $J=nJ'$ for an integer $n>1$. Since $\calO_R(I)=\calO_L(J')=\mathfrak{O}$ and
\begin{equation}
    IJ' \subseteq I\mathfrak{O}=I \subseteq \mathfrak{O}_0,
\end{equation}
$IJ'$ is an integral connecting $\mathfrak{O}_0, \mathfrak{O}'$-ideal. We have
\begin{equation}
    I'=IJ=n(IJ'),
\end{equation}
which is a contradiction to our assumption that $I'$ is primitive. Therefore, $J$ is primitive.
\end{proof}

\begin{extra}
    \begin{tbox}
        old wrong proof:

Let $J'$ be the primitive $\mathfrak{O},\mathfrak{O}'$-ideal. We show that $J=J'$ locally.

Let $\ell \neq p$ be a prime and write $\mathfrak{O}_{0,\ell}:=\mathfrak{O}_0\otimes \Z_\ell$.

 Since $I$ and $I'$ are locally principal,
    \begin{align} 
        I_\ell & = \mathfrak{O}_{0,\ell}\alpha_\ell \\
    I_\ell'&=\mathfrak{O}_{0,\ell}\beta_\ell
    \end{align}
for some $\alpha_{\ell}, \beta_{\ell} \in B_{\ell}^\times$. It follows that
   \begin{align} \label{local-conjugate}
       \mathfrak{O}_\ell& = \alpha_{\ell}^{-1}\mathfrak{O}_{0,\ell} \alpha_{\ell} \\\mathfrak{O}'_{\ell}&=\beta_{\ell}^{-1} \mathfrak{O}_{0,\ell} \beta_{\ell}.
   \end{align}
We have
\begin{equation}
    J_\ell=I_\ell^{-1}I_{\ell}' =(\alpha_\ell^{-1}\mathfrak{O}_{0,\ell} )( \mathfrak{O}_{0,\ell}\beta_\ell)=\alpha_\ell^{-1}\mathfrak{O}_{0,\ell}\beta_\ell=\mathfrak{O}_\ell\alpha_{\ell}^{-1}\beta_\ell,
\end{equation}
where the second last equality follows from \eqref{local-conjugate}. In particular,
\begin{equation}
    J_\ell = \alpha_{\ell}^{-1}\mathfrak{O}_{0,\ell}\beta_\ell=\alpha_\ell^{-1} I_\ell' \subseteq \alpha_\ell^{-1}  \mathfrak{O}_{0, \ell}=\mathfrak{O}_\ell
\end{equation}

Scaling $\alpha_\ell, \beta_\ell$ by integers, we may assume that $\alpha_\ell, \beta_\ell \in  \mathfrak{O}_{0,\ell} \setminus \ell \mathfrak{O}_{0,\ell}$. By Lemma \ref{minimal reduced norm local}, $I_\ell = \mathfrak{O}_{0,\ell}\alpha_\ell$ is the primitive connecting $\mathfrak{O}_{0,\ell}, \mathfrak{O}_\ell$-ideal and $I_\ell'=\mathfrak{O}_{0,\ell}\beta_\ell$ is the primitive connecting $\mathfrak{O}_{0,\ell}, \mathfrak{O}_\ell'$-ideal. 

Also, $J_\ell$ is a connecting $\mathfrak{O}_\ell, \mathfrak{O}_\ell'$-ideal by part (a). Suppose $J_\ell$ is not primitive. Since $J'_\ell$ is the primitive connecting $\mathfrak{O}_{\ell}, \mathfrak{O}_\ell'$-ideal, $J_\ell=nJ_\ell'$ for some integer $n >1$. It follows that

Since there is a unique local maximal order in $B_p$, we also have 
\begin{equation}
    J_p = J'_p.
\end{equation}
Therefore, we conclude that $J= J'$.
    \end{tbox}
\end{extra}

We carefully describe how the cyclic isogeny $\varphi : E_0 \rightarrow E$ changes the associated $I_\varphi$ and $\mathfrak{O}_\varphi$ since in all subsequent proofs we require a marking $\varphi$ to obtain a left $\mathfrak{O}_0$-ideal or maximal order in $B_0$, rather than just an isomorphic copy.

\begin{lemma} \label{composition by endomorphism}
 Let $\varphi \in \Hom(E_0, E)$, $\delta \in \End(E_0)$, and $\epsilon \in \End(E)$ be separable isogenies. We have the following:
 \begin{enumerate}
   \item[(a)] $I_{\epsilon \varphi} = I_{\varphi}J$,  where 
   \begin{equation*}
    J=(I_{\epsilon\varphi}: I_{\varphi})_R=I^{-1}_{\varphi}I_{\epsilon\varphi}= \kappa_{\varphi}(\tilde{J})
   \end{equation*}
   and $\tilde{J}$ is a left $\End(E)$-ideal with $E[\tilde{J}]=\ker \epsilon$.
   \item[(b)] $\mathfrak{O}_{\epsilon \varphi} = \kappa_\varphi(\epsilon^{-1} \End(E) \epsilon)$,
   \item[(c)] $I_{\varphi \delta} = I_\varphi \delta$,
   \item[(d)] $\mathfrak{O}_{\varphi \delta} = \delta^{-1} \mathfrak{O}_{\varphi} \delta$.  
 \end{enumerate}
\end{lemma}

\begin{details}
\begin{tbox}
    \href{https://www.math.auckland.ac.nz/~sgal018/crypto-book/main.pdf}{[Gal18 Exercise 25.1.1.]} Let $\varphi : E_0 \rightarrow E$ be a separable isogeny.
    \begin{enumerate}
        \item $\lambda \circ \varphi \sim \varphi$ for $\lambda \in \Aut(E)$
        \item $\varphi \circ \lambda$ is not necessarily equivalent to $\varphi$ for $\lambda \in \Aut(E_0)$.
    \end{enumerate}
\end{tbox}
\end{details}
\begin{proof}

(a) Since $\ker \varphi \subseteq \ker \epsilon \varphi$, $I_{\epsilon\varphi} \subseteq I_{\varphi}$. Then by Proposition \ref{EHL+18 Proposition 12} and Lemma~\ref{colon ideals}(a), $J=I_{\varphi}^{-1}I_{\epsilon \varphi}$ is a connecting $\mathfrak{O}_\varphi, \mathfrak{O}_{\epsilon\varphi}$-ideal such that $I_{\epsilon \varphi}=I_{\varphi}J$ where $J=\kappa_{\varphi}(\tilde{J})$ for a left $\End(E)$-ideal $\tilde{J}$ with $E[\tilde{J}]=\ker\epsilon$. \\
(b) By Theorem \ref{bjection between curves, ideals, maximal orders}, we have that $\mathfrak{O}_{\epsilon \varphi} = \kappa_\varphi(\kappa_\epsilon(\End(E)))$.\\
(c) By the definition of $I_\varphi$ in Proposition \ref{Curves to Ideals}, we obtain (c).\\
(d) By Theorem \ref{bjection between curves, ideals, maximal orders}, we have that $\mathfrak{O}_{\varphi \delta} = \kappa_\delta(\kappa_\varphi (\End(E)))$.
\end{proof}

\begin{remark}
Note that $\mathfrak{O}_{\epsilon \varphi}, \mathfrak{O}_{\varphi}$ are not necessarily the same unless $\ker(\epsilon\varphi) = \ker \varphi$. If $\epsilon =\lambda \in \Aut(E)$, then $\mathfrak{O}_{\lambda \varphi}=\mathfrak{O}_{\varphi}$.
\end{remark}

\begin{proposition} \label{cyclic factor through} 
    Let $\varphi: E_0 \rightarrow E$ and $\varphi': E_0 \rightarrow E'$ be cyclic isogenies. If $\ker\varphi \subseteq \ker\varphi'$, then there is a cyclic isogeny $\psi:E\rightarrow E'$ such that $\varphi'=\psi\varphi$.
\end{proposition}
\begin{proof}
   Since $\varphi, \varphi'$ are cylic, $I_\varphi, I_{\varphi'}$ are primitive by Proposition \ref{cyclic isogeny and primitive ideal}. Also, $\ker \varphi \subseteq \ker \varphi'$ implies $I_{\varphi'} \subseteq I_{\varphi}$. It follows that $J=I_{\varphi}^{-1}I_{\varphi'}$ is primitive by Lemma \ref{colon ideals}(d).  Also, there is a separable isogeny $\psi: E \rightarrow E'$ such that $\varphi'=\psi \varphi$, and a left $\End(E)$-ideal $\tilde{J}$ such that $E[\tilde{J}]=\ker \psi$ and $J=\kappa_{\varphi}(\tilde{J})$ by Proposition \ref{EHL+18 Proposition 12}. Since $\kappa_{\varphi}: \End^0(E)\rightarrow \End^0(E_0)$ is an isomorphism of $\Q$-algebras and $J$ is primitive, $\tilde{J}$ is primitive. Therefore, $\psi$ is cyclic by Proposition \ref{cyclic isogeny and primitive ideal}.
\end{proof}

\begin{extra}
    \begin{tbox}
    {\color{red} Proposition 2.20 may not be true? An isogeny $\phi: E \rightarrow E'$ of degree $\ell \not= p$ factors as:
\begin{equation}
  E \rightarrow_\phi E' \rightarrow_{\hat \phi} E
\end{equation}
where $\hat \phi$ is the dual of $\phi$ and both are cyclic of degree $\ell$ but the composition has kernel $E[\ell]$ which is not cyclic.
}

{\color{blue} should be true for non-backtracking. I don't think we need this anyway}

\begin{proposition} \label{cyclic composite} {\color{red}
    Let $\varphi: E_0 \rightarrow E$ and $\psi: E \rightarrow E'$ be cyclic isogenies where $\psi \neq \hat{\varphi}$. Then $\psi \varphi$ is cyclic.}
\end{proposition}
\begin{proof}
    Write $\varphi'=\psi\varphi$. Since $\ker \varphi \subseteq \ker \varphi'$, $I_{\varphi'} \subseteq I_{\varphi}$. By Proposition \ref{EHL+18 Proposition 12}, $J=I_{\varphi}^{-1}I_{\varphi'}$ is a left $\mathfrak{O}_\varphi$-ideal such that $J=\kappa_{\varphi}(\tilde{J})$ for a left $\End(E)$-ideal $\tilde{J}$ with $E[\tilde{J}]=\ker\psi$. Since $\psi$ is cyclic, $J$ is primitive by Proposition \ref{cyclic isogeny and primitive ideal}. Assume for contradiction that $I_{\varphi'}$ is not primitive. Let $n>1$ be the minimal integer such that $I_{\varphi'} \subseteq n \mathfrak{O}_0$. Then $I'=n^{-1}I_{\varphi'}$ is primitive. Let $\mathfrak{O}'=\calO_R(I_{\varphi'})=\calO_R(I')$. By Lemma \ref{colon ideals}(d), $J'=I_\varphi^{-1}I'$ is the primitive connecting $\mathfrak{O}_{\varphi}, \mathfrak{O}'$-ideal. Since $J$ is also the primitive $\mathfrak{O}_{\varphi}, \mathfrak{O}'$-ideal, 
    \begin{equation}
        J=J',
    \end{equation}
    which is a contradiction as it implies $I_{\varphi'}=I'$. Therefore, $I_{\varphi'}$ is primitive and $\varphi'$ is cyclic.
\end{proof}
\end{tbox}
\end{extra}

Given two pairs $(E,\varphi), (E',\varphi') \in \SS(E_0)$, we will call $\psi: (E, \varphi) \rightarrow (E', \varphi')$ a compatible isogeny if $\psi: E \rightarrow E'$ is an isogeny such that $m\varphi' \sim n\psi\varphi$ for some non-zero integers $m,n$. We say a compatible isogeny is cyclic if the underlying isogeny is cyclic. Two compatible isogenies $\psi: (E,\varphi) \rightarrow (E', \varphi'), \psi' : (E, \varphi) \rightarrow (E'',\varphi'')$ are equivalent if the underlying isogenies are equivalent.

\begin{details}
\begin{tbox}
    Initially defined a compatible isogeny to be an isogeny such that $\varphi' \sim \psi\varphi$. But this definition is not well-defined for dual as $\psi^{-1}$ is not necessarily an isogeny. Note that $\varphi' \sim \psi\varphi$ implies $\hat{\psi}\varphi' \sim (\deg \psi)\varphi$.
\end{tbox}
\end{details}

\begin{definition}
    Let $\mathfrak{O}, \mathfrak{O}'$ be maximal orders in $B_0$. Given a connecting $\mathfrak{O},\mathfrak{O}'$-ideal $I$, the primitive part of $I$, denoted $\operatorname{pr}(I)$, means the primitive connecting $\mathfrak{O},\mathfrak{O}'$-ideal.
\end{definition}

\begin{proposition}\label{bijection of edges} There is a bijection
 \begin{equation} 
\begin{array}{c@{\;}c@{\;}c}
\left\{ \substack{\text{equivalence classes of cyclic isogenies between} \\ \text{pairs in } \SScy(E_0)} \right\}
& \longleftrightarrow & 
\left\{ \substack{\text{primitive connecting ideals between} \\ \text{maximal orders in } \End^0(E_0) }\right\} \\[1ex]
\psi:(E,\varphi) \rightarrow (E',\varphi') /\sim
& \mapsto & 
\operatorname{pr}(I_{\varphi}^{-1}I_{\varphi}),
\end{array}
\end{equation}
where $\varphi : E_0 \rightarrow E$ and $\varphi': E_0 \rightarrow E$ are cyclic isogenies.
\end{proposition}
\begin{proof}
Consider two equivalence classes $(E,\varphi)/\sim$ and $(E',\varphi')/\sim$ of pairs in $ \SScy(E_0)$. The first bijection in Theorem \ref{bjection between curves, ideals, maximal orders} sends $\varphi: E_0 \rightarrow E /\sim$ and $\varphi': E_0 \rightarrow E'/ \sim$ to $I_\varphi$ and $I_{\varphi'}$ respectively, and the second bijection sends $(E,\varphi)/\sim$ and $(E',\varphi')/\sim$ to $\mathfrak{O}_{\varphi}$ and $\mathfrak{O}_{\varphi'}$ respectively.
\[
\begin{tikzcd}
(E, \varphi)/\sim \arrow[rr, "\psi/\sim"] && (E', \varphi')/\sim  \\
& E_0  \arrow[ul, "\varphi/\sim"] \arrow[ur, "\varphi'/\sim"']&
\end{tikzcd}
\qquad
\begin{tikzcd}
\mathfrak{O}_{\varphi} \arrow[rr, "J=\operatorname{pr}(I_{\varphi}^{-1} I_{\varphi'})"] && \mathfrak{O}_{\varphi'}  \\
& \mathfrak{O}_0 \arrow[ul, "I_{\varphi}"]  \arrow[ur, "I_{\varphi'}"']  &
\end{tikzcd}
\]

Given a cyclic compatible isogeny $\psi : (E, \varphi) \rightarrow (E',\varphi')$, $I_{\varphi}^{-1}I_{\varphi'}$ is a  connecting $\mathfrak{O}_{\varphi}, \mathfrak{O}_{\varphi'}$-ideal by Lemma \ref{colon ideals}(a). We show that the map
\begin{equation}
    \psi : (E, \varphi) \rightarrow (E', \varphi')/\sim \ \mapsto \operatorname{pr}(I_{\varphi}^{-1}I_{\varphi})
\end{equation}
gives a bijection between the two sets. 

For the subjectivity of the map, let $J$ be the primitive connecting $\mathfrak{O}_{\varphi}, \mathfrak{O}_{\varphi'}$-ideal. Then $I_{\varphi}^{-1}I_{\varphi}=r J$ for some $r=\frac{m}{n} \in \Q^\times$ where $m,n \in \Z$. Let $\tilde{J}$ be the primitive left $\End(E)$-ideal such that $\kappa_{\varphi}(\tilde{J})=J$. By Theorem \ref{bjection between curves, ideals, maximal orders}, there is a bijection between $\SScy(E)$ and the set of non-zero primitive left $\End(E)$-ideals. Hence, there is a cyclic isogeny $\psi: E \rightarrow E'$ such that $\tilde{J}=\Hom(E', E)\psi$. By \cite[Lemma 42.2.22]{Voi21}, there is a bijection
\begin{equation}
    \begin{split}
        \Hom(E', E) & \rightarrow I_{\varphi}^{-1}I_{\varphi'} \\
        \phi & \mapsto \varphi^{-1}\phi \varphi'
    \end{split}
\end{equation}
as $\ker\varphi=E_0[I_{\varphi}]$ and $\ker\varphi'=E_0[I_{\varphi'}]$. Hence,
\begin{equation}
    J=rI_{\varphi}^{-1}I_{\varphi'}=\varphi^{-1}\Hom(E',E)r\varphi'
\end{equation}
as sets. Also, 
\begin{equation}
    \kappa_{\varphi}(\tilde{J})= \varphi^{-1}\Hom(E', E)\psi \varphi.
\end{equation}
Hence, $\kappa_{\varphi}(\tilde{J})=J$ implies $m\varphi'\sim n\psi\varphi$. Hence, $\psi: (E, \varphi) \rightarrow (E', \varphi')$ is a cyclic compatible isogeny.

For the injectivity of the map, let $\psi: (E, \varphi)\rightarrow (E', \varphi'), \psi':(E_1, \varphi_1) \rightarrow (E'_1, \varphi_1')$ be cyclic compatible isogenies between pairs in $\SS^\ast(E_0)$ such that $J=\operatorname{pr}(I_{\varphi}^{-1}I_{\varphi'})=\operatorname{pr}(I_{\varphi_1}^{-1}I_{\varphi_1'})$. Then we have $\mathfrak{O}_{\varphi}=\mathfrak{O}_{\varphi_1}=\calO_L(J)$ and $\mathfrak{O}_{\varphi'}=\mathfrak{O}_{\varphi_1'}=\calO_R(J)$. This implies
\begin{equation} \label{bijection of edges eq1}
    \varphi \sim \varphi_1 \quad \text{and} \quad \varphi' \sim \varphi_1'
\end{equation}
by Theorem \ref{bjection between curves, ideals, maximal orders}. Let $J=r I_{\varphi}^{-1}I_{\varphi'}=r_1I_{\varphi_1}^{-1}I_{\varphi'_1}$ where $r, r_1 \in \Q^\times$. By Lemma \ref{colon ideals}(b),
\begin{equation}
    \nrd(J)=r^2\frac{\nrd(I_{\varphi'})}{\nrd(I_{\varphi})}=r_1^2 \frac{\nrd(I_{\varphi_1'})}{\nrd(I_{\varphi_1})}.
\end{equation}
By \eqref{bijection of edges eq1}, $I_{\varphi}=I_{\varphi_1}$ and $I_{\varphi'}=I_{\varphi_1'}$. Hence, $r=r_1$. Let $r=\frac{m}{n}$ for $m, n \in \Z$. Then $m\varphi' \sim  n\psi \varphi$ and $m\varphi_1' \sim n \psi' \varphi_1$ by the same argument in the proof of surjectivity. It follows that
\begin{equation}
    n\psi \sim m\varphi'\varphi^{-1}\sim m\varphi_1'\varphi_1^{-1} \sim n\psi'.
\end{equation}
Therefore, $\psi \sim \psi'$.
\end{proof}

\section{Orientations and maximal orders} \label{Orientations and maximal orders}
In this section, we introduce useful notions of orientations in the quaternionic
context. In particular, we extend the notion of orientations to what we call double-orientations, and establish results analogous to those for single-orientations.

\begin{extra}
    the usual $K$-orientations correspond to embeddings of imaginary quadratic fields into the endomorphism algebras of supersingular elliptic curves. Our goal is to generalize this by considering embeddings of two such fields simultaneously. 
\end{extra}

\subsection{Single-orientations with maximal orders} Let $K$ be an imaginary quadratic field. Consider a $K$-orientation $\iota_0$ on $E_0$. Each pair $(E,\varphi)$ in $\SS(E_0)$ admits an induced $K$-orientation $\iota=\varphi_\ast(\iota_0)$, and $\varphi:(E_0, \iota_0) \rightarrow (E, \iota)$ is a $K$-oriented isogeny. We say a $K$-oriented isogeny is cyclic if the underlying isogeny is cyclic. By Remark \ref{unique cyclic isogeny}, there is a cyclic $K$-oriented isogeny $\varphi':(E_0, \iota_0) \rightarrow (E,\iota)$ and $\varphi \sim \varphi'[n]$ for some integer $n \geq 1$.

Since the isomorphism of $\Q$-algebras $\kappa_\varphi: \End^0(E) \rightarrow \End^0(E_0)$ restricts to an isomorphism of $\Q$-subalgebras $\iota(K) \rightarrow \iota_0(K)$, by definition, $\iota_0$ factors through $\kappa_\varphi$ as follows:
\[
\begin{tikzcd}
K \arrow[r, "\iota_0"] \arrow[dr, swap, "\iota"] & \iota_0(K) &   \\
& \iota(K) \arrow[u, "\kappa_\varphi"'] & 
\end{tikzcd}
\]
\begin{equation} \label{embedding single orientation}
    \kappa_{\varphi}(\iota(\alpha)) = \iota_0(\alpha) \ \forall \alpha \in K.
\end{equation}

\begin{details}
        \begin{tbox}
    Since $\iota(\alpha)=\frac{1}{\deg\varphi}\varphi\iota_0(\alpha)\hat{\varphi}$ for $\alpha \in K$,
        \begin{equation}
            \kappa_\varphi(\iota(\alpha))= \frac{1}{\deg \varphi} \hat\varphi \iota(\alpha) \varphi=\frac{1}{\deg \varphi} \hat\varphi \left(\frac{1}{\deg\varphi}\varphi\iota_0(\alpha)\hat{\varphi} \right) \varphi=\iota_0(\alpha).
        \end{equation}
    \end{tbox}
\end{details}

It follows that given an order $\calO$ in $K$,
    \begin{equation} \label{orientation to optimal embedding}
    \begin{split}
        \iota \ \text{is a primitive} \ \mathcal{O}\text{-orientation} \ \text{on} \ E & \Longleftrightarrow \iota(\calO)=\End(E) \cap \iota(K) \\
        & \Longleftrightarrow \kappa_\varphi(\iota(\calO))=\kappa_\varphi(\End(E)\cap \iota(K))\\
        & \Longleftrightarrow \iota_0(\calO)=\mathfrak{O}_\varphi\cap \iota_0(K).
    \end{split}
\end{equation}

Let $\varphi: (E_0, \iota_0) \rightarrow (E,\iota)$, $\varphi': (E_0, \iota_0) \rightarrow (E', \iota')$, and $\psi : (E, \iota) \rightarrow (E', \iota')$ be $K$-oriented isogenies such that $\varphi'=\psi\varphi$.
\[
\begin{tikzcd}
(E_0, \iota_0) \arrow[r, "\varphi" above] \arrow[rr, bend right=20, "\varphi'" below] 
    & (E, \iota) \arrow[r, "\psi" above] 
    & (E', \iota')
\end{tikzcd}
\]
Then we have the following commutative diagram of isomorphisms of $\Q$-algebras
    \[
\begin{tikzcd}
K \arrow[r, "\iota_0"] \arrow[dr, swap, "\iota"] \arrow[ddr, "\iota'"'] & \iota_0(K) &  \\
& \iota(K) \arrow[u, "\kappa_\varphi"'] &    \\
& \iota'(K)  \arrow[u, "\kappa_\psi"'] \arrow[uu, bend right=40, "\kappa_{\varphi'}"'] & 
\end{tikzcd}
\]
with
\begin{equation} \label{induced homomorphism of quadratic fields}
    \kappa_{\varphi'}(\End(E') \cap \iota'(K))=\kappa_\varphi(\kappa_\psi(\End(E') \cap \iota'(K))).
\end{equation}

For an embedding $\iota_0:K\rightarrow B_0$, we define $\iota_0$-isomorphism classes of maximal orders in $B_0$ as follows.
\begin{definition}
We say two maximal orders $\mathfrak{O},\mathfrak{O}'$ in $B_0$ are $\iota_0$-isomorphic, denoted by $\mathfrak{O}\cong_{\iota_0} \mathfrak{O}'$, if $\mathfrak{O}'=\alpha^{-1}\mathfrak{O}\alpha$ for some $\alpha \in \iota_0(K)^\times$. We will call the $\iota_0$-isomorphism class of $\mathfrak{O}$ the $\iota_0$-conjugacy class of $\mathfrak{O}$.  
\end{definition}

\begin{proposition} \label{single oriented isogeny}
    Let $\varphi: (E_0, \iota_0) \rightarrow (E, \iota)$ and $\varphi': (E_0, \iota_0) \rightarrow (E', \iota')$ be cyclic $K$-oriented isogenies. The following hold:
    \begin{enumerate}
        \item[(a)] If $(E, \iota) \cong (E', \iota')$, then $\mathfrak{O}_{\varphi}\cap \iota_0(K)=\mathfrak{O}_{\varphi'}\cap \iota_0(K)$.
        \item[(b)] $\mathfrak{O}_{\varphi}\cap \iota_0(K)=\mathfrak{O}_{\varphi'}\cap \iota_0(K)$ if and only if $\mathfrak{O}_{\varphi'}= \iota_0(\a)^{-1}\mathfrak{O}_{\varphi}\iota_0(\a)$ for an order $\calO \subseteq K$ and an invertible fractional $\mathcal{O}$-ideal $\a$. In particular, $\mathfrak{O}_{\varphi} \cong_{\iota_0} \mathfrak{O}_{\varphi'}$ implies $\mathfrak{O}_{\varphi}\cap \iota_0(K)=\mathfrak{O}_{\varphi'}\cap \iota_0(K)$.
        \item[(c)] $(E, \iota) \cong (E', \iota')$ if and only if $\frakO_\varphi \cong_{\iota_0}\frakO_{\varphi'}$. In particular, $(E, \iota) \cong (E', \iota')$ doesn't necessarily imply $\frakO_{\varphi} = \frakO_{\varphi'}$ unless $\varphi \sim  \varphi'$.
        \item[(d)] If $\ker \varphi \subseteq \ker \varphi'$, then there exists a cyclic $K$-oriented isogeny $\psi: (E,\iota) \rightarrow (E', \iota')$ such that $\varphi'=\psi \varphi$.
    \end{enumerate}

\end{proposition}

\begin{proof}
 (a) Suppose $(E,\iota) \cong (E',\iota')$, $(E,\iota)$ is primitively $\calO$-oriented for an order $\calO$ in $K$, and $(E',\iota')$ is primitively $\calO'$-oriented for an order $\calO'$ in $K$. Since $(E,\iota) \cong (E',\iota')$, there exists an isomorphism $\lambda : (E, \iota) \rightarrow (E', \iota')$. Hence, 
 \begin{equation} \label{isomorphism of endomorphism rings}
\End(E)=\hat{\lambda}\End(E')\lambda=\kappa_\lambda(\End(E'))
 \end{equation}
and $\iota'=\lambda_\ast(\iota)$. It follows that
\begin{equation}
    \begin{split}
        \iota(\calO) &= \End(E) \cap \iota(K)\\
        & = \kappa_\lambda (\End(E') \cap \iota'(K)) \\
        & = \kappa_{\lambda}(\iota'(\calO'))\\
        & = \iota(\calO').
    \end{split}
\end{equation}
This implies $\calO=\calO'$ as $\iota$ is injective.

\begin{extra} Original proof implicitly assuming $\varphi'=\lambda \varphi:$
     (a) Suppose $(E,\iota) \cong (E',\iota')$, $(E,\iota)$ is primitively $\calO$-oriented for an order $\calO$ in $K$, and $(E',\iota')$ is primitively $\calO'$-oriented for an order $\calO'$ in $K$. Since $(E,\iota) \cong (E',\iota')$, there exists an isomorphism $\lambda : (E, \iota) \rightarrow (E', \iota')$ of $K$-oriented curves. Hence, 
 \begin{equation} \label{isomorphism of endomorphism rings}
\End(E)=\hat{\lambda}\End(E')\lambda=\kappa_\lambda(\End(E'))
 \end{equation}
and $\iota'=\lambda_\ast(\iota)$. Then 
\begin{equation} \label{same orientation}
\begin{alignedat}{2}
 \mathfrak{O}_\varphi \cap \iota_0(K) &= \iota_0(\calO) && \quad \because (\ref{orientation to optimal embedding})\\ 
 &= \kappa_\varphi(\iota(\calO)) && \quad \because (\ref{embedding single orientation})\\ 
&= \kappa_\varphi(\End(E)\cap \iota(K)) && \quad \because (E,\iota) \text{ is primitively }\calO\text{-oriented} \\ 
&= \kappa_{\varphi} (\kappa_\lambda(\End(E')\cap \iota'(K))) && \quad \because (\ref{isomorphism of endomorphism rings}) \\ 
&= \kappa_{\varphi'}(\End(E')\cap \iota'(K)) && \quad \because (\ref{induced homomorphism of quadratic fields}) \\
&= \kappa_{\varphi'}(\iota'(\calO')) && \quad \because (E',\iota') \text{ is primitively }\calO'\text{-oriented}\\
&= \iota_0(\calO') && \quad \because (\ref{embedding single orientation}) \\
& = \mathfrak{O}_{\varphi'} \cap \iota_0(K)  && \quad \because (\ref{orientation to optimal embedding})
\end{alignedat}
\end{equation}
where $\iota_0(\calO)=\iota_0(\calO')$ implies $\calO=\calO'$ as $\iota_0$ is injective.
\end{extra}

(b) The forward direction is \cite[Exercise 30.10]{Voi21}. 

\begin{details}
    \begin{tbox} 
    Also, \cite[Satz 7]{Eic55}
     according to Love and Boneh [Theorem 5.2, \href{https://arxiv.org/pdf/1910.03180}{arxiv}] (note that this is removed in the published version).
\end{tbox}
\end{details}

 Conversely, suppose $\mathfrak{O}_{\varphi'}=\iota_0(\a)^{-1}\mathfrak{O}_{\varphi}\iota_0(\a)$ for an order $\calO \subseteq K$ and an invertible fractional $\mathcal{O}$-ideal $\a$. Since $\iota_0(\a)$ commutes with elements in $\iota_0(K)$, we have 
\begin{equation}
   \mathfrak{O}_{\varphi'} \cap\iota_0(K)= \iota_0(\a)^{-1}\mathfrak{O}_{\varphi}\iota_0(\a) \cap \iota_0(K)=\mathfrak{O}_{\varphi}\cap \iota_0(K).
\end{equation}

\begin{details}
    \begin{tbox}
        If $y \in \mathfrak{O}_{\varphi} \cap \iota_0(K)$, then $\alpha^{-1}y\alpha =y$. Hence, $y \in \alpha^{-1}\mathfrak{O}_{\varphi}\alpha \cap \iota_0(K)$. This shows $ \mathfrak{O}_{\varphi} \cap \iota_0(K) \subseteq  \alpha^{-1}\mathfrak{O}_{\varphi}\alpha \cap \iota_0(K)$

        Conversely, let $x \in \alpha^{-1}\mathfrak{O}_{\varphi}\alpha \cap \iota_0(K)$. Then $x=\alpha^{-1}\beta \alpha$ for some $\beta \in \mathfrak{O}_{\varphi}$. Since $x, \alpha \in \iota_0(K)^\times$, we have $\beta=\alpha x \alpha^{-1}=x \in \iota_0(K)$. Hence, $x \in \mathfrak{O}_{\varphi} \cap \iota_0(K)$. Therefore, $\alpha^{-1}\mathfrak{O}_{\varphi}\alpha \cap \iota_0(K) \subseteq \mathfrak{O}_{\varphi}\cap \iota_0(K)$.
    \end{tbox}
\end{details}

Lastly, by definition, $\mathfrak{O}_{\varphi}\cong_{\iota_0}  \mathfrak{O}_{\varphi'}$ implies $\mathfrak{O}_{\varphi'}=\alpha^{-1}\mathfrak{O}_{\varphi}\alpha$ for some $\alpha \in \iota_0(K)^\times$. Hence,
\begin{equation}
    \mathfrak{O}_{\varphi'} \cap \iota_0(K)=\alpha^{-1}\mathfrak{O}_{\varphi}\alpha \cap \iota_0(K)=\mathfrak{O}_{\varphi} \cap \iota_0(K).
\end{equation}
 
(c) Suppose $(E,\iota)\cong (E',\iota')$. Then there exists a $K$-oriented isomorphism $\lambda : (E,\iota) \rightarrow (E',\iota')$. The composite $\varphi_1=\lambda \varphi$ is a cyclic $K$-oriented isogeny from $(E_0, \iota_0)$ to $(E',\iota')$. If $\varphi' \sim\varphi_1$, then $\mathfrak{O}_{\varphi}=\mathfrak{O}_{\varphi'}$ by Theorem \ref{bjection between curves, ideals, maximal orders} and we are done. Suppose $\varphi' \not\sim \varphi_1$. Then we have two non-equivalent $K$-oriented isogenies from $(E_0, \iota_0)$ to $(E',\iota')$ and they give a $K$-oriented endomorphism $\theta=\hat{\varphi_1} \varphi'$ of $(E_0, \iota_0)$. By definition, we have
\begin{equation}
    \iota_0=\theta_\ast(\iota_0).
\end{equation}
This implies
\begin{equation}
    \iota_0(K)\theta=\theta \iota_0(K),
\end{equation}
 i.e., $\theta \in \iota_0(K)^\times$. Let $n=\deg\varphi$ and $\alpha=\varphi_1^{-1}\varphi'=\theta/n \in \iota_0(K)^\times$. Then 
 \begin{equation}
    nI_{\varphi'}= I_{n \varphi'}=I_{\varphi_1\theta}=I_{\varphi_1}\theta=I_{\varphi}\theta,
 \end{equation}
where the third equality follows from Lemma \ref{composition by endomorphism}(c). Therefore, $I_{\varphi'}=I_{\varphi}\alpha$ and
\begin{equation}
\mathfrak{O}_{\varphi'}=\alpha^{-1}\mathfrak{O}_{\varphi}\alpha.
\end{equation}
\begin{extra}
    \begin{tbox}
        By (\ref{isomorphism of endomorphism rings}), 
\begin{equation}
    \mathfrak{O}_{\varphi} \cong \End(E) \cong \End(E') \cong \mathfrak{O}_{\varphi'}. 
\end{equation}
Hence, there exists $\beta \in B_0^\times$ such that
\begin{equation}
    \mathfrak{O}_{\varphi'}= \beta^{-1}\mathfrak{O}_{\varphi}\beta.
\end{equation}
By part (a), 
\begin{equation} \label{common quadratic intersection}
    \mathfrak{O}_{\varphi}\cap \iota_0(K)=\mathfrak{O}_{\varphi'}\cap \iota_0(K)= \beta^{-1}\mathfrak{O}_{\varphi}\beta \cap \iota_0(K).
\end{equation}
Let $\calO \subseteq K$ be an order such that $\iota_0(\calO)=\mathfrak{O}_{\varphi} \cap \iota_0(K)$. \eqref{common quadratic intersection} implies 
\begin{equation} \label{}
    \iota_{\beta^{-1}}(\calO)= \mathfrak{O}_{\varphi} \cap \iota_{\beta^{-1}}(K)
\end{equation}
(cf. \cite[30.3.7]{Voi21}), where $\iota_{\beta^{-1}}:K \rightarrow B_0$ is an embedding such that $\iota_{\beta^{-1}}(\alpha)=\beta\iota_0(\alpha)\beta^{-1}$.
    \end{tbox}
\end{extra}

Conversely, suppose $\mathfrak{O}_{\varphi}\cong_{\iota_0} \mathfrak{O}_{\varphi'}$. Then
\begin{equation}
\mathfrak{O}_{\varphi'}=\alpha^{-1}\mathfrak{O}_{\varphi}\alpha
\end{equation}
for some $\alpha \in \iota_0(K)^\times$. Scaling $\alpha$ by an integer, we may assume that $\alpha \in \mathfrak{O}_{\varphi}$ is primitive. Then
\begin{equation}
    I=\mathfrak{O}_{\varphi}\alpha
\end{equation}
is the primitive connecting $\mathfrak{O}_{\varphi}, \mathfrak{O}_{\varphi'}$-ideal by Lemma \ref{index for isomorphic maximal orders}(b). By part (b),
\begin{equation}
    \iota_0(\calO)=\mathfrak{O}_\varphi \cap \iota_0(K)=\mathfrak{O}_{\varphi'}\cap \iota_0(K)
\end{equation}
for an order $\calO$ in $K$. Hence, there exists a principal integral $\calO$-ideal $\a$ such that
\begin{equation}
    \iota_0(\a) = I \cap \iota_0(K)=\iota_0(\calO)\alpha.
\end{equation}
 Since $E[I] \subseteq E[\a]$ and $\# E[\a]=\#E[I]=\nrd(\alpha)$, we have $E[I]=E[\a]$. Then $(E,\iota)\simeq (E',\iota')$ since the class of principal $\calO$-ideals in $\cl(\calO)$ acts trivially on $\SSpr$.

(d) If $\ker \varphi \subseteq \ker \varphi'$, then there exists a cyclic isogeny $\psi: E \rightarrow E'$ such that $\varphi' = \psi \varphi$ by Proposition \ref{cyclic factor through}. Also,
\begin{equation} \label{single oriented isogeny converse1}
    \kappa_{\varphi'}(\iota'(\alpha))=\kappa_{\varphi}(\kappa_{\psi}(\iota'(\alpha)))
\end{equation}
for all $\alpha \in K$ by \eqref{induced homomorphism of quadratic fields}, and
\begin{equation} \label{single oriented isogeny converse2}
    \kappa_{\varphi}(\iota(\alpha))=\iota_0(\alpha)=\kappa_{\varphi'}(\iota'(\alpha))
\end{equation}
for all $\alpha \in K$ by \eqref{embedding single orientation}. Together with (\ref{single oriented isogeny converse1}) and (\ref{single oriented isogeny converse2}), we have
\begin{equation} \label{single oriented isogeny converse3}
    \kappa_{\varphi}(\iota(\alpha))= \kappa_{\varphi}(\kappa_{\psi}(\iota'(\alpha)))
\end{equation}
for all $\alpha \in K$. Since the map $\kappa_{\varphi}: \iota(K) \rightarrow \iota_0(K)$ is injective, (\ref{single oriented isogeny converse3}) implies 
\begin{equation}
    \iota(\alpha)=\kappa_{\psi}(\iota'(\alpha)))= \frac{1}{\deg \psi}\hat{\psi}\iota'(\alpha)\psi
\end{equation}
for all $\alpha \in K$. This is equivalent to
\begin{equation}
    \iota'(\alpha)= \frac{1}{\deg \psi}\psi\iota(\alpha)\hat{\psi}
\end{equation}
for all $\alpha \in K$, i.e., $\iota'=\psi_\ast(\iota)$. Therefore, $\psi: (E,\iota) \rightarrow (E',\iota')$ is a $K$-oriented isogeny.
\end{proof}

\begin{definition}
    We say two lattices $I, I'$ in $B_0$ are in the same left $\iota_0$-class, denoted by $I \sim_{\iota_0} I'$, if there exists $\alpha \in \iota_0(K)^\times$ such that $I'=I\alpha$.
\end{definition}

\begin{corollary} \label{bijections in single orientation}
    Let $\iota_0$ be a $K$-orientation on $E_0$. There are bijections
\begin{equation}
\begin{array}{c@{\;}c@{\;}c}
\left\{ \substack{\text{isomorphism classes of } K\text{-oriented}\\ \text{ supersingular elliptic curves over $\Fbar_p$}} \right\} 
& \longleftrightarrow & 
\left\{\substack{\iota_0\text{-conjugacy classes of}\\
\text{maximal orders in} \End^0(E_0)}  \right\} \\[1ex]
(E,\iota)/\cong
& \mapsto & \mathfrak{O}_{\varphi}/\cong_{\iota_0} \\[1ex]
& \longleftrightarrow & \{ \substack{ \text{left }\iota_0\text{-classes of primitive} \\ \text{left  }\mathfrak{O}_0\text{-ideals}}  \} \\[1ex]
& \mapsto & I_\varphi /\sim_{\iota_0}
\end{array}
\end{equation}    
where $\varphi: (E_0, \iota_0) \rightarrow (E, \iota)$ is a cyclic $K$-oriented isogeny.
\end{corollary}
\begin{proof}
    The first bijection is a direct result of Proposition \ref{single oriented isogeny}(c). 

   We show the second bijection; consider two $\iota_0$-isomorphic maximal orders $\mathfrak{O}_{\varphi}, \mathfrak{O}_{\varphi'}$ such that
    \begin{equation}
        \mathfrak{O}_{\varphi'}=\alpha^{-1}\mathfrak{O}_{\varphi}\alpha
    \end{equation}
    for $\alpha \in \iota_0(K)^\times$. Scaling $\alpha$ by an integer, we may assume that $\alpha \in \mathfrak{O}_{\varphi}$ is primitive. By Lemma \ref{index for isomorphic maximal orders}(b), $I=\mathfrak{O}_{\varphi}\alpha$ is the primitive connecting $\mathfrak{O}_{\varphi},\mathfrak{O}_{\varphi'}$-ideal. Also, by Proposition \ref{bijection of edges}, $I_{\varphi}^{-1}I_{\varphi'}=rI$ for some $r \in \Q^\times$. Hence,
    \begin{equation}
I_{\varphi'}=rI_{\varphi}I=rI_{\varphi}\mathfrak{O}_{\varphi}\alpha =I_{\varphi}(r\alpha)
    \end{equation}
    where $r\alpha \in \iota_0(K)^\times$.
    Therefore, $I_\varphi \sim_{\iota_0} I_{\varphi'}$.
    \begin{details}
    \begin{tbox}
      By definition, $\calO_{R}(I) & = \left\{ \alpha \in B_0 : I \alpha \subseteq I \right\}$, so $I\calO_R(I) \subseteq I$. Also, since $1 \in \calO_R(I)$, $I\calO_R(I) \supseteq I$.
    \end{tbox}
\end{details}
    Conversely, for two primitive left $\mathfrak{O}_0$-ideals $I_{\varphi}, I_{\varphi'}$ in the same left $\iota_0$-class, there exists $\alpha \in \iota_0(K)^\times$ such that $I_{\varphi'}=I_{\varphi}\alpha$. It follows that
    \begin{equation}
        \mathfrak{O}_{\varphi'}= \calO_R(I_{\varphi'})=\calO_R(I_{\varphi}\alpha)=\alpha^{-1}\calO_R(I_{\varphi})\alpha=\alpha^{-1}\mathfrak{O}_{\varphi}\alpha.
    \end{equation}
    Therefore, $\mathfrak{O}_{\varphi} \cong_{\iota_0}\mathfrak{O}_{\varphi'}$.
\end{proof}

Since an isomorphism class of $K$-oriented supersingular elliptic curves is represented by an $\iota_0$-conjugacy class of maximal orders in $\End^0(E_0)$, where the intersection of maximal orders in the $\iota_0$-conjugacy class with $\iota_0(K)$ is identical, we can use these maximal orders to determine whether a prime-degree $K$-oriented isogeny is ascending, horizontal, or descending. 

\begin{corollary}
Let $\varphi : (E_0, \iota_0) \rightarrow (E,\iota)$ and $\varphi' : (E_0, \iota_0) \rightarrow (E',\iota')$ be cyclic $K$-oriented isogenies. Suppose there is a cyclic $K$-oriented isogeny $\psi : (E, \iota) \rightarrow (E',\iota')$ of degree $\ell\neq p$ prime. If $\iota$ is a primitive $\calO$-orientation on $E$ for an order $\calO \subseteq K$ with $\iota(\calO)=\Z[\theta]$, then
\begin{enumerate}
    \item[(1)] $\psi$ is ascending if and only if $\theta/\ell \in \frakO_{\varphi'}$.
    \item[(2)] $\psi$ is horizontal if and only if $\theta \in \frakO_{\varphi'}$, but $\theta/\ell \not\in \frakO_{\varphi'}$.
    \item[(3)] $\psi$ is descending if and only if $\ell\theta \in \frakO_{\varphi'}$, but $\theta \not\in \frakO_{\varphi'}$.
\end{enumerate} 
\end{corollary}

 \begin{proof}
 Suppose $(E',\iota')$ is primitively $\calO'$-oriented for an order $\calO' \subseteq K$. Then
    \begin{equation}
        \iota_0(\calO')=\mathfrak{O}_{\varphi'}\cap \iota_0(K).
    \end{equation}
   Then the statements follow since
    \begin{enumerate}
        \item $\psi$ is ascending $\Longleftrightarrow$ $\calO \subsetneq \calO'$  $\Longleftrightarrow$ $\theta/\ell \in \calO'$.
        \item $\psi$ is horizontal  $\Longleftrightarrow$ $\calO=\calO'$  $\Longleftrightarrow$ $\theta \in \calO'$, but $\theta/\ell \not\in \calO'$.
        \item $\psi$ is descending  $\Longleftrightarrow$ $\calO\supsetneq \calO'$  $\Longleftrightarrow$ $\ell\theta \in \calO'$, but $\theta \notin \calO'$.
    \end{enumerate}
\end{proof}

\begin{extra}
We remark that the notion of embeddings of imaginary quadratic orders into maximal orders in a quaternion algebra which are equivalent to orientations of supersingular elliptic curves is not new and has a long history in the study of quaternion algebra (cf. \cite[Remark 30.6.18]{Voi21}). Such embeddings are referred to as optimal embeddings. We recall the basic definitions and a property of optimal embeddings; for a more detailed treatment, we refer the reader to \cite[\S30–\S31]{Voi21}.

Let $B_0/\Q$ be a quaternion algebra ramified at $p$ and $\infty$. Consider an imaginary quadratic field $K$ that embeds into $B_0$. Let $\calO \subseteq K$ be a quadratic order and $\mathfrak{O} \subseteq B_0$ be a quaternion order. Given a $\Z$-algebra embedding $\iota_0: \calO \rightarrow \mathfrak{O}$, we can $\Q$-linearly expand it and obtain a $\Q$-algebra embedding $\iota_0:K \rightarrow B_0$. Then the embedding $\iota_0: \calO \rightarrow \mathfrak{O}$ is called optimal if
\begin{equation}
    \iota_0(\calO)=\mathfrak{O} \cap \iota_0(K).
\end{equation}
Let
\begin{equation}
    \operatorname{Emd(\calO, \mathfrak{O})} = \left\{ \iota_0 : \calO \hookrightarrow \mathfrak{O} \text{ is an optimal embedding} \right\}.
\end{equation}

Let $\iota_0 :K \rightarrow B_0$ be an embedding and consider
\begin{align}
    \operatorname{Conj}(\calO, \mathfrak{O}, \iota_0) & :=\{\beta \in B_0^\times: \iota_0(\calO)=(\beta^{-1} \mathfrak{O} \beta) \cap \iota_0(K)\} \\
    & =\{\beta \in B_0^\times: \beta \iota_0(\calO) \beta^{-1} = \mathfrak{O} \cap \beta \iota_0(K) \beta^{-1} \}.
\end{align}

\begin{lemma} \label{Voi21 Lemma 30.3.8}
    The map
\begin{equation}
   \begin{split}
        \iota_0(K)^\times \backslash \operatorname{Conj}(\calO, \mathfrak{O}, \iota_0) & \overset{\sim}{\rightarrow} \operatorname{Emd}(\calO, \mathfrak{O}) \\
        \beta & \mapsto \iota_{\beta}
   \end{split}
\end{equation}
where $\iota_{\beta}(\alpha):=\beta^{-1}\alpha \beta$ for $\alpha \in \calO$ is a bijection.
\end{lemma}
\begin{proof}
   This is \cite[Lemma 30.3.8]{Voi21}.  
\end{proof}

We now introduce some related concepts which appear in the trace formula for Hecke operators acting on spaces of modular forms \cite[\S 6.5]{Miyake}. Given a maximal order $\mathfrak{O}$ in $B_0$ and an element $\alpha \in B_0$ not in $\Q$, define
\begin{align}
   C(\alpha) & = \left\{ \beta \alpha \beta^{-1} : \beta \in B_0^\times \right\} \\
   C(\alpha, \calO) & = \left\{ \beta \alpha \beta^{-1} : \beta \in B_0^\times, \Q[\alpha] \cap \beta^{-1} \mathfrak{O} \beta = \calO \right\}.
\end{align}

We have that
\begin{equation}
    C(\alpha) = \coprod C(\alpha,\calO).
\end{equation}

According to the results in \cite[\S 6.5]{Miyake}, if $\calO$ is $p$-fundamental, then $C(\alpha,\calO) \not= \emptyset$ and
\begin{equation}
    |C(\alpha,\calO)| = h(\calO).
\end{equation}

Hence, for any $\calO$ which is $p$-fundamental, there is a conjugate $\mathfrak{O}' = \beta^{-1} \mathfrak{O} \beta$ such that
\begin{equation}
    \mathfrak{O}' \cap \Q[\alpha] = \calO.
\end{equation}

\todo[inline]{Proved earlier but maybe there is a different proof. This section is related to the trace formula for Brandt matrices given in Voight or Goren-Love.}

\begin{proposition} \label{optimal embedding and conjugation by quadratic field} Let $\iota_0 : K \rightarrow B_0$ be an embedding and let $\mathfrak{O}, \mathfrak{O}'$ be maximal orders in $B_0$. Then 
    $\frakO \cap \iota_0(K)=\frakO' \cap \iota_0(K)$ implies $\mathfrak{O}'=\a^{-1}\mathfrak{O}\a$ by an $\alpha \in \iota_0(K)^\times$.
\end{proposition}
\begin{proof}
Suppose $\frakO \cap \iota_0(K) = \frakO' \cap \iota_0(K) := \calO$. This is equivalent to say we have underlying optimal embeddings $\iota: \calO \hookrightarrow \mathfrak{O}$ and $\iota: \calO \hookrightarrow \mathfrak{O}'$ for the same order $\calO \subseteq K$, i.e., $\iota$ lies in $\operatorname{Emd}(\calO, \mathfrak{O}) \cap\operatorname{Emd}(\calO, \mathfrak{O}')$. By Lemma~\ref{Voi21 Lemma 30.3.8}, there exists $\beta \in \operatorname{Conj}(\calO, \mathfrak{O}, \iota_0)$ such that 
        \begin{equation}
            \iota(\alpha) = \beta^{-1}\iota_0(\alpha)\beta
        \end{equation}
        for all $\alpha \in \calO$. 
        that is, $\mathfrak{O}'=\beta^{-1}\mathfrak{O}\beta$. Since $\iota, \iota'$ agree on $\calO$, and $K$ is commutative, conjugation by an element of $K^\times$ fixes the elements of $K$, i.e.,
        \begin{equation}
            \gamma^{-1} \alpha \gamma = \alpha
        \end{equation}
    for $\gamma \in K^\times$ and $\alpha \in K$. It follows that $\beta$ must satisfy
\begin{equation}
    \beta^{-1}\alpha \beta = \alpha
\end{equation}
for all $\alpha \in \calO$. This implies that $\beta$ is in the centralizer $\operatorname{Cent}_{B_0}(\calO)$ of $\iota(\calO)$ in $B_0$, which is $\iota(K)$. Hence,
    \begin{equation}
        \beta \in B_0^\times \cap \operatorname{Cent}_{B_0}(\calO)=\iota(K)^\times.
    \end{equation}
\end{proof}
\end{extra}

\subsection{Double-orientations with maximal orders}
Since the endomorphism rings of supersingular elliptic curves over $\Fbar_p$ are isomorphic to maximal orders in the quaternion algebra $B_0$, it is natural to consider orientations by quaternion algebras. Whereas the usual single-orientations stratify the vertices in the isogeny graphs by quadratic orders, such orientations stratify the vertices by quaternion orders, and they are equivalent to specifying a pair of single-orientations; we refer to them as double-orientations. 
 
Let $K_1=\Q(\sqrt{d_1})$ and $K_2=\Q(\sqrt{d_2})$ be imaginary quadratic fields, where $d_1, d_2<0$ are square-free integers. Suppose that $K_1$ and $K_2$ have simultaneous embeddings into $B_0$ and their images generate $B_0$. Then there exist $\beta_1, \beta_2 \in B_0$ such that
\begin{equation}
    \beta_1^2 = d_1, \beta_2^2 =d_2,  \beta_1\beta_2+\beta_2\beta_1=s \in \Q,
\end{equation}
and the quadratic subalgebras $\Q(\beta_1)$ and $\Q(\beta_2)$ generate $B_0$. In particular, $\{1, \beta_1, \beta_2, \beta_1\beta_2\}$ form a basis for $B_0$. The conditions for such fields are given as follows.
\begin{proposition} \label{condition for simultaneous embeddings - fields}
    $K_1$ and $K_2$ have simultaneous embeddings into $B_0$ and their images generate $B_0$ if and only if there is $s \in \Q$ such that $s^2<4d_1d_2$ and $d_1, d_2$ satisfy the following condition on the Hilbert symbol at prime $\ell$:
    \begin{equation}
        (d_1, s^2-4d_1d_2)_{\ell}=\begin{cases}
            -1 & \text{if } \ell=p \\
            1 & \text{otherwise.}
        \end{cases}
    \end{equation}
\end{proposition}
\begin{proof}
    This is \cite[Proposition 1(a)]{BE92}.
\end{proof}

We define a $K_1, K_2$-orientation on a supersingular elliptic curve $E$ to be simultaneous embeddings $\jmath: K_1, K_2 \rightarrow \End^0(E)$ such that images of $K_1, K_2$ generate $\End^0(E)$. A $K_1, K_2$-orientation $\jmath$ on $E$ induces a pair of $K_1$-orientation and $K_2$-orientation on $E$ obtained by restricting $\jmath$ to $K_1$ and $K_2$ respectively. Accordingly, we will denote by $\jmath=(\iota_1, \iota_2)$ if $\iota_1=\left.\jmath\right|_{K_1}$ and $\iota_2=\left.\jmath\right|_{K_2}$. We will call a pair $(E, \jmath)$ of a supersingular elliptic curve $E$ and a $K_1, K_2$-orientation $\jmath$ on $E$ a $K_1, K_2$-oriented curve. 

\begin{details}
    \begin{tbox}
        In literature, single orientations are defined for any elliptic curves over finite fields. For double-orientations, by definition, we can only consider supersingular elliptic curves. 
    \end{tbox}
\end{details}

Given an isogeny $\varphi: E \rightarrow E'$ and a $K_1, K_2$-orientation $\jmath$ on $E$, there is an induced $K_1, K_2$-orientation on $E'$ given by
\begin{equation}
    \varphi_\ast(\jmath):=(\varphi_\ast(\iota_1), \varphi_\ast(\iota_2)).
\end{equation}
We will call $\varphi: (E, \jmath) \rightarrow (E', \jmath')$ a $K_1, K_2$-oriented isogeny if $\jmath'=\varphi_\ast(\jmath)$. A $K_1, K_2$-oriented isogeny is called a $K_1, K_2$-isomorphism (resp. $K_1, K_2$-automorphism) if it is an isomorphism (resp. automorphism) of the underlying curves. Two $K_1, K_2$-oriented isogenies $\varphi: (E, \jmath) \rightarrow (E', \jmath'), \psi: (E, \jmath) \rightarrow (E'', \jmath'')$ are called equivalent if the underlying isogenies are equivalent.

Let $\mathcal{O}_{K_1}$ and $\mathcal{O}_{K_2}$ be maximal orders of $K_1$ and $K_2$ respectively. 

\begin{proposition} \label{condition for simultaneous embeddings - maximal orders}
    The maximal orders $\mathcal{O}_{K_1}$ and $\mathcal{O}_{K_2}$ have simultaneous embeddings into $B_0$ and their images generate an order in $B_0$ if and only if there is $s \in \Z$ such that in addition to the conditions on $s$ in Proposition \ref{condition for simultaneous embeddings - fields}, it satisfies $s \equiv 2 \pmod{\delta}$ where
    \begin{equation}
        \delta=\begin{cases}
            1 & \text{if } d_1, d_2 \not\equiv 1 \pmod{4}, \\
            4 & \text{if } d_1, d_2 \equiv 1  \pmod{4},\\
            2 & \text{otherwise}. 
            \text{}
        \end{cases}
    \end{equation}
\end{proposition}
\begin{proof}
   This is \cite[Proposition 1(b)]{BE92}. 
\end{proof}

For the rest of the paper, we will assume that $K_1, K_2$ satisfy the conditions in Proposition \ref{condition for simultaneous embeddings - fields} and Proposition \ref{condition for simultaneous embeddings - maximal orders}. Let $\calO_1 \subseteq K_1$ and $\calO_2 \subseteq K_2$ be orders such that $\iota_1$ is a primitive $\calO_1$-orientation on $E$ and $\iota_2$ is a primitive $\calO_2$-orientation on $E$. Then we call $\jmath=(\iota_1, \iota_2)$ a primitive $\mathcal{O}_1, \mathcal{O}_2$-orientation on $E$. 

\begin{remark}
 Since the images of $\mathcal{O}_{K_1}, \mathcal{O}_{K_2}$ generate an order in $B_0$ and 
 \begin{align}
     \calO_1&=\Z+f_1\mathcal{O}_{K_1},\\
     \calO_2&=\Z+f_2\mathcal{O}_{K_2}
 \end{align}
 for some $f_1, f_2 \in \Z$, the images of $\calO_1, \calO_2$ generate an order in $B_0$ by \cite[Exercise 10.5]{Voi21}.
\end{remark}

Note that we can view $K_1$ and $K_2$ as quadratic $\Q$-subalgebras of an abstract quaternion algebra $B_0$ by considering a pair of fixed embeddings of $K_1$ and $K_2$ into $B_0$ and identify them with their images. In this way, we can alternatively view a $K_1, K_2$-orientation $\jmath$ on $E$ as an isomorphism of $\Q$-algebras $\jmath: B_0 \rightarrow
\End^0(E)$. 

Now, we study the properties of double-orientations.

Consider a $K_1, K_2$-orientation $\jmath_0$ on $E_0$. Each pair $(E, \varphi)$ in $\SS(E_0)$ admits an induced $K_1,K_2$-orientation $\jmath=\varphi_\ast(\jmath_0)$, and $\varphi:(E, \jmath_0) \rightarrow (E,\jmath)$ is a $K_1, K_2$-oriented isogeny. We say a $K_1, K_2$-oriented isogeny is cyclic if the underlying isogeny is cyclic. By Remark \ref{unique cyclic isogeny}, there is a cyclic $K_1, K_2$-oriented isogeny $\varphi' : (E_0, \jmath_0) \rightarrow (E,\jmath)$ and $\varphi \sim \varphi'[n]$ for some integer $n \geq 1$. 

By definition, $\jmath_0$ factors through the isomorphism of $\Q$-algebras $\kappa_\varphi:\End^0(E) \rightarrow \End^0(E_0)$ as follows:
\[
\begin{tikzcd}
B_0 \arrow[r, "\jmath_0"] \arrow[dr, swap, "\jmath"] & \End^0(E_0)  \\
& \End^0(E) \arrow[u, "\kappa_\varphi"']
\end{tikzcd}
\]
\begin{equation} \label{embedding double orientation}
    \kappa_{\varphi}(\jmath(\beta)) = \jmath_0(\beta) \ \forall \beta \in B_0.
\end{equation}

Let $\varphi: (E_0, \jmath_0) \rightarrow (E,\jmath)$, $\varphi': (E_0, \jmath_0) \rightarrow (E', \jmath')$, and $\psi : (E, \jmath) \rightarrow (E', \jmath')$ be $K_1,K_2$-oriented isogenies such that their $\varphi'=\psi\varphi$.
\[
\begin{tikzcd}
(E_0, \jmath_0) \arrow[r, "\varphi" above] \arrow[rr, bend right=20, "\varphi'" below] 
    & (E, \jmath) \arrow[r, "\psi" above] 
    & (E', \jmath')
\end{tikzcd}
\]
Then we have the following commutative diagram of isomorphisms of $\Q$-algebras
 \[
\begin{tikzcd}
B_0 \arrow[r, "\jmath_0"] \arrow[dr, swap, "\jmath"] \arrow[ddr, "\jmath'"'] & \End^0(E_0)  \\
& \End^0(E) \arrow[u, "\kappa_\varphi"'] \\
& \End^0(E') \arrow[u, "\kappa_\psi"'] \arrow[uu, bend right=50, "\kappa_{\varphi'}"']
\end{tikzcd}
\]
with
\begin{equation} \label{induced homomorphism of quaternion algebras}
    \kappa_{\varphi'}(\End(E'))=\kappa_\varphi(\kappa_\psi(\End(E'))).
\end{equation}

\begin{proposition} \label{single and double conjugate}
    Let $\mathfrak{O}, \mathfrak{O}'$ be maximal orders in $B_0$. Let $\iota_1, \iota_2$ be embeddings of $K_1, K_2$ into $B_0$ such that their images generate $B_0$ and the images of their maximal orders generate an order in $B_0$. Then $\mathfrak{O}= \mathfrak{O}'$ if and only if $\mathfrak{O}\cong_{\iota_1}\mathfrak{O}'$ and $\mathfrak{O}\cong_{\iota_2}\mathfrak{O}'$.
\end{proposition} 
\begin{proof}
    The forward direction is trivial. For the other direction, suppose
    \begin{equation} 
\mathfrak{O}'=\alpha^{-1}\mathfrak{O}\alpha=\beta^{-1}\mathfrak{O}\beta
    \end{equation}
    for some $\alpha \in \iota_1(K_1)^\times$ and $\beta \in \iota_2(K_2)^\times$. This implies $\alpha\beta^{-1}$ is in the centralizer of $\mathfrak{O}$ which is contained the centralizer $\Q$ of $B_0$. Hence, $\alpha=r\beta \in \iota_2(K)$ for some $r \in \Q$. Since $\iota_1(K_1) \cap \iota_2(K_2)=\Q$, we must have $\alpha, \beta \in \Q$, i.e., $\mathfrak{O}'=\mathfrak{O}$.
\end{proof}

\begin{details}
    \begin{tbox}
    $\alpha, \beta \in B_0$ commute if and only if their pure quaternion parts are scalar multiples of each other
\end{tbox}
\end{details}

\begin{proposition} \label{double orientaion and maximal order}
    Let $\varphi: (E_0, \jmath_0) \rightarrow (E, \jmath)$ and $\varphi': (E_0, \jmath_0) \rightarrow (E', \jmath')$ be cyclic $K_1,K_2$-oriented isogenies. The following hold:
    \begin{enumerate}
        \item[(a)] $(E,\jmath) \cong (E', \jmath')$ if and only if $\frakO_{\varphi}=\frakO_{\varphi'}$.
        \item[(b)] If $\ker\varphi \subseteq \ker\varphi'$, then there exists a cyclic $K_1, K_2$-oriented isogeny $\psi: (E, \jmath)\rightarrow (E', \jmath')$ such that $\varphi'=\psi\varphi$.
    \end{enumerate}
\end{proposition}
\begin{proof}
    Let $\jmath_0=(\iota_{0,1}, \iota_{0,2})$, $\jmath=(\iota_1, \iota_2)$, and $\jmath'=(\iota_1',\iota_2')$.

   (a) By Proposition \ref{single oriented isogeny}(c) and Proposition \ref{single and double conjugate}, $(E,\jmath) \cong (E',\jmath')$ if and only if $(E,\iota_1) \cong (E', \iota_1')$ and $(E,\iota_2) \cong (E', \iota_2')$ if and only if $\mathfrak{O}_{\varphi} \cong_{\iota_{0,1}}\mathfrak{O}_{\varphi'}$ and $\mathfrak{O}_{\varphi} \cong_{\iota_{0,2}}\mathfrak{O}_{\varphi'}$ if and only if $\mathfrak{O}_{\varphi}=\mathfrak{O}_{\varphi'}$. 

    (b) If $\ker \varphi \subseteq \ker \varphi'$, then there exists a cyclic isogeny $\psi: E \rightarrow E'$ such that $\varphi'=\psi\varphi$ by Proposition \ref{cyclic factor through}. By Proposition \ref{single oriented isogeny}(d), $\iota_1'=\psi_{\ast}(\iota_1)$ and $\iota_2'=\psi_{\ast}(\iota_2)$. Hence, $\jmath'=\psi_{\ast}(\jmath)$.
\end{proof}

Suppose we endow the initial curve $E_0$ with a $K_1,K_2$-orientation $\jmath_0$ on $E_0$.
\begin{corollary}  \label{marked curves to oriented curves}
    There is a bijection
     \begin{equation}
\begin{array}{c@{\;}c@{\;}c}
\SS^\ast(E_0)
& \longleftrightarrow & 
\{ \substack{\text{isomorphism classes of } K_1,K_2\text{-oriented} \\ \text{supersingular  elliptic curves over } \Fbar_p}\} \\[1ex]
(E,\varphi)/\sim 
& \mapsto & 
(E, \jmath)/\cong
\end{array}
\end{equation}
where $\jmath=\varphi_\ast(\jmath_0)$.
\end{corollary}

\begin{corollary} \label{double oriented curves to maximal orders}
There are bijections
\begin{equation}
\begin{array}{c@{\;}c@{\;}c}
\{ \substack{\text{isomorphism classes of } K_1,K_2\text{-oriented} \\ \text{supersingular  elliptic curves over } \Fbar_p}\} 
& \longleftrightarrow & 
\left\{\text{maximal orders in} \End^0(E_0)  \right\} \\[1ex]
(E,\jmath)/\cong
& \mapsto & \mathfrak{O}_{\varphi}\\[1ex]
& \longleftrightarrow & \{\text{primitive left integral} \ \mathfrak{O}_0\text{-ideals}\} \\[1ex]
& \mapsto & I_\varphi
\end{array}
\end{equation}    
where $\varphi: (E_0, \jmath_0) \rightarrow (E, \jmath)$ is a cyclic $K_1,K_2$-oriented isogeny.
\end{corollary}

\begin{corollary}\label{bijection of edges with orientations} 
There is a bijection
 \begin{equation} 
\begin{array}{c@{\;}c@{\;}c}
\left\{ \substack{\text{equivalence classes of cyclic $K_1,K_2$-oriented isogenies of} \\ K_1,K_2\text{-oriented supersingular elliptic curves over } \Fbar_p} \right\} 
& \longleftrightarrow & 
\left\{ \substack{\text{primitive connecting ideals between} \\ \text{maximal orders in } \End^0(E_0) }\right\} \\[1ex]
\psi:(E,\jmath) \rightarrow (E',\jmath') /\sim
& \mapsto & 
\operatorname{pr}(I_{\varphi}^{-1}I_{\varphi}),
\end{array}
\end{equation}
where $\varphi : (E_0, \jmath_0) \rightarrow (E, \jmath)$ and $\varphi': (E_0, \jmath_0) \rightarrow (E',\jmath')$ are cyclic $K_1, K_2$-oriented isogenies.
\end{corollary}
\begin{proof}
    This follows from Proposition \ref{bijection of edges} as a cyclic $K_1, K_2$-oriented isogenies naturally arise from compatible isogenes; given a cyclic compatible isogeny $\psi: (E, \varphi) \rightarrow (E', \varphi')$, we have induced $K_1, K_2$-orientations $\jmath=\varphi_\ast(\jmath_0)$ and $ \jmath'=\varphi'_\ast(\jmath_0)$, where
    \begin{equation}  \psi_\ast(\jmath)=\psi_\ast(\varphi_\ast(\jmath_0))=\varphi'_\ast(\jmath_0)=\jmath'
    \end{equation}
    as we have $m\varphi' =n\psi \varphi$ for some non-zero integers $m,n$. Hence, $\psi:(E,\jmath) \rightarrow (E', \jmath')$ is a cyclic $K_1, K_2$-oriented isogeny.
\end{proof}

\begin{details}
\begin{tbox}
    $m\varphi' =n\psi \varphi$ implies $m\hat{\varphi'}=n\hat{\varphi}\hat{\psi}$ and $m^2\deg\varphi'=n^2\deg\psi \deg\varphi$. It follows that
    \begin{equation}
\psi_\ast(\iota)=\psi_\ast(\varphi_\ast(\iota_0))=\frac{1}{\deg\psi \deg\varphi}\psi\varphi \iota_0 \hat{\varphi}\hat{\psi}=\frac{1}{n^2\deg\varphi \deg\psi}(n\psi\varphi) \iota_0 (n \hat{\varphi}\hat{\psi})=\frac{1}{\deg\varphi'}\varphi' \iota_0 \hat{\psi}'
    \end{equation}
    for $K$-orientations.
\end{tbox}
\end{details}

\begin{lemma} \label{no loop maximal order}
    Let $\ell \neq p$ be a prime and $\mathfrak{O}$ be a maximal order in $B_0$. There is no integral two-sided $\mathfrak{O}$-ideal of reduced norm $\ell$.
\end{lemma}
\begin{proof}
  By \cite[23.2.3]{Voi21}, every two-sided ideals of $\mathfrak{O}_\ell$ are powers of $\ell \mathfrak{O}_\ell$. It follows that for every integral two-sided $\mathfrak{O}$-ideal $I$, $\ell^2 \mid \nrd(I)$.
\end{proof}

\begin{corollary} \label{no loop curve}
    Let $\ell \neq p$ be a prime and $(E, \jmath) \cong (E',\jmath')$ be $K_1, K_2$-oriented supersingular elliptic curves over $\Fbar_p$. There is no $K_1,K_2$-oriented isogeny $\psi: (E, \jmath) \rightarrow (E',\jmath')$ of degree $\ell$.
\end{corollary}
\begin{proof}
    Let $\varphi: (E_0, \jmath_0) \rightarrow (E,\jmath)$ be a cyclic $K_1, K_2$-oriented isogeny. Suppose that there is a $K_1, K_2$-oriented isogeny $\psi : (E,j)\rightarrow (E',\jmath')$ of degree $\ell$. Since $(E,\jmath) \cong (E',\jmath')$, there is an isomorphism $\lambda: (E,\jmath) \rightarrow (E,\jmath')$ of $K_1, K_2$-oriented curves, and post-composition by $\lambda$ gives an endomorphism $\epsilon : (E,\jmath) \rightarrow (E,\jmath)$ of $K_1, K_2$-oriented curves. By Proposition \ref{double orientaion and maximal order}(a), $\mathfrak{O}_{\varphi}=\mathfrak{O}_{\epsilon\varphi}$. By the proof of Lemma \ref{composition by endomorphism}(a), there is an integral two-sided $\mathfrak{O}_{\varphi}$-ideal of reduced norm equal to $\deg \epsilon =\ell$, which is a contradiction by Lemma \ref{no loop maximal order}. 
\end{proof}

\section{Zoology of supersingular isogeny graphs and quaternion ideal graphs} \label{Supersingular isogeny graphs and quaternion ideal graphs} In this section, we give a complete analogy of variants of the supersingular $\ell$-isogeny graph and the corresponding variants of the quaternion $\ell$-ideal graph, and show that there exist graph isomorphisms between them.

\subsection{Double-oriented isogeny graphs} We define double-oriented isogeny graphs analogously to single-oriented isogeny graphs.

\begin{definition}
    The $K_1, K_2$-oriented supersingular $\ell$-isogeny graph $G_{K_1, K_2}(p, \ell)$ in characteristic $p$ is the directed multigraph whose vertices are isomorphism classes of $K_1, K_2$-oriented supersingular elliptic curves over $\Fbar_p$ and edges are the equivalence classes of $K_1, K_2$-oriented isogenies between them.
\end{definition}

The graph $G_{K_1, K_2}(p, \ell)$ has $\ell+1$ outgoing edges from each vertex. Given a double-oriented isogeny, we have two underlying single-oriented isogenies. Hence, we give the following definitions.

\begin{definition}
Let $\varphi : (E, \jmath) \rightarrow (E',\jmath')$ be a $K_1, K_2$-oriented isogeny, where $\jmath=(\iota_1, \iota_2)$ and $\jmath'=(\iota_1', \iota_2')$. We say $\varphi$ is $K_1$-ascending, $K_1$-horizontal, or $K_1$-descending if the underlying $K_1$-oriented isogeny $\varphi:(E,\iota_1) \rightarrow (E,\iota_1')$ is ascending, horizontal, or descending respectively, and similarly for $K_2$. We say that $\varphi$ is simultaneously ascending, horizontal, or descending, if both $\iota_1$ and $\iota_2$ have these properties, respectively. Finally, we say $\varphi$ is mixed property $1$-property $2$ if $\iota_1$ has property $1$ and $\iota_2$ has property $2$ and property $1$ is different from property $2$.
\end{definition}

Each vertex in $G_{K_1, K_2}(p,\ell)$ can be represented by a maximal order in the endomorphism algebra of a chosen supersingular ellipitc curve $E_0/\Fbar_p$. Such a representation is instrumental in analyzing the structure of $G_{K_1, K_2}(p,\ell)$ through the following notion, which is analogous to the rim of the single-oriented isogeny graph.
\begin{definition}
Given a vertex $(E,\jmath)/\cong$ in $G_{K_1, K_2}(p, \ell)$, we will call it a local root of $G_{K_1, K_2}(p, \ell)$ if $\jmath$ is a primitive $\calO_1,\calO_2$-orientation on $E$ for $\ell$-fundamental orders $\calO_1 \subseteq K_1$ and $\calO_2 \subseteq K_2$. It is local in the sense that we cannot further ascend, i.e.\ there is no $K_1,K_2$-oriented isogeny from $(E,\jmath)$ that is either $K_1$-ascending or $K_2$-ascending in its connected component in $G_{K_1,K_2}(p,\ell)$.
\end{definition}

Observe that for any prime $\ell' \neq p$, $G_{K_1, K_2}(p,\ell)$ and $G_{K_1, K_2}(p,\ell')$ share the same vertex set with variation only in their edge structures. 

\begin{definition} We will call $(E,\jmath)/\cong$ a global root if $\jmath$ is a primitive $\calO_{K_1}, \calO_{K_2}$-orientation on $E$. It is global in the sense that it is a local root for any prime $\ell\neq p$. 
\end{definition}

\subsection{Reduced isogeny graphs}

The $j$-invariant of a supersingular elliptic curve $E/\overline{\F}_p$ lies in $\F_{p^2}$ and we refer to such an element $j \in \F_{p^2}$ as a supersingular $j$-invariant. If $j$ is a supersingular $j$-invariant, then its $\F_{p^2}$-conjugate $j^p$ is a supersingular $j$-invariant as well. 

Moreover, for supersingular $j$-invariants $j, j' \in \mathbb{F}_{p^2}$, we have that $\Phi_{\ell}(j, j')=0$ if and only if $\Phi_{\ell}(j^p, j'^p)=0$. Since the modular polynomial $\Phi_\ell(X,Y) \in \Z[X,Y]$, there is a bijection between the roots of $\Phi_{\ell}(j, Y)$ and the roots of $\Phi_{\ell}(j^p, Y)$.


Let $[j]$ denote the Galois orbit $\{j, j^p\}$ of a supersingular $j$-invariant $j$. We now take the graph in Definition~\ref{defn-j-graph} and identify conjugate vertices as well as conjugate edges.

\begin{definition}
    The reduced supersingular $\ell$-isogeny graph $\tilde{G}(p,\ell)$ in characteristic $p$ is a directed multigraph whose vertices are supersingular $j$-invariants in $\F_{p^2}$ up to Galois conjugation $\Gal(\mathbb{F}_{p^2}/\mathbb{F}_p)$ and there is an edge from $[j]$ to $[j']$ for each root $j'$ of $\Phi_{\ell}(j, X)$.
\end{definition}

\begin{details}
\begin{tbox}
    \begin{proposition} \label{isomorphism of curves to conjugacy class of maximal orders}
      There is a bijection
     \begin{equation}
\begin{array}{c@{\;}c@{\;}c}
\{\substack{\text{isomorphism clases of supersingular} \\ \text{elliptic curves over $\Fbar_p$}}\}
& \longleftrightarrow & 
\Cls_L(\mathfrak{O}_0) \\[1ex]
E/\cong
& \mapsto & 
I/\sim
\end{array}
\end{equation}
where $\End(E) \cong \calO_R(I)$. Also, there is a surjection
  \begin{equation}
\begin{array}{c@{\;}c@{\;}c}
\Cls_L(\mathfrak{O}_0)
& \rightarrow & \typee(\mathfrak{O}_0) \\[1ex]
I/\sim
& \mapsto & 
\calO_R(I)/\cong
\end{array}
\end{equation}
where the fiber $\{ I/\sim \ \in \Cls(\mathfrak{O}) : \calO_L(I) \cong \mathfrak{O}\}$ of this map over $\mathfrak{O}/\cong$ in $\typee(\mathfrak{O}_0)$ is in bijection with $\Idl(\mathfrak{O})/\PIdl(\mathfrak{O})$ given by
 \begin{equation}
\begin{array}{c@{\;}c@{\;}c}
\Idl(\mathfrak{O})/\PIdl(\mathfrak{O})
& \longleftrightarrow & 
\{ I/\sim \ \in \Cls_L(\mathfrak{O}) : \calO_R(I) \cong \mathfrak{O}\} \\[1ex]
I/\PIdl
& \mapsto & 
I/\sim.
\end{array}
\end{equation}
Furthermore, abelian group
\begin{equation}
    \Pic(\mathfrak{O}) \cong  \Z/2\Z
\end{equation}
is generated by the unique maximal two-sided ideal $\mathfrak{P} \subseteq \mathfrak{O}$ of reduced norm $p$ and there is a surjection
\begin{equation}
 \Pic(\mathfrak{O}) \rightarrow  \Idl(\mathfrak{O})/\PIdl(\mathfrak{O})
\end{equation}
where the class of $\mathfrak{P}$ in the quotient is trivial if and only if $\mathfrak{P}$ is principal.
 \cite[Theorem 18.1.3]{Voi21}
\end{proposition}
\begin{proof}
    This is \cite[Corollary 42.3.7]{Voi21}, \cite[Lemma 17.4.13]{Voi21}, and  \cite[Proposition 18.5.10]{Voi21}
\end{proof}
\end{tbox}
\end{details}

\subsection{Quaternion ideal graphs}
Let $B_0/\Q$ be a quaternion algebra ramified at $p$ and $\infty$ and $\mathfrak{O}_0$ be a maximal order in $B_0$.  

Let $I, I'$ be invertible left $\mathfrak{O}_0$-ideals. We say $I$ is a $\ell$-neighbor of $I'$ if $I \subseteq I'$ and $\ell \nrd(I')=\nrd(I)$. By \cite[Exercise 17.5]{Voi21}, we may assume $I \subseteq I'$ are integral and their reduced norms are coprime to $p$; suppose $I_1'=I'\alpha, \alpha \in B_0^\times$ is integral and $\nrd(I_1')$ is coprime to $p$. Then $I_1=I\alpha$ is integral with $\nrd(I_1)$ coprime to $p$ and it is a $\ell$-neighbor of $I_1'$. We may further assume that $I, I'$ are primitive; if $n\geq 1$ is the largest integer such that $I \subseteq n\mathfrak{O}_0$, then $n^{-1}I \subseteq \mathfrak{O}_0$ is primitive. Assume for contradiction that $n^{-1}I' \subseteq m \mathfrak{O}_0$ for some integer $m>1$. Then $I \subseteq I' \subseteq nm\mathfrak{O}_0$, so $n^{-1}I \subseteq m\mathfrak{O}_0$. This is a contradiction. Hence, $n^{-1}I'$ is primitive as well.

\begin{definition}
    The quaternion $\ell$-Brandt graph $\operatorname{Br}(p,\ell)$ is a directed multigraph whose vertex set  is a fixed set of representatives $I_1, \ldots, I_n$ for $\operatorname{Cls}_L(\frakO_0)$ and there is a directed edge from $I_i$ to $I_j$ for each $\ell$-neighbor $J$ of $I_i$ with $J\sim I_j$.
\end{definition}

Let $n=\#\Cls_L(\mathfrak{O}_0)$ be the left class number of $\mathfrak{O}_0$, and let $I_1, \cdots, I_n$ be the representatives for distinct classes in $\Cls_L(\mathfrak{O}_0)$. The $\ell$-Brandt matrix $B(p, \ell)$ is a $n \times n$-matrix with entries
\begin{equation}
    \begin{split}
        b_{ij} & :=\#\{J \subset I_i : \nrd(J)=\ell \nrd(I_i) \text{ and } J\sim I_j\} \\
        & = \#\{J\subset I_i : [I_i: J]=\ell^2 \text{ and } J \sim I_j\}\\
        & = \frac{1}{2a_j} \#\{\alpha \in I_j^{-1}I_i : \nrd(\alpha)\nrd(I_j)=\ell\nrd(I_i)\}
    \end{split}
\end{equation}
where $a_j=\#\calO_R(I_j)^\times/\{\pm 1\}$. The last equality holds since $ I_j\alpha=J \subseteq I_i$ with $\nrd(J)=\ell\nrd(I_i)$ if and only if $\alpha \in I_j^{-1}I_i$ and $\nrd(\alpha)\nrd(I_j)=\ell\nrd(I_i)$
where $\alpha$ is well-defined up to left multiplication by $\mu \in \calO_R(I_j)^\times$ (cf. \cite[41.1.3]{Voi21}). 

\begin{details}
    \begin{tbox}
    If $\mu \in B_0^\times$, then $I\mu = I$ if and only if $\mu \in \mathcal{O}_R(I)^\times$ \cite[Exercise 16.3]{Voi21}.

    In particular, 
\begin{equation}
    b_{ii}=\frac{1}{2a_i}\# \{\alpha \in \calO_R(I_i) : \nrd(\alpha)=\ell\}.
\end{equation}
\end{tbox}
\end{details}

Hence, the $\ell$-Brandt matrix $B(p,\ell)$ is the adjacency matrix of the quaternion $\ell$-Brandt graph $\operatorname{Br}(p,\ell)$, where $b_{ij}$ is the number of directed edges from $I_i/\sim$ to $I_j/\sim$. See \cite{Gro87} and \cite[Chapter 41]{Voi21} for more details on the quaternion $\ell$-Brandt graph and the $\ell$-Brandt matrix.

For an order $\mathfrak{O}$ in $B_0$, the genus of $\mathfrak{O}$ is the set of orders in $B_0$ locally isomorphic to $\mathfrak{O}$. The type set of $\mathfrak{O}$, denoted by $\typee(\mathfrak{O})$, is the set of isomorphism classes of orders in the genus of $\mathfrak{O}$. When $\mathfrak{O}=\mathfrak{O}_0$ is maximal, since any local maximal orders in $B_\ell$ are conjugate to each other, $\typee(\mathfrak{O}_0)$ is the set of conjugacy classes of maximal orders in $B_0$. 

\begin{definition}
    The quaternion $\ell$-ideal graph $Q(p,\ell)$ is a directed multigraph whose vertex set is $\typee(\mathfrak{O}_0)$ and there is a directed edge from $\mathfrak{O}/\cong$ to $\mathfrak{O}'/\cong$ for each primitive connecting $\mathfrak{O},\mathfrak{O}'$-ideal of reduced norm $\ell$.
\end{definition}

Analogously, for a $\Q$-algebra embedding $\iota_0: K \rightarrow B_0$, we define $\Cls_L(\mathfrak{O}_0, \iota_0)$ to be the set of left $\iota_0$-classes of $\mathfrak{O}_0$ and $\typee(\mathfrak{O}_0,\iota_0)$ to be the set of $\iota_0$-conjugacy classes of $\mathfrak{O}_0$. For $\Q$-algebra embeddings $\iota_{0,1}: K_1 \rightarrow B_0$ and $\iota_{0,2}:K_2 \rightarrow B_0$, we can similarly define $\typee(\mathfrak{O}_0, \jmath_0)$, where $\jmath_0=(\iota_{0,1}, \iota_{0,2})$ is the induced isomorphism of quaternion algebras, but it is just the set of maximal orders in $B_0$ by Proposition \ref{single and double conjugate}. We define $\Cls_L(\mathfrak{O}_0, \jmath_0)$ to be the set of primitive left $\mathfrak{O}_0$-ideals.

\begin{definition}
The $K$-oriented quaternion $\ell$-ideal graph $Q_K(p,\ell)$ is a directed multigraph whose vertex set is $\typee(\mathfrak{O}_0, \iota_0)$ and there is a directed edge from $\mathfrak{O}/\cong_{\iota_0}$ to $\mathfrak{O}'/\cong_{\iota_0}$ for each left $\iota_0$-class of primitive connecting $\mathfrak{O},\mathfrak{O}'$-ideal of reduced norm $\ell$.
\end{definition}

\begin{definition}
    The $K_1,K_2$-oriented quaternion $\ell$-ideal graph $Q_{K_1,K_2}(p,\ell)$ is a directed graph whose vertex set is $\typee(\mathfrak{O}_0, \jmath_0)$ and there is a directed edge from $\mathfrak{O}$ to $\mathfrak{O}'$ if there is a primitive connecting $\mathfrak{O},\mathfrak{O}'$-ideal of reduced norm $\ell$.
\end{definition}

\begin{extra}
    \begin{tbox}
        \begin{definition} \label{extended left class set}
    Analogous to the equivalence of isogenies, for conjugate maximal orders $\mathfrak{O}, \mathfrak{O}' =\alpha\mathfrak{O} \alpha^{-1}$ in $B_0$, we can extend the definition of equivalence relation on left ideals in $\Cls_L(\mathfrak{O})$ and $\Cls_L(\mathfrak{O}')$; for a left $\mathfrak{O}$-ideal $I$ and left $\mathfrak{O}'$-ideal $I'$, we will say $I, I'$ are equivalent if $I_1 \sim I$ for $I_1=(\mathfrak{O}\alpha)I'$. We can similarly extend the definition of equivalence relation on $\iota_0$-left ideals in $\Cls_L(\mathfrak{O},\iota_0)$ and $\Cls_L(\mathfrak{O}',\iota_0)$ as well. 
\end{definition}
{\color{red} Note following the above definition: we need to have an equivalence relation on primitive connecting $\mathfrak{O},\mathfrak{O'}$-ideals of reduced norm $\ell$ as the equivalence classes of these should correspond to an edge in the quaternion $\ell$-ideal graphs.}

 \begin{remark}
     We show that edges in $Q(p,\ell)$ and $Q_K(p,\ell)$ are well-defined.
\[
\begin{tikzcd}
\alpha^{-1}\mathfrak{O}_1\alpha=\hspace{-3em}& \mathfrak{O} \arrow[r, "J"] \arrow[dr, swap, "J'"] & \mathfrak{O}' \arrow[d, "\mathfrak{O}'\alpha'", "\rotatebox{90}{$\sim$}"']    \\
& \mathfrak{O}_1\arrow[u, "\mathfrak{O}_1\alpha", "\rotatebox{90}{$\sim$}"']  \arrow[r, "J_1"'] & \mathfrak{O}_1'& \hspace{-3em} = {\alpha'}^{-1}\mathfrak{O}'\alpha' 
\end{tikzcd}
\]

Let $\mathfrak{O} \cong \mathfrak{O}_1$ and $\mathfrak{O}' \cong \mathfrak{O}_1'$ be maximal orders in $B_0$. Then $\mathfrak{O}= \alpha^{-1}\mathfrak{O}_1\alpha$ and $\mathfrak{O}_1'={\alpha'}^{-1}\mathfrak{O}'\alpha'$ for some $\alpha, \alpha' \in B_0^\times$. Scaling $\alpha, \alpha'$ by some integers, we may assume that $\alpha \in \mathfrak{O}_1$ and $\alpha' \in \mathfrak{O}'$ are primitive. By Lemma \ref{index for isomorphic maximal orders}(b), $\mathfrak{O}_1\alpha$ is the primitive $\mathfrak{O}_1, \mathfrak{O}$-connecting ideal and $\mathfrak{O}'\alpha'$ is the primitive $\mathfrak{O}',\mathfrak{O}_1'$-connecting ideal. Let $J$ be the primitive connecting $\mathfrak{O},\mathfrak{O}'$-ideal, $J'$ be the primitive connecting $\mathfrak{O},\mathfrak{O}_1'$-ideal, and $J_1$ be the primitive connecting $\mathfrak{O}_1,\mathfrak{O}_1'$-ideal. By Lemma \ref{colon ideals}(d), $J'=J\alpha'$ and $J_1=\mathfrak{O}_1\alpha J\alpha'$. Hence, $J\sim J'$. In particular, $J \sim_{\iota_0} J'$ if $\mathfrak{O} \cong_{\iota_0} \mathfrak{O}_1$ and $\mathfrak{O}' \cong_{\iota_0} \mathfrak{O}_1'$ with $\alpha, \alpha' \in \iota_0(K)^\times$. With Definition \ref{extended left class set}, we have $J \sim J_1$. In particular, $J \sim_{\iota_0} J_1$ if $\mathfrak{O} \cong_{\iota_0} \mathfrak{O}_1$ and $\mathfrak{O}' \cong_{\iota_0} \mathfrak{O}_1'$.
 \end{remark}

{\color{red} I'm still not understanding this remark: We want to show that if $J$ and $J'$ are equivalent $\mathfrak{O}, \mathfrak{O'}$-ideals, then $J_1$ and $J_1'$ are equivalent $\mathfrak{O}_1,\mathfrak{O}_1'$-ideals.
}

    \end{tbox}
\end{extra}

\begin{extra}
    \subsection{Multi-edges and loops}

In undireted graphs, a multi-edge is between two vertices.

We give a basic description of multi-edges and loops in the family of supersingular isogeny graphs and the family of quaternion ideal graphs. 

Let $\ell \neq p$ be a prime, $E, E'$ be supersingular elliptic curves over $\Fbar_p$, $\mathfrak{O} \cong \End(E)$ be a maximal order in $B_0$. Note that $\mathfrak{O}^\times \cong \Aut(E)$.

\begin{remark}[Multi-edges to another vertex]
    Suppose $E, E'$ are non-isomorphic. Let $\varphi_1, \cdots, \varphi_n : E \rightarrow E', n \leq \ell+1$ be distinct $\ell$-isogenies between  up to equivalence. By Theorem \ref{bjection between curves, ideals, maximal orders}, there are corresponding primitive left $\End(E)$-ideals $\tilde{J}_i=\Hom(E', E)\varphi_i, 1 \leq i \leq n$ of reduced norm $\ell$ where they are all distinct.

Let $\varphi, \varphi' : E \rightarrow E'$ be equivalent isogenies. The followings are equivalent
\begin{enumerate}
    \item 
\end{enumerate}
Recall that elliptic curve $E/\Fbar_p$ has extra automorphisms when its $j$-invariant is either $0$ or $1728$ (cf. \cite[Theorem 10.1]{Sil09}).
\end{remark}

\begin{remark}[Loops]
Suppose $E, E'$ are isomorphic. Let $\varphi: E \rightarrow E'$ be isogeny of degree $\ell$. Since $E \cong E'$, there is an equivalent endomorphism of endomorphism of $E$. 

Such an endomorphism exists if and only if 

Endomorphisms of $E$ corresponds to two-sided $\mathfrak{O}$-ideals. See \cite[Chapter 18]{Voi21} for details on two-sided $\mathfrak{O}$-ideals.

Let $\mathfrak{O}$ be a maximal order in $B_0$. Let $N_{B_0^\times}(\mathfrak{O})$ be the normalizer of $\mathfrak{O}$ in $B_0^\times$ given by
\begin{equation}
    N_{B_0^\times}(\mathfrak{O}):=\{\alpha \in B^\times : \alpha^{-1}\mathfrak{O}\alpha=\mathfrak{O}\}.
\end{equation}
By \cite[Exercise 16.7]{Voi21}, $I=\mathfrak{O}\alpha$ is an integral two-sided $\mathfrak{O}$-ideal if and only if $\alpha \in N_{B_0^\times}(\mathfrak{O}) \cap \mathfrak{O}$. In case $\nrd(I)=\nrd(\alpha)=\ell$.

There is an exact sequence of groups
\begin{equation}
    \begin{split}
          1 \rightarrow \mathfrak{O}^\times \rightarrow N_{B_0^\times}(\mathfrak{O}) & \rightarrow \PIdl(\mathfrak{O}) \rightarrow 1 \\
          \alpha & \mapsto \mathfrak{O}\alpha \mathfrak{O}
    \end{split}
\end{equation}
\cite[Lemma 18.5.1]{Voi21}.
\end{remark}

If $p>3$, then $j=0$ is supersingular if and only if $p \equiv 2 \pmod 3$, and $j=1728$ is supersingular if and only if $p \equiv 3 \pmod 4$. chehck for $p=2, 3$.

\end{extra}

\subsection{Graph isomorphisms} We give graph isomorphisms between supersingular isogeny graphs and quaternion ideal graphs.

\begin{remark} We give an additional description on the vertex sets of the graphs in the right column of (\ref{graph isomorphisms}).

Recall that $\Cls_L(\mathfrak{O}_0)$ is the vertex set of $\operatorname{Br}(p,\ell)$ and $\typee(\mathfrak{O}_0)$ is the vertex set of $Q(p,\ell)$. There is a surjection from $\Cls_L(\mathfrak{O}_0)$ to $\typee(\mathfrak{O}_0)$ given by
  \begin{equation}
\begin{array}{c@{\;}c@{\;}c}
\Cls_L(\mathfrak{O}_0)
& \rightarrow & \typee(\mathfrak{O}_0) \\[1ex]
I/\sim
& \mapsto & 
\calO_R(I)/\cong
\end{array}
\end{equation}
\cite[Lemma 17.4.13]{Voi21}. We have the following diagram of maps where each double arrow means bijection between sets and each one-way arrow means surjection from one set to the other. The surjection among the variants of left class set and the surjection among the variants of type set follow naturally from their definition. The bijection between $\Cls(\mathfrak{O},\jmath_0)$ and $\typee(\mathfrak{O},\jmath_0)$ is given in Theorem \ref{bjection between curves, ideals, maximal orders}, and the bijection between $\Cls(\mathfrak{O},\iota_0)$ and $\typee(\mathfrak{O},\iota_0)$ is given in Corollary \ref{bijections in single orientation}.

\begin{equation}
\begin{array}{c@{\;}c@{\;}c@{\;}c@{\;}c}
\Cls_L(\mathfrak{O}_0,\jmath_0) & \longrightarrow & \Cls_L(\mathfrak{O}_0,\iota_0) & \longrightarrow & \Cls_L(\mathfrak{O}_0) \\[1ex]
I & \longmapsto & I/\sim_{\iota_0} & \longmapsto & I/\sim \\[-1ex]
\rotatebox{270}{$\longleftrightarrow$} & & \rotatebox{270}{$\longleftrightarrow$} & & \rotatebox{270}{$\longmapsto$} \\[5ex]
\calO_R(I) & \longmapsto & \calO_R(I)/\cong_{\iota_0} & \longmapsto & \calO_R(I)/\sim \\[1ex]
\typee(\mathfrak{O}_0,\jmath_0) & \longrightarrow & \typee(\mathfrak{O}_0,\iota_0) & \longrightarrow & \typee(\mathfrak{O}_0) 
\end{array}
\end{equation}
\end{remark}

We could have defined $\Cls_L(\mathfrak{O}_0, \jmath_0)$ to be the set of left $\mathfrak{O}_0$-ideals, then the map $I/\Q \mapsto \operatorname{pr}(I)  \mapsto \calO_R(I)$ gives a bijection between the sets $\Cls_L(\mathfrak{O}_0,\jmath_0)/\Q$ and $\typee(\mathfrak{O}_0,\jmath_0)$.

\begin{theorem} \label{graph isomorphisms} There are graph isomorphisms:
\begin{equation} 
\begin{array}{c@{\;}c@{\;}c}
\substack{ \text{reduced supersingular} \\ \ell\text{-isogeny graph } \tilde{G}(p,\ell) }
& \longleftrightarrow & 
\text{quaternion } \ell\text{-ideal } \text{graph } Q(p,\ell)  \\[1ex]
[j] & \mapsto & 
\mathfrak{O}/\cong \\[1ex]
 & & I^{-1}I/\sim \\[1ex]
\text{supersingular } \ell\text{-isogeny graph } G(p,\ell) 
& \longleftrightarrow & 
\text{quaternion $\ell$-Brandt graph } \operatorname{Br}(p,\ell)  \\[1ex]
E / \cong 
& \mapsto & 
I/\sim \\
E\rightarrow E'/\sim
&  &  \\[1ex]
\substack{K\text{-oriented supersingular} \\ \ell\text{-isogeny graph } G_K(p,\ell) }
& \longleftrightarrow & 
\substack{K\text{-oriented quaternion $\ell$-ideal}\\ \text{graph } Q_{K}(p,\ell)} \\[1ex]
(E,\iota) / \cong 
& \mapsto & 
\mathfrak{O}/\cong_{\iota_0} \\
(E, \iota) \rightarrow (E',\iota')/\sim  & \mapsto & I^{-1}I'/\sim_{\iota_0} \\[1ex]
\substack{K_1,K_2\text{-oriented supersingular} \\ \ell\text{-isogeny graph } G_{K_1, K_2}(p,\ell) }
& \longleftrightarrow & 
\substack{K_1,K_2\text{-oriented quaternion $\ell$-ideal}\\ \text{graph } Q_{K_1,K_2}(p,\ell)} \\[1ex]
(E,\jmath) / \cong 
& \mapsto & 
\mathfrak{O} \\ (E, \jmath) \rightarrow (E',\jmath')/\sim  & \mapsto & 
\operatorname{pr}(I^{-1}I')
\end{array}
\end{equation}
\end{theorem}

\begin{details}
\begin{tbox}
    For the description of quaternion $\ell$-ideal graph, see Definition 2.3, from \href{https://eprint.iacr.org/2021/372.pdf}{Explicit connections between supersingular isogeny graphs and
Bruhat–Tits trees}.
\end{tbox}
\end{details}

\begin{proof}   
    Let $E_0/\Fbar_p$ be a supersingular elliptic curve and $\mathfrak{O}_0=\End(E_0)$. Instead of an abstract quaternion algebra, we consider $B_0=\End^0(E_0)$. 

    We first show bijections between vertex sets.

    \begin{extra}
    \begin{tbox}
      Let $\pi_p:\Fbar_p \rightarrow \Fbar_p, x \mapsto x^p$ be the Frobenius automorphism of finite field of characteristic $p$. $\pi_p$ is the generator of $\Gal(\F_{p^n}/\F_p)\cong \Z/n\Z$. There is a bijection between up to Frobenius twist
\begin{equation}
\begin{array}{c@{\;}c@{\;}c}
\{j \in \F_{p^2} : j \text{ is supersingular}\}/\Gal(\F_{p^2}/\F_p) 
& \longleftrightarrow & 
\left\{j \in \Fbar_p  : j \text{ is supersingular} \right\}/\pi_p
\end{array}
\end{equation}

    quotienting by the action of the Frobenius automorphism 
\end{tbox}
\end{extra}

    By \cite[Lemma 42.4.1]{Voi21}, there exists a bijection between the set of Galois orbits of supersingular $j$-invariants in $\F_{p^2}$ and the set of conjugacy classes of maximal orders in $B_0$
\begin{equation}
\label{classical-deuring}
\begin{array}{c@{\;}c@{\;}c}
\{j \in \F_{p^2} : j \text{ is supersingular}\}/\Gal(\F_{p^2}/\F_p) 
& \longleftrightarrow & 
\typee(\mathfrak{O}_0)
\end{array}
\end{equation}    
sending $j \in \F_{p^2}$ to a maximal order $\mathfrak{O}$ such that $\End(E(j))\cong \mathfrak{O}$, where $E(j)$ is a choice of elliptic curve over $\overline{\F}_p$ with $j$-invariant equal to $j$.

\begin{details}
    \begin{tbox}
    Voight commented in \cite[Remark 42.3.6]{Voi21} that the equivalence of categories are enriched with a Frobenius morphism in \cite[Theorem 45]{Koh96}
\end{tbox}
\end{details}

    By \cite[Corollary 42.3.7]{Voi21}, there exists a bijection between isomorphism classes of supersingular elliptic curves over $\Fbar_p$ and $\Cls_L(\mathfrak{O}_0)$ that sending $E/\cong$ to $I/\sim$ such that $\End(E) \cong \calO_R(I)$ and $\Aut(E)\cong \calO_R(I)^\times$.

Fix a $K$-orientation $\iota_0$ on $E_0$. By Corollary \ref{bijections in single orientation}, there exists a bijection between the set of isomorphism classes of $K$-oriented supersingular elliptic curves over $\Fbar_p$ and the set of $\iota_0$-conjugacy classes of maximal orders in $\End^0(E_0)$. 

Fix a $K_1,K_2$-orientation $\jmath_0$ on $E_0$. By Corollary \ref{double oriented curves to maximal orders}, there exists a bijection between the set of isomorphism classes of $K_1,K_2$-oriented supersingular elliptic curves over $\Fbar_p$ and the set of maximal orders in $\End^0(E_0)$.  

Next we show bijections between edge sets.

Let $n=\#\Cls_L(\mathfrak{O}_0)$. Let $I_1, I_2, \cdots, I_n$ be the representatives for classes in $\Cls_L(\mathfrak{O}_0)$ and $E_1, E_2, \cdots, E_n$ be the representatives for the corresponding isomorphism classes of supersingular elliptic curves. We may assume $I_1, I_2, \cdots, I_n$ are primitive. The entry $b_{ij}$ of the quaternion $\ell$-Brandt matrix $B(p, \ell)$ is the number of equivalence classes of isogenies from $E_i$ to $E_j$ of degree $\ell$ \cite[Proposition 2.3]{Gro87}. In particular, recall that $b_{ij}$ counts the number of elements $\alpha \in I_j^{-1}I_i$ with $\nrd(\alpha)\nrd(I_j)=\ell\nrd(I_i)$ well-defined up to left multiplication by $\mu \in \mathcal{O}_R(I_j)^\times$. Since $\End(E_j) \cong \calO_R(I_j)$ and $\Aut(E_j) \cong \calO_R(I_j)^\times$, this is equivalent to counting the number of $\ell$-isogenies from $E_i$ to $E_j$ up to post-composition by an automorphism of $E_j$. This gives a bijection between the edge sets of $G(p,\ell)$ and $\Br(p,\ell)$.

\begin{details}
\begin{tbox}
  Found the reference in  \href{https://arxiv.org/abs/2101.08761}{Loops, multi-edges and collisions in supersingular isogeny graphs 
} p.9. it is the number of subgroups $G$ of $E_i$ of order $\ell$ such that $E_i/G \cong E_j$. Therfore, $b_{ij}$ is the adjacency matrix of $G(p,\ell)$.\\

Also, see \href{https://www.math.mcgill.ca/goren/Students/NicoleThesis.pdf}{Proposition 2.31} for modern writing of Gross.

Note that it seems in Ghantous's paper and the original reference, their Brandt matrix is the transpose of the one given by Voight. 
\end{tbox}
\end{details}

$\tilde{G}(p,\ell)$ and $Q(p,\ell)$ are quotient graphs of $G(p,\ell)$ and $\Br(p, \ell)$ respectively, where in $\tilde{G}(p,\ell)$, we identify the $j$-invariants $j$ and $j^p$, and in $Q(p,\ell)$, we identify two classes $[I], [I']$ in $\Cls_L(\mathfrak{O}_0)$ if $\calO_R(I)\cong \calO_R(I')$. From the bijection between the edge sets of $G(p,\ell)$ and $\Br(p,\ell)$, we get a natural bijection between the edge sets of $\tilde{G}(p,\ell)$ and $Q(p,\ell)$. 

By Lemma \ref{no loop maximal order} and Corollary \ref{no loop curve}, there are no loops, i.e., edges to itself, in $G_{K_1,K_2}(p,\ell)$ and $Q_{K_1, K_2}(p,\ell)$. Let $(E,\jmath)/\cong$ and $(E', \jmath')/\cong$ be vertices in $G_{K_1, K_2}(p,\ell)$ which are neighbors. The first bijection in Corollary \ref{double oriented curves to maximal orders} sends $(E,\jmath)/\cong$ and $(E',\jmath')/\cong$ to vertices $\mathfrak{O}_{\varphi}$ and $\mathfrak{O}_{\varphi'}$ in $Q_{K_1, K_2}(p,\ell)$ respectively, and the second bijection sends $\varphi: (E_0,\jmath_0) \rightarrow (E,\jmath) /\sim$ and $\varphi': (E_0,\jmath_0) \rightarrow (E',\jmath') / \sim$ to $I_\varphi$ and $I_{\varphi'}$ respectively. Let $\psi:(E,\jmath)\rightarrow (E',\jmath')/\sim$ be an edge in $G_{K_1, K_2}(p,\ell)$ from $(E,\jmath)/\cong$ to $(E',\jmath')/\cong$. By Corollary \ref{bijection of edges with orientations}, there is a bijection between the edge sets of $G_{K_1, K_2}(p,\ell)$ and $Q_{K_1, K_2}(p,\ell)$ sending $\psi:(E,\jmath)\rightarrow (E',\jmath')/\sim$ to $\operatorname{pr}(I_{\varphi}^{-1}I_{\varphi})$.

By setting $K_1=K$, $G_K(p,\ell)$ and $Q(p,\ell)$ are quotient graphs of $G_{K_1, K_2}(p,\ell)$ and $Q_{K_1,K_2}$ respectively. From the bijection between edge sets of $G_{K_1, K_2}(p,\ell)$ and $Q_{K_1, K_2}(p,\ell)$, we get a natural bijection between the edge sets of $G_{K}(p,\ell)$ and $Q_{K}(p,\ell)$.
\end{proof}

\begin{corollary} \label{no multi-degs and loops}
    There are no multi-edges and loops in $G_{K_1, K_2}(p,\ell)$ and $Q_{K_1, K_2}(p, \ell)$. 
\end{corollary}
\begin{proof}
     Since edges in $Q_{K_1, K_2}(p,\ell)$ are precisely the primitive connecting ideals of reduced norm $\ell$ between maximal orders in $B_0$, there is at most one directed edge from $\mathfrak{O}$ to $\mathfrak{O}'$ for maximal orders $\mathfrak{O}, \mathfrak{O}'$ in $B_0$ by Proposition~\ref{primitive and minimal reduced norm ideal}. We already showed that there are no loops in $G_{K_1, K_2}(p,\ell)$ and $Q_{K_1, K_2}(p,\ell)$. 
\end{proof}

\begin{remark}
    The bijections between vertex sets and edge sets of graphs in Theorem \ref{graph isomorphisms} justify that one graph is well-defined when the corresponding graph is well-defined, i.e.\ after taking equivalence classes of vertices. For example, we show that $Q_K(p,\ell)$ is well-defined; let $\mathfrak{O}/\cong_{\iota_0}, \mathfrak{O}'/\cong_{\iota_0}$ be distinct vertices in $Q_{K}(p,\ell)$ and let $\mathfrak{O}_1 \cong_{\iota_0} \mathfrak{O}, \mathfrak{O}'_1 \cong_{\iota_0}\mathfrak{O}'$ be maximal orders in $B_0$. We want to show that there is a bijection between the set of left $\iota_0$-classes of primitive connecting $\mathfrak{O},\mathfrak{O}'$-ideals of reduced norm $\ell$ and the set of $\iota_0$-classes of primitive connecting $\mathfrak{O}_1, \mathfrak{O}_1'$-ideals of reduced norm $\ell$. Let $(E, \iota), (E',\iota'), (E_1, \iota_1), (E_1', \iota_1')$ be $K$-oriented supersingular elliptic curves over $\Fbar_p$ corresponding to maximal orders $\mathfrak{O}, \mathfrak{O}', \mathfrak{O}_1, \mathfrak{O}_1'$ respectively. 
\[
\begin{tikzcd}
 \mathfrak{O} \arrow[d, phantom, "\rotatebox{-90}{\ $\cong_{\iota_0}$}", ""'] \arrow[r, "J/\sim_{\iota_0}"] & \mathfrak{O}' \arrow[d, phantom, "\rotatebox{-90}{\ $\cong_{\iota_0}$}", ""']    \\
 \mathfrak{O}_1 ,  \arrow[r, "J_1/\sim_{\iota_0}"'] & \mathfrak{O}_1'
\end{tikzcd} \qquad
\begin{tikzcd}
 (E,\iota) \arrow[d, phantom, "\rotatebox{-90}{$\cong$}", ""'] \arrow[r, "\psi/\sim"] & (E',\iota') \arrow[d, phantom, "\rotatebox{-90}{$\cong$}", ""']    \\
 (E_1,\iota_1) ,  \arrow[r, "\psi_1/\sim"'] & (E_1', \iota_1')
\end{tikzcd}
\]
Given a left $\iota_0$-class of primitive connecting $\mathfrak{O}, \mathfrak{O}'$-ideal $J/\sim_{\iota_0}$ of reduced norm $\ell$, there is an equivalence class of $K$-oriented isogeny $\psi : (E,\iota)\rightarrow (E',\iota')/\sim$ corresponding to $J/\sim_{\iota_0}$. Since $G_{K}(p,\ell)$ is well-defined, we have an equivalence class of $K$-oriented isogeny $\psi_1: (E_1, \iota_1) \rightarrow (E_1', \iota_1')/\sim$ corresponding to $\psi: (E,\iota) \rightarrow (E',\iota')/\sim$, and we identify them as an edge from $(E,\iota)/\cong$ to $(E', \iota')/\cong$. There is also a left $\iota_0$-class of primitive connecting $\mathfrak{O},\mathfrak{O}'$-ideal $J_1 /\sim_{\iota_0}$ corresponding to $\psi_1:(E_1, \iota_1) \rightarrow (E_1', \iota_1')/\sim$, and we identify it with $J/\sim{\iota_0}$ as an edge from $\mathfrak{O}/\cong \rightarrow \mathfrak{O}'/\cong$.
\end{remark}

\subsection{Connection between oriented isogeny graphs and Bruhat-Tits trees}
Next, we discuss how the connectivity in oriented isogeny graphs is related to Bruhat-Tits trees. For standard background on Bruhat-Tits trees, see \cite[23.5]{Voi21} or \cite{AIL+21, Mil15}. The link between supersingular isogeny graphs and Bruhat-Tits trees was discussed in \cite{AIL+21, Mil15}.

We say two homothety classes of $\Z_\ell$-lattices in $\Q_\ell^2$ are adjacent if there exist representatives $\Lambda, \Lambda'$ such that 
\begin{equation}
    \ell \Lambda \subsetneq \Lambda' \subsetneq \Lambda.
\end{equation}
This is equivalent to $\Lambda'$ being a cyclic sublattice of index $\ell$ in $\Lambda$.

\begin{definition}
   Let $\ell$ be a prime. The Bruhat-Tits tree $\mathcal{T}(\ell)$ associated to $\PGL_2(\Q_\ell)$ is an undirected $(\ell+1)$-regular connected tree graph whose vertices are homothety classes of $\Z_\ell$-lattices in $\Q_\ell^2$ and edges are pairs of adjacent homothety classes.
\end{definition}

The vertices of $\mathcal{T}(\ell)$ are in bijection with local maximal orders in $\M_2(\Q_\ell)$ by 
\begin{equation}
    \Lambda/\sim \ \mapsto \End_{\Z_\ell}(\Lambda)
\end{equation}
for any choice of representative \cite[Lemma 23.5.2]{Voi21}.

Now assume $\ell \neq p$. For a maximal order $\mathfrak{O} \subseteq B_0$, let $S_\ell(\mathfrak{O})$ denote the set of maximal orders $\mathfrak{O}'$ in $B_0$ such that there exists an integral connecting $\mathfrak{O},\mathfrak{O}'$-ideal of $\ell$-power reduced norm $\not=1$. We will show that maximal orders in $S_\ell(\mathfrak{O})$ is in bijection with vertices in $\mathcal{T}(\ell)$.

\begin{proposition} \label{connected component and local maximal orders}
  There are bijections
 \begin{equation}
\begin{array}{c@{\;}c@{\;}c}
\left\{ \substack{ \text{vertices in a connected} \\ \text{component of } G_{K_1,K_2}(p,\ell) } \right\} 
& \longleftrightarrow & 
S_\ell(\mathfrak{O}) \\[1ex]
& \longleftrightarrow & 
\{\text{local maximal orders in } B_\ell\} \\[1ex]
& \longleftrightarrow &  \{\text{vertices in } \mathcal{T}(\ell)\}.
\end{array}
\end{equation}
\end{proposition}
\begin{proof}
    Let $\mathfrak{O}' \subseteq B_0$ be a maximal order and suppose there exists an integral connecting $\mathfrak{O}, \mathfrak{O}'$-ideal $I$ of $\ell$-power reduced norm $\not= 1$. Let $\ell'\neq \ell$ be a prime. Since $I$ is integral and locally principal,    
    \begin{equation}
        I_{\ell'}=\mathfrak{O}_{\ell'} \mu=\mu\mathfrak{O}'_{\ell'}
    \end{equation}
    for some $\mu \in \mathfrak{O}_{\ell'} \cap \mathfrak{O'}_{\ell'}$, where $\mu$ is a unit in $\mathfrak{O}_{\ell'}$ and in $\mathfrak{O}_{\ell'}'$ since $\nrd(\mu) \in \Z_\ell^\times$.
    Hence,
\begin{equation}
      \mathfrak{O}_{\ell'}=\mathfrak{O}_{\ell'} \mu=I_{\ell'}=\mu\mathfrak{O}'_{\ell'}=\mathfrak{O}'_{\ell'}.
\end{equation}
As well we have that
\begin{equation}
    I_\ell=\mathfrak{O}_\ell \alpha =\alpha\mathfrak{O}'_\ell
\end{equation}
for some $\alpha \in B_{\ell}^\times$. Hence,
    \begin{equation}
\mathfrak{O}'_\ell=\alpha^{-1}\mathfrak{O}_\ell \alpha.
    \end{equation} The set of such maximal orders $\mathfrak{O}'_\ell$ is in bijection with the set of local maximal orders in $B_{\ell}$.

    Therefore, there is a one-to-one correspondence between $S_\ell(\mathfrak{O})$ and the set of local maximal orders in $B_{\ell}$.
\end{proof}
\begin{details}
   \begin{tbox}
        Also see \href{https://eprint.iacr.org/2021/372.pdf}{Explicit connections between supersingular isogeny graphs and Bruhat–Tits trees}
   \end{tbox}
\end{details}
\begin{details}
    \begin{tbox}
        (i) $I$ is locally principal since $\calO_L(I), \calO_R(I)$ are maximal by \cite[Corollary 17.2.3]{Voi21}. \\
(ii) $\nrd(I)=(\ell^k)$ as an ideal. \\
(iii) $I$ is integral if and only if $I \subseteq \calO_L(I)\cap \calO_R(I)$ \cite[Lemma 16.2.8]{Voi21}. \\
(iv) For any prime $\ell'\neq \ell$, $\nrd(\mu)=\nrd(I_{\ell'})=\nrd(I)_{\ell'} \in \Z_{\ell'}^\times$. \\
(v) $\mu$ is a unit in $\calO_{\ell'}$ as $\nrd(\mu) \in \Z_{\ell'}^\times$
    \end{tbox}
\end{details}

\begin{extra}
    \begin{remark}
For a given left $\mathfrak{O}$-ideal $I$, KLPT algorithm returns a left $\mathfrak{O}$-ideal $I'$ in the same class with $\ell$-power reduced norm. This is done by strong approximation: suppose $n=\nrd(I)$. There exists $\alpha \in B_{0}$ such that $I=\mathfrak{O}n+\mathfrak{O}\alpha$. The strong approximation theorem implies that there exists $\beta \in I$ such that $\beta \equiv \alpha \mod n\mathfrak{O}$ and $\nrd(\alpha)=N\ell^e$ for some $e>0$. In conclusion $I=\mathfrak{O}n+\mathfrak{O}\alpha=\mathfrak{O}n+\mathfrak{O}\beta$.

{\color{red} 
\[
\begin{tikzcd}
\mathfrak{O} \arrow[r, "I \text{ with } \nrd(I)=n"] \arrow[dr, swap, "\jmath"] & \calO_R(I) \\
& \mathfrak{O}
\end{tikzcd}
\]}
\end{remark}
\end{extra}

\begin{theorem} \label{tree structure of double oriented graphs}
   Identifying the unique conjugate directed edges between two neighboring vertices, the graphs $G_{K_1, K_2}(p,\ell)$ and $Q_{K_1, K_2}(p,\ell)$ become an undirected graph which is a tree.
\end{theorem}
\begin{proof}
    By Proposition \ref{connected component and local maximal orders}, there is a bijection between vertices in a connected component of $Q_{K_1, K_2}(p,\ell)$ and vertices in $\mathcal{T}(\ell)$. Given two vertices $\mathfrak{O}, \mathfrak{O}'$ in $Q_{K_1, K_2}(p,\ell)$, $\mathfrak{O}_\ell, \mathfrak{O}'_\ell$ are vertices in $\mathcal{T}(\ell)$, and there is a directed edge from $\mathfrak{O}$ to $\mathfrak{O}'$ in $Q_{K_1, K_2}$ if and only if there is an edge between $\mathfrak{O}_\ell, \mathfrak{O}_\ell'$ in $\mathcal{T}(\ell)$. If there is a directed edge from $\mathfrak{O}$ to $\mathfrak{O}'$ given by the primitive connecting $\mathfrak{O},\mathfrak{O}'$-ideal of reduced norm $\ell$, then $\bar{I}$ gives a directed edge from $\mathfrak{O}'$ to $\mathfrak{O}$. We can identify $I$ and $\bar{I}$, and obtain an undirected version of $Q_{K_1, K_2}(p,\ell)$. The resulting undirected graph is a tree since there are no multi-edges and loops in $Q_{K_1, K_2}(p,\ell)$ by Corollary \ref{no multi-degs and loops}. 
\end{proof}

\begin{corollary} \label{local roots are neigbors}
    If two distinct vertices lie in the same connected component of $G_{K_1, K_2}(p,\ell)$, then they lie in different connected components of $G_{K_1, K_2}(p,\ell')$ for any prime $\ell' \neq \ell$.
\end{corollary}
\begin{proof}
Let $(E,\jmath)/\cong$ and $(E',\jmath')/\cong$ be distinct vertices in $G_{K_1, K_2}(p, \ell)$. There are corresponding maximal orders $\mathfrak{O}$ and $\mathfrak{O}'$ in $Q_{K_1, K_2}(p,\ell)$ by Theorem \ref{graph isomorphisms}. By Corollary~\ref{connected component and local maximal orders}, $(E,\jmath)$ and $(E',\jmath')$ are in the same connected component of $G_{K_1, K_2}(p,\ell)$ if and only if there is an integral connecting $\mathfrak{O},\mathfrak{O}'$-ideal of $\ell$-power reduced norm $\neq 1$. In the proof of Proposition \ref{connected component and local maximal orders}, we showed that if maximal orders $\mathfrak{O}, \mathfrak{O}'$ in $B_0$ have an integral connecting $\mathfrak{O},\mathfrak{O}'$-ideal of $\ell$-power reduced norm, then $\mathfrak{O}_{\ell'} = \mathfrak{O}'_{\ell'}$ for any prime $\ell' \neq \ell$. Assume for contradiction that there exists a connecting $\mathfrak{O}_{\ell'}, \mathfrak{O}'_{\ell'}$-ideal $I$ of $\ell'$-power reduced norm $\not=1$ for some prime $\ell'\neq\ell$. Then 
\begin{equation}
\mathfrak{O}_{\ell'}'=\alpha^{-1}\mathfrak{O}_{\ell'}\alpha
\end{equation}
for some $\alpha \in B_{\ell'}^\times$ such that 
\begin{equation}
\nrd(\alpha)=[\mathfrak{O}_{\ell'}:\mathfrak{O}_{\ell'}\cap\mathfrak{O}'_{\ell'}] =\nrd(I) \neq 1
\end{equation}
 where $I = \mathfrak{O} \alpha$ by Lemma \ref{minimal reduced norm local}. This is a contradiction as $\mathfrak{O}_{\ell'}=\mathfrak{O}_{\ell'}'$. Hence, there does not exits a connecting $\mathfrak{O},\mathfrak{O}'$-ideal of $\ell'$-power reduced norm $\not=1$ for any prime $\ell' \neq \ell$.
\end{proof}

Let $\iota_0: K \rightarrow B_0$ be an embedding of an imaginary quadratic field $K$. We can naturally extend it and get a canonical embedding $\iota_0: K \rightarrow B_\ell$.

\begin{definition}
    We say two local maximal orders $\mathfrak{O}, \mathfrak{O}'$ in $B_\ell$ are $\iota_0$-isomorphic, denoted by $\mathfrak{O}\cong_{\iota_0} \mathfrak{O}'$, if $\mathfrak{O}'=\alpha^{-1}\mathfrak{O}\alpha$ for some $\alpha \in \iota_0(K)^\times$. We will call the $\iota_0$-isomorphism class of $\mathfrak{O}$ the $\iota_0$-conjugacy class of $\mathfrak{O}$. We say two lattices $I, I'$ in $B_\ell$ are in the same left $\iota_0$-class, denoted by $I \sim_{\iota_0} I'$, if there exists $\alpha \in \iota_0(K)^\times$ such that $I'=I\alpha$.
\end{definition}

\begin{proposition}
    Being $\iota_0$-isomorphic is a local property.
\end{proposition}

\begin{proof}
    Let $\mathfrak{O}, \mathfrak{O}'$ be maximal orders in $B_0$ and $\alpha \in \iota_0(K)^\times$. Scaling $\alpha$ by an integer, we may assume that $\alpha \in \mathfrak{O}$ is primitive. By Proposition \ref{Primitivity local property} and Lemma \ref{index for isomorphic maximal orders}(b), $\mathfrak{O}'=\alpha^{-1}\mathfrak{O}\alpha$ if and only if $I=\mathfrak{O}\alpha$ is the primitive connecting $\mathfrak{O},\mathfrak{O}'$-ideal if and only if $I_\ell=\mathfrak{O}_\ell\alpha$ is the primitive connecting $\mathfrak{O}_\ell,\mathfrak{O}'_\ell$-ideal for every prime $\ell$ if and only if $\mathfrak{O}_{\ell}' = \alpha^{-1}\mathfrak{O}_{\ell}\alpha$ for every prime $\ell$.   
\end{proof}

\begin{definition}
     A $K$-oriented Bruhat-Tits graph $\mathcal{T}_{K}(\ell)$ associated to $\operatorname{PGL}_2(\mathbb{Q}_\ell)$ is a graph whose vertices are $\iota_0$-isomorphism classes of local maximal orders in $B_\ell$ and edges are $\iota_0$-equivalence classes of connecting ideals of reduced norm $\ell$ between them.
\end{definition}

Let $K=K_1$. Since the Bruhat-Tits tree $\mathcal{T}(\ell)$ is an undirected version of a connected component of $Q_{K_1, K_2}(p, \ell)$ by Theorem \ref{tree structure of double oriented graphs}, and $Q_{K}(p,\ell), \mathcal{T}_{K}(\ell)$ are quotient graphs of $Q_{K_1, K_2}(p,\ell), \mathcal{T}(\ell)$ respectively, $\mathcal{T}_{K}(\ell)$ is an undirected version of a connected component of $Q_{K}(p,\ell)$. Hence, $\mathcal{T}_K(\ell)$ inherits a volcano structure from $G_{K}(p,\ell)$ given in Proposition \ref{Onu21 Proposition 4.1}.

\begin{corollary}  \label{connected component and local maximal orders - oriented}
Let $\mathfrak{O}$ be a maximal order in $B_0$. There are bijections
 \begin{equation}
\begin{array}{c@{\;}c@{\;}c}
\left\{ \substack{ \text{vertices in a connected} \\ \text{component of } G_{K}(p,\ell) } \right\} 
& \longleftrightarrow & 
\{\substack{\text{$\iota_0$-conjugacy classes of} \\ \text{maximal orders in $S_\ell(\mathfrak{O})$}}\} \\[1ex]
& \longleftrightarrow & 
\{\substack{\text{$\iota_0$-conjugacy classes of} \\ \text{local maximal orders in } B_\ell}\} \\[1ex]
& \longleftrightarrow &  \{\text{vertices in } \mathcal{T}_K(\ell)\}.
\end{array}
\end{equation}
\end{corollary}

\subsection{Example of global roots} \label{Example of roots} We end this section by giving an explicit example of global roots of double-oriented graphs.

Let $p > 2 $ be a prime. Let $1, i, j, k$ be a standard basis for $B_0$ such that
\begin{equation*}
    i^2= -q, \quad j^2 = -p, \quad k = ij=-ji,
\end{equation*}
where $q$ is prime. By \cite[Proposition 5.1 and 5.2]{Piz80}, we can take
\begin{equation} \label{Pizer basis}
   q = \begin{cases}
            1 & \text{ if }  p \equiv 3 \pmod 4, \\
            2 & \text{ if }  p \equiv 5 \pmod 8, \\
            q \equiv 3 \pmod 4 \text{ and } (p/q)= -1 & \text{ if } p \equiv 1 \pmod 8 
        \end{cases}
\end{equation}
In the last case, under the generalized Riemann hypothesis (GRH), $q = O((\log p)^2)$ by \cite[Proposition 1]{EHL+18}.

\begin{remark}
\label{opposite-mod-4}
  In all cases, at most one of $-q, -p$ is congruent to $1 \pmod 4$.
\end{remark}

The full ring of integers of the quadratic fields $K_i=\Q(i)$ and $K_j=\Q(j)$ are given by 
\begin{equation*}
\calO_{K_i}, \calO_{K_j} = \begin{cases}
 \Z[i], \Z[\frac{1+j}{2}] & \text{if} \ p \equiv 3 \pmod 4, \\
 \Z[i], \Z[j] & \text{if} \ p \equiv 5 \pmod 8, \\
 \Z[\frac{1+i}{2}], \Z[j] & \text{if} \ p \equiv 1 \pmod 8.
\end{cases}
\end{equation*}
Denote by $\Z\langle \alpha_1, \ldots, \alpha_n \rangle$ the subring of $B_{p, \infty}$ generated by $\alpha_1, \ldots, \alpha_n \in B_0$. There are at most two special maximal orders that will serve as the roots of the supersingular isogeny graph in characteristic $p$. The following proposition from \cite[Lemma 3, 4, and 5]{KLPT14} describes them explicitly.


\begin{proposition} \label{roots}
Let $p > 2$ be a prime. There are exactly two maximal orders $\frakO_0$ and $\frakO_1$ in $B_0$ (possibly conjugate) containing $\Z\langle i, j \rangle$ if $p \not\equiv 1 \pmod 8$, and $\Z\langle \frac{1+i}{2}, j \rangle$ if $p \equiv 1 \pmod 8$. These are given explicitly as follows:
\begin{equation*}
\frakO_0, \frakO_1 = \begin{cases}
\Z\langle i, \frac{1+j}{2} \rangle, \Z\langle  i, \frac{1+k}{2} \rangle & \text{if} \ p \equiv 3 \pmod 4, \\
\Z \langle i, \frac{1+j+k}{2}, \frac{i+2j+k}{4}\rangle, \Z \langle i, \frac{1+j+k}{2}, \frac{i+2j-k}{4} \rangle & \text{if} \ p \equiv 5 \pmod 8, \\
\Z \langle \frac{1+i}{2}, j, \frac{ci+k}{q} \rangle, \Z \langle \frac{ 1+i}{2}, j, \frac{ci-k}{q} \rangle & \text{if} \ p \equiv 1 \pmod 8,
\end{cases}
\end{equation*}
where $c$ is an integer such that $q \mid c^2p+1$.
\end{proposition}

\begin{corollary}
\label{quat-root}
Let $p > 2$ be a prime. There are exactly two maximal orders $\frakO_0$ and $\frakO_1$ in $B_0$ (possibly conjugate) containing both $\calO_{K_i}$ and $\calO_{K_j}$. These are given explicitly as follows:
\begin{equation}
\frakO_0, \frakO_1 = \begin{cases}
\Z\langle i, \frac{1+j}{2} \rangle  & \text{if} \ p \equiv 3 \pmod 4, \\
\Z \langle i, \frac{1+j+k}{2}, \frac{i+2j+k}{4}\rangle, \Z \langle i, \frac{1+j+k}{2}, \frac{i+2j-k}{4} \rangle & \text{if} \ p \equiv 5 \pmod 8, \\
\Z \langle \frac{1+i}{2}, j, \frac{ci+k}{q} \rangle, \Z \langle \frac{ 1+i}{2}, j, \frac{ci-k}{q} \rangle & \text{if} \ p \equiv 1 \pmod 8,
\end{cases}
\end{equation}
where $c$ and $q$ are as in Proposition~\ref{roots}.
\end{corollary}
\begin{proof}
In the case $p \equiv 1 \pmod 4$, the listed maximal orders contain both $\calO_{K_i}$ and $\calO_{K_j}$. For $p \equiv 5 \pmod{8}$, we note that the order
\begin{equation}
  \Z \langle i, \frac{1+j+k}{2}, \frac{i+2j-k}{4} \rangle 
\end{equation}
contains the element $(-i+2j+k)/4$ by conjugating the third element by $j$. It follows that the order contains the element $j$.

In the case $p \equiv 3 \pmod{4}$, the order
\begin{equation}
    \Z\langle  i, \frac{1+k}{2} \rangle 
\end{equation}
does not contain $\calO_{K_j}$, for if it did, it would contain $(i+k)/2$,
and hence $(1+i)/2$ which is a contradiction.
\end{proof}

As a result, we have the following corollary.
\begin{corollary}
  There are at most two supersingular elliptic curves over $\overline{\F}_p$ up to isomorphism which are both $\calO_{K_i}$-oriented and $ \calO_{K_j}$-oriented.
\end{corollary}

Let $E_0, E_1$ denote the supersingular elliptic curves over $\overline{\F}_p$ corresponding to $\frakO_0, \frakO_1$, respectively.

\begin{lemma} \label{orientation-determination}
Let $E/\overline{\F}_{p}$ be a supersingular elliptic curve and $u \in \left\{ i, j \right\} \subseteq B_0$. Suppose we are given an isomorphism $\varphi : \End(E) \cong \frakO \subseteq B_0$ where $\mathfrak{O}$ is a maximal order in $B_0$. Then we can determine a $\Q(u)$-orientation on $E/\overline{\F}_{p}$ from $\mathfrak{O}$ and the corresponding primitive $\calO_u$-orientation on $E/\overline{\F}_p$.
\end{lemma}
\begin{proof}
The isomorphism $\varphi$ extends to an isomorphism $\varphi: \End^0(E)  \cong B_0$. The $\Q(u)$-orientation is given by $\varphi^{-1} \mid_{\Q(u)} : \Q(u) \hookrightarrow \End^0(E)$. Let $\calO_u$ be the order in $\Q(u)$ given by $\calO_u = \frakO \cap \Q(u)$. Then $\varphi^{-1} \mid_{\calO_u} : \calO_u \hookrightarrow \End(E)$ is the primitive $\calO_u$-orientation corresponding to the given $\Q(u)$-orientation on $E/\overline{\F}_p$.
\end{proof}

\section{Embedding numbers of Bass orders and counting roots} \label{embedding numbers of bass orders and counting roots}

Recently, Bass orders have been used in the study of endomorphism rings of supersingular elliptic curves \cite{EHL+20} and single-orientations \cite{Ler24}. Following this philosophy, we explore Bass orders to study double-orientations, specifically as a tool for computing the number of local and global roots.

We begin with a review of Bass orders and their embedding numbers. Let $R$ be a Dedekind domain, let $F$ be its field of fractions, let $B$ be a quaternion algebra over $F$, and let $\frakO \subseteq B$ be an $R$-order. 
\begin{definition}
We say $\frakO$ is Gorenstein if its codifferent
\begin{equation*}
    \left\{ \alpha \in B: \trd(\alpha \frakO) \subseteq R \right\}
\end{equation*}
is an invertible (two-sided) $\frakO$-module.
\end{definition}
\begin{definition}
We say $\frakO$ is Bass if its every superorder in $B$ is Gorenstein.
\end{definition}

The properties of Gorenstein and Bass orders were studied by Brzezinski \cite{Bre82, Bre83} and Eichler \cite{Eic36}. They also studied the problem of counting the number of distinct maximal orders containing a given Bass order \cite{BE92, Bre83}. These numbers are referred to as the embedding numbers of Bass orders. We present some useful properties of Bass orders. Our main reference in this section would be \cite{BE92}.

\begin{remark}
  In \cite{Voi21}, there is another notion of embedding number which counts the number of optimal embeddings of a quadratic order $\calO$ of a quadratic algebra $K$ into a maximal order $\mathfrak{O}$ of a quaternion algebra $B$ up to some equivalence.
\end{remark}

The following is proven in \cite{CSV21}.
\begin{theorem}
Let $\ell$ be a prime and assume that $R=\Z$ or $R=\Z_\ell$. Then $\mathfrak{O}$ is Bass if and only if it contains a maximal order in a quadratic subfield of $B$. Furthermore, being Bass is a local property.
\end{theorem}

\begin{details}
\begin{tbox}
See \cite{CSV21}, \cite[Chapter 24]{Voi21}, and \cite{BE92} for more details on Bass orders. 
\end{tbox}
\end{details}

From now on, we work in a quaternion algebra $B_0/\Q$ ramified at $p$ and $\infty$. 
 
\begin{definition} \label{Eichler symbol}
    Let $\mathfrak{O}$ be a Bass order in $B_0$. Given a local order $\mathfrak{O}_\ell:=\mathfrak{O} \otimes \Z_\ell$ in $B_\ell:=B_0 \otimes \Q_\ell $, let $\rad(\mathfrak{O}_\ell)$ be the Jacobson radical of $\mathfrak{O}_\ell$. The Eichler symbol is defined as follows:
\begin{equation}
 \left(\mathfrak{O}/\ell\right) = \begin{cases}
        1 & \text{if} \ \mathfrak{O}_\ell/\rad(\mathfrak{O}_\ell) \cong \F_\ell \times \F_\ell,\\
        0 & \text{if} \ \mathfrak{O}_\ell/\rad(\mathfrak{O}_\ell) \cong \F_\ell,\\
        -1 & \text{if} \ \mathfrak{O}_\ell/\rad(\mathfrak{O}_\ell) \ \text{is a quadratic extension of} \ \F_\ell.
    \end{cases}
\end{equation}
\end{definition}
In particular, if $\mathfrak{O}_\ell = \mathfrak{O}'_\ell$ for an order $\mathfrak{O}'$ in $B_0$, then $(\mathfrak{O}/\ell)=(\mathfrak{O}'/\ell)$.

\begin{details}
    \begin{tbox}
        Implicitly, we are only considering the case $\mathfrak{O}$ is non-maximal. By \cite[Lemma 24.3.6]{Voi21}, a local order $\mathfrak{O}$ is Eichler if and only if $\mathfrak{O}$ is maximal or residually split, i.e., $\mathfrak{O}/\rad(\mathfrak{O}) \simeq \F_\ell \times \F_\ell$. Voight defined $(\mathfrak{O}/\ell)=\ast$ if $\mathfrak{O}_\ell$ is maximal.

        In both Proposition \ref{Eichler Symbol Computation} and \ref{Formula for local embedding number}, we are assuming the reduced discriminant of the Bass order is divisible by $\ell$, which is again implicitly assuming the Bass order is not locally maximal at $\ell$.
    \end{tbox}
\end{details}

We can compute the Eichler symbol of a Bass order using the following property.

\begin{proposition}\label{Eichler Symbol Computation}
Let $\mathfrak{O} \subseteq B_0$ be a Bass order with the reduced discriminant $D$ and let $K \subseteq B_0$ be any quadratic subfield such that $\mathcal{O}_K \subseteq \mathfrak{O}$. 
\begin{enumerate}
\item[(a)] If $\ell \parallel D$, then
\begin{equation}
        \left(\frac{\mathfrak{O}}{\ell}\right)=\begin{cases}
        1 & \text{if and only if} \ \ell \neq p,\\
        -1 & \text{if and only if} \ \ell=p.
        \end{cases}
\end{equation}
In the first case, $\ell$ is split or ramified in $K$, and in the second, $\ell$ is ramified or inert in $K$.
\item[(b)] If $\ell^2 \mid D$, then
    \begin{equation}
        \left(\frac{\mathfrak{O}}{\ell}\right)=\begin{cases}
        1 & \text{if and only if} \ \ell \ \text{splits in} \ K,\\
        0 & \text{if and only if} \ \ell \ \text{is ramified in} \ K,\\
        -1 & \text{if and only if} \ \ell \ \text{is inert in} \ K.
        \end{cases}
    \end{equation}

\end{enumerate}
\end{proposition}

\begin{proof}
  This is \cite[Proposition 3]{BE92}.
\end{proof}

\begin{remark}
 It seems Brzezinski used term ``unramified'' for inert in \cite{BE92}. In his other paper \cite{Brz90}, which contains the proof of Proposition \ref{Eichler Symbol Computation}, he wrote ``$L\supseteq K$ is unramified or split''.        
\end{remark}

\begin{details}
    \begin{tbox}
         Let $\ell \neq 2, p$ and 
    \begin{equation}
        -d_1=a\cdot\ell^m \quad \text{and} \quad -d_2=b\cdot\ell^{n},
    \end{equation}
    where $m=\nu(-d_1)$ and $n=\nu(-d_2)$.

    By \cite[Proposition 1(a)]{BE92}, the Hilbert symbol formula over $\Q_\ell$ must satisfy
\begin{equation}
    (-d_1, -4d_1d_2)_\ell =1
\end{equation}
as $\ell \neq 2, p$. Using basic properties for Hilbert symbol in \cite[Lemma 12.4.3]{Voi21}, we can simplify as follows
\begin{equation}
    (-d_1, -4d_1d_2)_\ell=(-d_1, -d_1d_2)_\ell=(-d_1, d_1)_\ell(-d_1, -d_2)_\ell=(-d_1, -d_2)_\ell.
\end{equation}
Then by \cite[Exercise 12.16]{Voi21},
\begin{equation}
    (-d_1, -d_2)_\ell = (-1)^{mn(\ell-1)/2}\left(\frac{a}{\ell}\right)^{n}\left(\frac{b}{\ell}\right)^{m}.
\end{equation}
In particular, if $\ell \nmid \discrd(\mathfrak{O}_{i,j})$, then $m=n=0$ so $(-d_1, -d_2)=1$.
    \end{tbox}
\end{details}

\begin{definition}
Let $e_\ell(\mathfrak{O})$ denote the number of local maximal orders in $B_\ell$ containing $\mathfrak{O}_\ell$.
\end{definition}

Note that $e_\ell(\mathfrak{O})=1$ if $\ell \nmid D$ in which case $\mathfrak{O}_\ell$ is maximal by \cite[Lemma 15.5.3]{Voi21}. Also, $e_\ell(\mathfrak{O})=1$ if $\ell=p$ as the valuation ring in $B_\ell$ is the unique maximal order in $B_\ell$.

\begin{proposition} \label{Formula for local embedding number}
    Let $\mathfrak{O}$ be a non-maximal Bass order in $B_0$ of reduced discriminant $D$ and let $\ell \neq p$ be a prime dividing $D$. The local embedding number $e_\ell(\mathfrak{O})$ is given by
    \begin{equation}
        e_\ell(\mathfrak{O}) = \begin{cases}
            \nu_\ell(D) + 1 & \text{if} \ (\mathfrak{O}/\ell) = 1,\\
            2 & \text{if} \ (\mathfrak{O}/\ell) = 0 ,\\
            1 & \text{if} \ (\mathfrak{O}/\ell) = -1,
        \end{cases}
    \end{equation}
\end{proposition}
\begin{proof} 
   This is \cite[Proposition 2]{BE92}.
\end{proof}

By the local-to-global principle, we can write the formula of the global embedding number $e(\frak{O})$, the number of global maximal orders in $B_0$ containing $\mathfrak{O}$, as the product of local embedding numbers:
\begin{equation}
    e(\mathfrak{O}) = \prod_\ell e_\ell(\frakO),
\end{equation}
where $\ell$ runs over primes.

\begin{details}
    \begin{tbox}
        For orders $\mathfrak{O}, \mathfrak{O}'$ in  $B_0$,
    \begin{equation}
        \mathfrak{O}\subseteq \mathfrak{O}' \ \Leftrightarrow \ \mathfrak{O}_{(\ell)} \subseteq \mathfrak{O}'_{(\ell)}
      \ \Leftrightarrow \  \mathfrak{O}_\ell \subseteq \mathfrak{O}'_\ell \ \text{for every prime} \ \ell.
    \end{equation}

   \begin{itemize}
       \item Being Bass is a local property (Voight says it follows from the definition and being Gorenstein is lcoal property as invertibility is so). For the $R$-order $\mathfrak{O}$, We have
       \begin{enumerate}
           \item If $R$ is local, then $\mathfrak{O}$ is basic if and only if $\mathfrak{O}$ is Bass \cite[Proposition 24.5.8]{Voi21}.
           \item If $F=\Frac(R)$ is a number field, the following are equivalent: 
           \begin{enumerate}
               \item[(i)] $\mathfrak{O}$ is basic.
               \item[(ii)] $\mathfrak{O}_{(\mathfrak{p})}$ is basic for all primes $\mathfrak{p}$ of $R$.
               \item[(iii)] $\mathfrak{O}_{(\mathfrak{p})}$ is Bass for all primes $\mathfrak{p}$
 of $R$               
 \item[(iv)] $\mathfrak{O}$ is Bass.
           \end{enumerate}           
       \end{enumerate}, i.e., for orders in a quaternion algebra over a number field, being basic is a local property and it is equivalent to being Bass.
       \item For Bass order $\mathfrak{O}$, $e(\mathfrak{O})=\prod_\ell e_\ell(\mathfrak{O})$.
   \end{itemize}
    \end{tbox}
\end{details}

In Proposition \ref{roots}, we presented an example of maximal orders in $B_0$ corresponding to the roots of double-oriented supersingular isogeny graphs. In this section, we will describe these maximal orders in a more general setting and discuss how to count the number of roots in the double-oriented supersingular isogeny graph using them. The key idea is to find a pair of quadratic subfields of $B_0$ whose maximal orders generate a Bass order $\mathfrak{O}$ in $B_0$. Then, we count the number of maximal orders in $B_0$ containing $\mathfrak{O}$, either locally or globally. 

\begin{extra}
\begin{tbox}
Let $i,j$ be elements in $B_0$ such that 
\begin{equation}
    i^2 = a, j^2 = b, \tr(ij)=ij+ji=s\in \Q
\end{equation}
for some square-free integers $a,b <0$. Consider the following conditions on $s$:
 \begin{enumerate}
     \item[(a)] $s^2-4ab<0$ and we have Hilbert symbols over $\Q_\ell$
     \begin{equation}
         (a, s^2-4ab)_\ell =(b, s^2-4ab)_\ell=\begin{cases}
             1 & \text{if } \ell\neq p \\
             -1 & \text{if } \ell = p.
         \end{cases}
     \end{equation}
     This is the condition for imaginary quadratic subfields $\Q(\sqrt{a})$ and $\Q(\sqrt{b})$ have simultaneous embeddings into $B_0$, whose images are $K_i=\Q(i)$ and $K_j=\Q(j)$ respectively \cite[Proposition 1(a)]{BE92}. 
     \item[(b)] $s \equiv 2 \pmod{\delta}$, where 
\begin{equation} 
    \delta =\begin{cases}
        1 & \text{if} \ a \not\equiv 1 \pmod{4} \ \text{and} \ b \not\equiv 1 \pmod{4},\\
        4 & \text{if} \ a \equiv 1 \pmod{4} \ \text{and} \ b \equiv 1 \pmod{4},\\
        2 & \text{otherwise}.
    \end{cases}
\end{equation}
Under this condition, maximal orders in $K_i$ and $K_j$ generate a Bass order in $B_0$ \cite[Proposition 1(b)]{BE92}.
 \end{enumerate}

Suppose $s$ satisfies the above conditions. Since $1, i,j,ij$ form a basis for $B_0$, 
\begin{equation}
    1, i'=i, j'=\frac{2ij-s}{n}, i'j'=j'i',
\end{equation}
 where $n$ is the square root of the square-part of $(2ij-s)^2= s^2-4ab<0$, form a standard basis for $B_0$. Moreover, maximal orders of $K_{i'}=\Q(i')$ and $K_{j'}=\Q(j')$ form a Bass order in $B_0$ if and only if 
 \begin{equation}
     \text{either } a \not\equiv 1 \pmod{4} \text{ or } b \not\equiv 1 \pmod{4}.
 \end{equation}
In the latter case we may use $1, i'=2ij-t, j', i'j'$ for the standard basis). In particular, when $a, b \equiv 1 \pmod{4}$, maximal orders of $K_{i'}$ and $K_{j'}$ do not generate an order in $B_0$. 
\end{tbox}
\end{extra}

\begin{extra}
    \begin{tbox}
    Let $t\neq 0$ be an integer. Given a standard basis 
    \begin{equation}
        1, i^2=d_1, j^2=d_2, k=ij=ji
    \end{equation}
    for $B_0$, where $d_1, d_2<0$ are square-free integers, when $\left|\frac{t^2-d_2}{4d_1^2}\right| = \frac{a}{b}$ with $a,b>0$ and $\gcd(a,b)=1$. we get
    \begin{equation}
        1, i'=i, j'=\frac{b}{\sqrt{mn}} \cdot\frac{ti+k}{2d_1}, i'j',
    \end{equation}
   where $m$ is the square-part of $a$ and $n$ is the square-part of $b$, form a non-standard basis such that
    \begin{equation}
        j'^2 =\frac{b^2}{mn} \cdot\frac{t^2-d_2}{4d_1^2}=\frac{b^2}{mn}\cdot\left(-\frac{a}{b}\right)=-\frac{ab}{mn}<0
    \end{equation}
is a square-free integer and    
    \begin{equation}
        i'j'+j'i'=s=\frac{bt}{\sqrt{mn}}.
    \end{equation}
\end{tbox}
\end{extra}

Let $1, i, j, k$ be a standard basis for $B_0$ such that 
\begin{equation} \label{standard basis-first}
    i^2=d_i, \quad j^2=d_j, \quad k=ij=-ji,
\end{equation}
where $d_i, d_j < 0 $ are square-free integers. For the rest of the paper, $K_u=\Q(u)\subseteq B_0$, $u \in \{i,j\}$ denote the quadratic subfields given by the standard basis $1,i,j,k$ satisfying the relations in \eqref{standard basis-first}. This is equivalent to say $K_i$ and $K_j$ are perpendicular quadratic subalgebras generating $B_0$ as they give orthogonal basis with respect to the symmetric bilinear form
\begin{equation}
    x,y \mapsto \frac{1}{2}\trd(x\overline{y})
\end{equation}
on $B_0$. 

\begin{details}
\begin{tbox}
    See the following: \href{https://arxiv.org/pdf/2503.03478}{page 3}
\end{tbox}
\end{details}

Let
\begin{equation}
    \omega_u = \begin{cases}
        \frac{ 1 + u}{2} & \text{if} \ d_u\equiv 1\pmod 4, \\
        u & \text{if} \ d_u \not\equiv 1\pmod 4
    \end{cases}
\end{equation}
so that $\calO_{K_u}=\Z[\omega_{u}]$ is the maximal order of $K_u$. Let $d_{K_u}$ be the fundamental discriminant of $K_u$

As we have $s:=\trd(ij)=0$, by Proposition \ref{condition for simultaneous embeddings - maximal orders}, the maximal orders $\calO_{K_i}$ and $\calO_{K_j}$ have simultaneous embeddings into $B_0$ and their images generate a Bass order in $B_0$ if and only if
\begin{equation}\label{condition1}
    d_u\not\equiv 1\pmod 4 \quad \text{for at least one of} \ u=i,j.
\end{equation}
 or equivalently,
\begin{equation}
    d_{K_u} \equiv 0 \pmod{4} \quad \text{for at least one of} \ u=i,j.
\end{equation}

For the rest of the section, we will assume that condition \ref{condition1} holds for $K_i$ and $K_j$. 

Let $\mathfrak{O}_{i, j}$ denote the Bass order in $B_0$ generated by $\mathcal{O}_{K_i}$ and $\mathcal{O}_{K_j}$. The following lemma would be useful for computing the discriminant of $\mathfrak{O}_{i, j}$ from its generators $\omega_i, \omega_j$.

\begin{lemma} \label{disriminant-formula}
Let $\mathfrak{O} \subseteq B_{p, \infty}$ be an order generated by two non-commuting elements $\alpha_1, \alpha_2$. The reduced discriminant of $\mathfrak{O}$ is given by 
\begin{equation}\discrd (\mathfrak{O}) = \frac{ d_1 d_2-(t_1 t_2-2t)^2 }{4},\end{equation}
where $d_n$ is the discriminant of the quadratic order $\Z[\alpha_n]$,  $t_n=\trd(\alpha_n),$ and $t= \trd(\alpha_1\alpha_2)$ for $n=1,2$.
\end{lemma}
\begin{proof}
See \cite[Exercise 15.9]{Voi21} or \cite[Proposition 4]{Ler24} which is referring to \cite[Proposition 83]{Koh96}.
\end{proof}


Now we compute the local embedding number of the Bass order $\mathfrak{O}_{i, j}$.
\begin{lemma} \label{number of local roots} $e_\ell(\mathfrak{O}_{i, j})=1, 2 $ for any prime $\ell$.\end{lemma}
\begin{proof} Since $e_\ell(\mathfrak{O}_{i,j})=1$ when $\ell=p$, assume $\ell \neq p$. Under the condition \ref{condition1}, we always have
\begin{equation}
\trd(\omega_i)\trd(\omega_j)=\trd(\omega_i\omega_j) = 0.
\end{equation}
 Hence, by Lemma \ref{disriminant-formula},
\begin{equation}\label{discriminant}\discrd (\mathfrak{O}_{i, j}) = \frac{ d_{K_i} d_{K_j}}{4}.\end{equation}
We  compute the Eichler symbol ($\mathfrak{O}_{i,j}/\ell)$ using Proposition \ref{Eichler Symbol Computation} and verify that $e_\ell(\mathfrak{O}_{i,j})=1, 2$ using Proposition \ref{Formula for local embedding number}.

Suppose $\ell \neq 2$. Then $\ell^2 \nmid d_{K_i}, d_{K_j}$.
\begin{itemize}
        \item If $\ell \nmid \discrd(\mathfrak{O}_{i,j})$, then $\ell$ is unramified in both fields $K_i$, $K_j$ and $e_\ell(\mathfrak{O}_{i,j})=1$.
        \item If $\ell \parallel \discrd(\mathfrak{O}_{i,j})$, without loss of generality, assume $\ell \mid d_{K_i}$ and $\ell \nmid d_{K_j}$.   
        Then $\ell$ is ramified in $K_i$ and split in $K_j$. Also as $\ell \not= p$, $(\mathfrak{O}_{i,j}/\ell)=1$. Hence $e_\ell(\mathfrak{O}_{i,j})=\nu_\ell(\discrd(\mathfrak{O}_{i,j}))+1=2$.
        \item If  $\ell^2 \mid \discrd(\mathfrak{O}_{i,j})$, then $\ell$ is ramified in the both fields $K_i, K_j$. Hence, $(\mathfrak{O}_{i,j}/\ell)=0$ and $e_\ell(\mathfrak{O}_{i,j})=2$.
\end{itemize}
Suppose $\ell=2$. Recall that $d_{K_u} = 4d_u$ for at least one of $u=i,j$.
\begin{itemize}
        \item If $d_{K_i}=4d_i$ and $d_{K_j}=4d_j$, then $\ell$ is ramified in both fields $K_i, K_j$ and $\ell^2 \mid \discrd(\mathfrak{O}_{i,j})$. Hence, $(\mathfrak{O}_{i,j}/\ell)=0$ and $e_\ell(\mathfrak{O}_{i,j})=2$.
        \item In the other case, $\discrd(\mathfrak{O}_{i,j})=d_id_j$. Since $d_i, d_j$ are square-free, $\ell^3 \nmid \discrd(\mathfrak{O}_{i,j})$.  Without loss of generality, assume $d_{K_i}=4d_i$ and $d_{K_j}=d_j$. 
        \begin{itemize}
            \item If $\ell \nmid \discrd(\mathfrak{O}_{i,j})$, then $\ell$ is ramified in $K_i$ and unramified in $K_j$. Also, $e_\ell(\mathfrak{O}_{i,j})=1$.
            \item If $\ell \parallel \discrd(\mathfrak{O}_{i,j})$, then $\ell$ divides exactly one of $d_i$ and $d_j$. If $\ell \mid d_i$, then $\ell$ is ramified in $K_i$ and split in $K_j$. If $\ell \mid d_j$, then $\ell$ is ramified in both fields $K_i, K_j$. In both cases, $(\mathfrak{O}_{i,j}/\ell)=1$. Hence, $e_\ell(\mathfrak{O}_{i,j})=\nu_\ell(\discrd(\mathfrak{O}_{i,j}))+1=2$. 
            \item If $\ell^2 \mid \discrd(\mathfrak{O}_{i,j})$, then $\ell$ is ramified in both fields $K_i, K_j$. Hence, $(\mathfrak{O}_{i,j}/\ell)=0$ and $e_\ell(\mathfrak{O}_{i,j})=2$.
        \end{itemize}
\end{itemize}
\end{proof}

\begin{corollary}
    The number of local roots in a connected component of $G_{K_i,K_j}(p, \ell)$ is $e_\ell(\mathfrak{O}_{i,j})$. The number of global roots in $G_{K_i,K_j}(p, \ell)$ is $e(\mathfrak{O}_{i,j})$.
\end{corollary}
\begin{proof}
    Let $(E_0,\jmath_0)/\cong$ be a vertex in $G_{K_i,K_j}(p, \ell)$, $\mathfrak{O}_0=\End(E_0)$, and $B_0=\End^0(E_0)$. By Proposition \ref{connected component and local maximal orders}, every vertex in the connected component of $(E_0,\jmath_0)/\cong$ is uniquely represented by a maximal order in $S_\ell(\mathfrak{O}_0)$. Let $(E,\jmath)/\cong$ be a vertex in the connected component of $(E_0,\jmath_0)/\cong$ and $\mathfrak{O}$ be the maximal order in $S_\ell(\mathfrak{O}_0)$ representing $(E,\jmath)/\cong$. $(E,\jmath)$ is a local root if and only if $\mathfrak{O}$ locally contains $\mathfrak{O}_{i,j}$ at $\ell$. Hence, the number of local roots in the connected component of $(E_0,\jmath_0)/\cong$ is given by the local embedding number $e_\ell(\mathfrak{O}_{i,j})$. It follows that the number of global roots in $G_{K_i, K_j}(p,\ell)$ is given by the global embedding number $e(\mathfrak{O}_{i,j})$. 
\end{proof}

\begin{example} \label{root example}
    Consider the standard basis $1, i, j, k$ for $B_0$ such that
\begin{equation}
    i^2= -q, \quad j^2 = -p, \quad k = ij=-ji,
\end{equation}
where $q$ is given as in (\ref{Pizer basis}). Let $K_i=\Q(i)$ and $K_j=\Q(j)$ be quadratic subfields of $B_0$ and $\mathfrak{O}_{i,j}$ be the Bass order generated by the maximal orders of $K_i$ and $K_j$.
\begin{enumerate}
    \item[(i)] If $p \equiv 3 \pmod 4$ and $q=1$, then $\discrd(\mathfrak{O}_{i,j})=p$ and $\mathfrak{O}_{i,j}=\Z\langle i, \frac{1+j}{2}\rangle$ is a maximal order. The local embedding numbers are $e_\ell(\mathfrak{O}_{i,j})=1$ for any prime $\ell$. The global embedding number is $e(\mathfrak{O}_{i,j})=1$.
    \item[(ii)] If $p \equiv 5 \pmod{8}$ and $q=2$, then $\discrd(\mathfrak{O}_{i,j})=8p$ and $\mathfrak{O}_{i,j}=\Z\langle i,j\rangle$ is a non-maximal Bass order. The local embedding numbers are $e_\ell(\mathfrak{O}_{i,j})=1$ for any prime $\ell \neq 2$ and $e_\ell(\mathfrak{O}_{i,j})=2$ for $\ell=2$. The global embedding number is $e(\mathfrak{O}_{i,j})=2$. 
    \item[(iii)] If $p\equiv 1 \pmod 8$ and $q \equiv 3 \pmod 4$, then $\discrd(\mathfrak{O}_{i,j})=pq$ and $\mathfrak{O}_{i,j}=\Z\langle \frac{1+i}{2}, j \rangle$ is a non-maximal Bass order. The local embedding numbers are $e_\ell(\mathfrak{O}_{i,j})=1$ for any prime $\ell \neq q$ and $e_\ell(\mathfrak{O}_{i,j})=2$ for $\ell=q$. The global embedding number is $e(\mathfrak{O}_{i,j})=2$.
\end{enumerate}
Maximal orders in $B_0$ containing $\mathfrak{O}_{i,j}$ are explicitly given in Proposition \ref{roots}. 
\end{example}

\begin{details}
    \begin{tbox}
\[
\begin{tikzpicture}[scale=2.5,
  roundnode/.style={
    draw=none,       
    rectangle,       
    inner sep=2pt,   
    align=center
  }]

\def\L{0.866}

\coordinate (E0) at (0,0);
\coordinate (E1) at (0,-\L);
\coordinate (E2) at ({\L*sin(60)},{\L*cos(60)});
\coordinate (E3) at ({-\L*sin(60)},{\L*cos(60)});

\node[roundnode] (N0) at (E0) {$(E, \mathfrak{O}_, \mathfrak{O}_\ell$)};
\node[roundnode] (N1) at (E1) {$(E_1, \mathfrak{O}_1, (\mathfrak{O}_1)_\ell)$};
\node[roundnode] (N2) at (E2) {$(E_2, \mathfrak{O}_2, (\mathfrak{O}_2)_\ell)$};
\node[roundnode] (N3) at (E3) {$(E_3, \mathfrak{O}_3, (\mathfrak{O}_3)_\ell)$};

\draw[-, shorten >=4pt, shorten <=4pt] (N0) -- (N1);
\draw[-, shorten >=4pt, shorten <=4pt] (N0) -- (N2);
\draw[-, shorten >=4pt, shorten <=4pt] (N0) -- (N3);

\draw[dotted, thick, shorten >=4pt] (N2) -- ++(60:0.6);
\draw[dotted, thick, shorten >=4pt] (N2) -- ++(30:0.6);

\draw[dotted, thick, shorten >=4pt] (N3) -- ++(120:0.6);
\draw[dotted, thick, shorten >=4pt] (N3) -- ++(150:0.6);

\draw[dotted, thick, shorten >=4pt] (N1) -- ++(-45:0.6);
\draw[dotted, thick, shorten >=4pt] (N1) -- ++(-135:0.6);

\end{tikzpicture}
\]
\end{tbox}
\end{details}

\section{Structure of double-oriented isogeny graphs} \label{Structure of double-oriented isogeny graphs}
\subsection{Simultaneous ascension}
We show that one can ascend simultaneously in double-oriented isogeny graphs. 

Recall that we defined the quadratic fields $K_i=\Q(i)$ and $K_j=\Q(j)$ for a standard basis $1, i, j, k$ for $B_0$
such that 
\begin{equation} \label{standard basis}
    i^2=d_i, \quad j^2=d_j, \quad k=ij=-ji,
\end{equation}
where $d_i, d_j < 0 $ are square-free integers. Let $\mathfrak{O}_{i,j}$ be the Bass order generated by the maximal quadratic orders $\calO_{K_i}$ and $\calO_{K_j}$.

For $u \in \{i, j\}$, let $\calO_u = \Z[\alpha_u]$ be an order in $K_u$ with the conductor $f_u$. We denote by $\tilde{\calO}_u$ the next order in the chain
\begin{equation*}
\cdots \subset  \Z+\ell\calO_u \subset  \calO_u \subset \Z+\ell^{-1}\calO_u \subset \cdots \subset \Z+\ell^{-\nu_\ell(f_u)}\calO_u 
\end{equation*}
that is strictly larger than $\calO_u$, under the ordering given by the conductor of the quadratic order, unless $\calO_u$ is already $\ell$-fundamental in $K_u$, in which case we set $\tilde{\calO}_u = \calO_u$.

\begin{proposition}
\label{joint-ascend-general}
Let $(E,\jmath)/\cong$ be a vertex in $G_{K_i, K_j}(p,\ell)$ where $(E,\jmath)$ is primitively $\calO_i, \calO_j$-oriented. Consider the connected component of $(E,\jmath)/\cong$ in $G_{K_i, K_j}(p,\ell)$. The following hold:
\begin{enumerate}
    \item[(a)] If both orders $\calO_i, \calO_j$ are $\ell$-fundamental and $\ell$ is split in them, then $(E, \jmath)/\cong$ is the unique local root in the connected component.
    \item[(b)] If one of the orders $\calO_i, \calO_j$ is $\ell$-fundamental and $\ell$ is inert in it, then $(E,\jmath)/\cong$ is the unique local root in the connected component.
    \item[(c)] In other cases, there is a unique $K_i,K_j$-oriented isogeny $\varphi: (E,\jmath) \rightarrow (\tilde{E}, \tilde{\jmath})$ of degree $\ell$ up to equivalence such that $(\tilde{E},\tilde{\jmath})$ has orientations by $\tilde{\calO}_i, \tilde{\calO}_j$, and there are two local roots in the connected component.
\end{enumerate}
\end{proposition}

Proposition \ref{joint-ascend-general} implies that given a vertex $(E,\jmath)/\cong$ in $G_{K_i, K_j}(p,\ell)$, we can always simultaneously ascend to the local roots in the connected component of $(E,\jmath)/\cong$. Also, if there are two local roots in the connected component, then they are neighbors.

\begin{proof}[Proof of Proposition \ref{joint-ascend-general}(a)(b)]
By Corollary~\ref{double oriented curves to maximal orders}, there is a maximal order $\mathfrak{O}$ in $B_0$ which corresponds to $(E,\jmath)/\cong$.

(a) Suppose $\calO_i, \calO_j$ are $\ell$-fundamental and $\ell$ splits in both of $\calO_i$ and $\calO_j$. Since $\mathfrak{O}$ contains $\calO_i$ and $\calO_j$, we have that $\mathfrak{O} \supseteq \mathfrak{O}_{i,j}$. As seen in the proof of Proposition \ref{number of local roots}, this case can only happen when $\ell \neq 2$ and $\ell \nmid d_{K_i}, d_{K_j}$, in which case 
\begin{equation}
    e_\ell(\mathfrak{O}_{i,j})=1.
\end{equation}

(b) Let $\mathfrak{O}'$ be any order in $B_0$ locally containing $\mathcal{O}_{K_i} \otimes \Z_\ell$, i.e., $\mathfrak{O}' \otimes \Z_\ell$ is Bass. If $\ell \nmid \disc(\mathfrak{O}')$, then $\mathfrak{O}_{\ell}'$ is maximal and $e_\ell(\mathfrak{O}')=1$. If $\ell \parallel \disc(\mathfrak{O}')$, then by Proposition \ref{Eichler Symbol Computation}(a), $\ell$ is split or ramified in $K_i$, which contradicts our assumption that $\ell$ is inert in $K_i$. If $\ell^2 \mid \disc(\mathfrak{O}')$, then Proposition \ref{Eichler Symbol Computation}(b) implies that the Eichler symbol satisfies
\[
\left( \frac{\mathfrak{O}'}{\ell} \right) = -1,
\]
and $e_\ell(\mathfrak{O}')=1$ by Proposition \ref{Formula for local embedding number}. 
Thus, the local embedding number of $\mathfrak{O}'$ is always
\begin{equation}
\label{local-1}
    e_\ell(\mathfrak{O}') = 1
\end{equation}
and there is a unique local maximal order in $B_0 \otimes \Q_\ell$ containing $\mathcal{O}_{K_i} \otimes \Z_\ell$. Now suppose that $\mathcal{O}_{j}$ is not $\ell$-fundamental but $\mathcal{O}_i$ is $\ell$-fundamental. Then for an order $\mathfrak{O}''$ in $B_0$ locally containing $\mathcal{O}_{K_i} \otimes \Z_\ell$, but not $\mathcal{O}_{K_j} \otimes \Z_\ell$,
\begin{equation}
    e_\ell(\mathfrak{O}'') > 1
\end{equation}
as there are two distinct maximal orders locally, namely $\mathfrak{O} \otimes \Z_\ell$ and $\mathfrak{O}_{i,j} \otimes \Z_\ell$, containing $\mathfrak{O}'' \otimes \Z_\ell$, which is a contradiction to \eqref{local-1}. Therefore, $\mathcal{O}_{j}$ must also be $\ell$-fundamental.
\end{proof}

\begin{extra}
    {\color{red}Since $\ell$ is inert in $\mathcal{O}_{K_i}$ and $\ell^2 \mid \discrd(\mathfrak{O}_{i,j})$, by Proposition \ref{Eichler Symbol Computation}(b), $\ell$ is inert in $\mathcal{O}_{K_j}$ as well.}
\end{extra}

Before proving part (c) of Proposition \ref{joint-ascend-general}, we list a few lemmas.

\begin{lemma} \label{ideal norm}
    Let $\mathfrak{O}$ be a maximal order in $B_0$ and $\alpha \in \mathfrak{O}$ be an element such that $\ell \mid \nrd(\alpha)$. Then $I=\mathfrak{O}\ell + \mathfrak{O}\alpha$ is an integral left $\mathfrak{O}$-ideal of reduced norm $\ell$. Conversely, any integral left $\mathfrak{O}$-ideal $I$ can be generated by its reduced norm and an element $\alpha \in \mathfrak{O}$ such that $\nrd(I) \mid \nrd(\alpha)$ as a left $\frakO$-ideal. 
\end{lemma}
\begin{proof}
If $I$ is generated by $\{\alpha_1, \cdots, \alpha_n\}$ for some $n \geq 4$ as a $\Z$-module, then $\nrd(I)$ is generated by $\nrd(\alpha_i)$ and $\nrd(\alpha_i+\alpha_j)-\nrd(\alpha_i)-\nrd(\alpha_j)$ for $i,j =1, \cdots, n$ by the proof of \cite[Lemma 16.3.2]{Voi21}. Let $\{\beta_1, \beta_2, \beta_3, \beta_4\}$ be a basis for $\mathfrak{O}$. Then $I$ is generated by $\beta_i\ell$ and $\beta_i\alpha$ for $i=1, \cdots, 4$. One can check that $\ell$ divides all of $\nrd(\beta_i\ell), \nrd(\beta_i\alpha), \nrd(\beta_i\ell+\beta_j\ell), \nrd(\beta_i\ell+\beta_j\alpha)$, and $\nrd(\beta_i\alpha+\beta_j\alpha)$ for $i, j=1,\cdots, 4$. Hence, $\ell \mid \nrd(I)$. Also, since $\ell \in I$ and $\nrd(I)=\gcd(\{\nrd(\alpha) : \alpha \in I\})$, $\nrd(I) \mid \ell$. 

\begin{details}
\begin{tbox}
For the second last divisibility statement, use $\nrd(\alpha)=\alpha\overline{\alpha}$.
\end{tbox}
\end{details}

The converse statement is \cite[Exercise 16.6]{Voi21}.
\end{proof}

For $\alpha \in \End(E) \setminus \Z$, consider an order $\mathcal{O}=\Z[\alpha]$ in $K=\Q(\alpha)$ with the conductor $f$. Let $\lambda_1, \lambda_2 \in \F_{\ell^2}$ be the eigenvalues of $\alpha$ acting on $E[\ell]$, i.e., the roots of $c_\alpha(x)$. 
  We have $\disc(\calO)=\trd(\alpha)^2-4\nrd(\alpha)$ and its reduction modulo $\ell$ coincides with  the discriminant of $c_\alpha(x)$.
  
  In case $\calO$ is not $\ell$-fundamental, both eigenvalues are zero as $\ell \mid f$ implies $c_{\alpha}(x)=x^2$. Otherwise, $\calO$ is $\ell$-fundamental and we have either $\disc(\calO)$ is a quadratic residue modulo $\ell$ and $c_\alpha(x)$ has two non-zero distinct roots in $\F_\ell$, 
 $\ell$ divides the discriminant $\disc(\calO)$ and $c_\alpha(x)$ has a non-zero double root, or $c_\alpha(x)$ has no solution in $\F_\ell$. Hence, we have the following cases:
\begin{enumerate}
    \item[(i)]$\lambda_1=\lambda_2=0$ if and only if $\calO$ is not $\ell$-fundamental,
    \item[(ii)] $\lambda_1 \neq \lambda_2$ are non-zero elements in $\F_\ell$ if and only if $\ell$ splits in $K$,
    \item[(iii)] $\lambda_1= \lambda_2$ is a non-zero element in $\F_\ell$ if and only if $\ell$ is ramified in $K$,
    \item[(iv)] $\lambda_1, \lambda_2 \in \F_{\ell^2} \setminus \F_\ell$ if and only if $\ell$ is inert in $K$.
\end{enumerate}

\begin{lemma} \label{Constructing ascending/horizontal isogeny}
Let $E/\Fbar_p$ be a supersingular elliptic curve, $\ell \neq p$ be a prime, and $\mathfrak{O}=\End(E)$. Suppose $\alpha \in \mathfrak{O}$ is $\ell$-suitable and $\ell$-primitive. Consider an order $\mathcal{O}=\Z[\alpha]$ in $K=\Q(\alpha)$ and an associated $K$-orientation $\iota$ on $E$ of $\alpha$. Let $\lambda \in \mathbb{F}_{\ell^2}$ be an eigenvalue of $\alpha$ acting on $E[\ell]$. Define a left integral $\mathfrak{O}$-ideal
\begin{equation}
        I(\alpha, \lambda):=\mathfrak{O}\ell+\mathfrak{O}(\alpha-\lambda).
\end{equation}
Let $I=I(\alpha, \lambda)$. We have the following properties:
\begin{enumerate}
    \item[(a)] If $\calO$ is not $\ell$-fundamental, then $E[I]$ is the kernel of the $K$-ascending $\ell$-isogeny from $(E,\iota)$.
    \item[(b)]  If $\mathcal{O}$ is $\ell$-fundamental and $\ell$ is not inert in $\mathcal{O}$, then $E[I]$ is the kernel of a $K$-horizontal $\ell$-isogeny from $(E,\iota)$. Moreover if $\ell$ is split in $\calO$, then for the non-zero distinct eigenvalues $\lambda_1, \lambda_2 \in \F_\ell$ of $\alpha$ acting on $E[\ell]$, we have that $E[I_1]$ and $E[I_2]$ are the kernels of distinct $K$-horizontal isogenies from $(E,\iota)$, where $I_1=I(\alpha, \lambda_1)$ and $I_2=I(\alpha, \lambda_2)$.
\end{enumerate}
\end{lemma}
\begin{proof}

Suppose $\calO$ is not $\ell$-fundamental. Then eigenvalues of $\alpha$ acting on $E[\ell]$ are zero. Let $\a =  \calO \ell +  \calO\alpha$. Then $\a$ is a non-invertible $\calO$-ideal with $N(\a)=\ell$ and $E[\a]$ is the kernel of a $K$-ascending $\ell$-isogeny by \cite[Proposition 4.8]{ACL+23}. Since $\ell$ divides the norms $N(\alpha)=\nrd(\alpha)$, $I(\alpha, 0)$ is a left integral $\mathfrak{O}$-ideal of reduced norm $\ell$ by Lemma \ref{ideal norm}. It follows that $E[I]=E[\a]$
since $E[I] \subseteq E[\a]$ and $\#E[I] = \#E[\a]=\ell$.

Suppose $\calO$ is $\ell$-fundamental and $\ell$ is not inert in $\calO$. There exists a non-zero eigenvalue $\lambda \in \F_\ell$ of $\alpha$ acting on $E[\ell]$. Let $\a=  \calO\ell+ \calO(\alpha-\lambda)$. Then $\a$ is an invertible $\calO$-ideal with $N(\a)=\ell$ and $E[\a]$ is the kernel of a $K$-horizontal $\ell$-isogeny by \cite[Proposition 4.8]{ACL+23}. In particular, $\ell$ divides the norms $N(\alpha-\lambda)=\nrd(\alpha-\lambda)$. Hence $I(\alpha,\lambda)$ is a left integral $\mathfrak{O}$-ideal of reduced norm $\ell$ by Lemma \ref{ideal norm}. It follows that $E[I]=E[\a]$
since $E[I] \subseteq E[\a]$ and $\#E[I] =\#E[\a]=\ell$. 

The last statement holds since $\calO\ell +\calO(\alpha - \lambda_1)$ and $ \calO\ell+\calO(\alpha - \lambda_2)$ give rise to the kernels of distinct $K$-horizontal isogenies \cite[Proposition 4.8]{ACL+23}.
\end{proof}

\begin{details}
    \begin{tbox}
        WLOG assume $\alpha=f \omega_K$. When $\calO$ is not $\ell$-fundamental, $\ell \mid f$ and
        \begin{center}
        $  N(\alpha)=f^2N(\omega_K)=$
            \begin{cases}
                \frac{f^2(1-d)}{4} & \text{if} \ d \equiv 1 \pmod 4 \\
              -f^2d& \text{if} \ d \not\equiv 1 \pmod 4
            \end{cases}
        \end{center}
Hence, $\ell \mid f \mid N(\alpha)$.
        
                When $\calO$ is $\ell$-fundamental, $\ell \nmid f$ and  
        \begin{center}
        $  N(\alpha-\lambda)=N(f\omega_K-\lambda)=$
            \begin{cases}
               \lambda^2-f\lambda+ \frac{f^2(1-d)}{4} & \text{if} \ d \equiv 1 \pmod 4 \\
              \lambda^2-f^2d& \text{if} \ d \not\equiv 1 \pmod 4
            \end{cases}
        \end{center}
Note that $\lambda$ is a root of the characteristic polynomial $x^2-f x +\frac{f^2(1-d)}{4}$ if $d \equiv 1 \pmod 4$ and the characteristic polynomial $x^2-f^2d$ if $d \not\equiv 1 \pmod 4$ of the action of $\alpha$ in $\F_\ell$. Hence, $\ell \mid N(\alpha-\lambda)$.
    \end{tbox}
\end{details}

Let $M_2(\F_\ell)$ be the ring of $2\times 2$ matrices with entries in $\F_\ell$. For any non-invertible elements $A$ and $B$ in $M_2(\F_\ell)$, we define an equivalence relation on them by
\begin{equation} \label{eqivalence matrix}
        A \sim B \text{ if } B=PA \text{ for an invertible } P \in M_2(\F_\ell).
\end{equation}

Then $A$ and $B$ generate the same left ideals in $M_2(\F_\ell)$. Let
\begin{equation}
   \omega = \begin{pmatrix}
    0 & 0 \\ 0 & 1 
\end{pmatrix} \quad \text{and} \quad \omega_x = \begin{pmatrix}
    1 & x \\ 0 & 0
\end{pmatrix} 
\end{equation}
for some $x \in \F_\ell$ which are non-invertible elements in $M_2(\F_\ell)$.

\begin{lemma} \label{ideal-matrix}
    The set of non-zero proper left ideals of $M_2(\F_\ell)$ is given by
\begin{equation}
    \{M_2(\F_\ell)\omega,  M_2(\F_\ell)\omega_x: x \in \F_\ell\}
\end{equation}
and this set is in bijection with the set of left $\mathfrak{O}$-ideals of reduced norm $\ell$. 
\end{lemma}
\begin{proof}
    We follow the proof of \cite[Theorem 5]{LOX20}.

    Note that a non-zero proper left ideal in $M_2(\F_\ell)$ is principally generated by an element in $\left\{ \omega \right\} \cup \left\{ \omega_x : x \in \F_\ell \right\}$.

    By Lemma \ref{ideal norm}, a left $\mathfrak{O}$-ideal of reduced norm $\ell$ is of the form $\mathfrak{O}\ell+\mathfrak{O}\alpha$ for an element $\alpha \in B_0^\times$ with $\ell \mid \nrd(\alpha)$. Then the isomorphism $\mathfrak{O} /\ell\mathfrak{O} \cong M_2(\F_\ell)$ in (\ref{Isomorphism of quotient and matrix ring}) induces a bijection
    \begin{equation}
\begin{array}{c@{\;}c@{\;}c} \label{ideal bijection}
\left\{ \text{left integral $\mathfrak{O}$-ideals of reduced norm $\ell$} \right\} 
& \longleftrightarrow & \left\{\text{non-zero proper left ideals of $M_2(\F_\ell)$}\right\}
 \\[1ex]
I=\mathfrak{O}\ell+\mathfrak{O}\alpha
& \mapsto & 
M_2(\F_\ell)M_\alpha, 
\end{array}
\end{equation}
where $M_\alpha$ is the matrix representation of an endomorphism $\alpha$ on $E[\ell]$ given in \eqref{Isomorphism of quotient and matrix ring}. The correspondence follows from the correspondence between left $\mathfrak{O}$-ideals of reduced norm $\ell$ and non-zero proper left ideals of $\mathfrak{O}/\ell \mathfrak{O}$. 
\end{proof}

\begin{extra}
    \begin{corollary}
    Let $\mathfrak{O}$ be a maximal order in $B_0$ and $\ell \neq p$ be a prime. A left $\mathfrak{O}$-ideal of reduced norm $\ell$ is integral.
\end{corollary}
\end{extra}

\begin{proof}[Proof of Proposition \ref{joint-ascend-general}(c)] 
By Corollary~\ref{double oriented curves to maximal orders}, there is a maximal order $\mathfrak{O}$ in $B_0$ which corresponds to $(E,\jmath)/\cong$.

Let $u \in \{i,j\}$. Suppose $\ell$ is not split in at least one of $K_i, K_j$ if $\calO_i, \calO_j$ are both $\ell$-fundamental, and suppose $\ell$ is not inert in both $K_i, K_j$. Let $f_u$ be the conductor of $\calO_u=\Z[\alpha_u]$. We may choose $\alpha_u=f_u\omega_u$. Let $\lambda_u$ be the eigenvalue of the action of $\alpha_u$ on $E[\ell]$. By Proposition \ref{Constructing ascending/horizontal isogeny}, 
    \begin{align}
        I_u&=I(\alpha_u, \lambda_u)
    \end{align}  
gives the kernel $E[I_u]$ of an isogeny which is either $K_u$-ascending or $K_u$-horizontal. We want to prove that $I_i=I_j$, which will imply that there is always at least one of simultaneously ascending, simultaneously horizontal, ascending-horizontal, or horizontal-ascending isogeny from $(E,\jmath)$. We will do so by showing that they correspond to the same proper left ideal of $M_2(\F_\ell)$ under the bijection in \eqref{ideal bijection}.

Consider the non-zero image $X_u$ of $\alpha_u$ under the isomorphism $\mathfrak{O}/\ell\mathfrak{O} \cong M_2(\F_\ell)$, i.e., $X_u=M_{\alpha_u}$. We first compute $X_u$.

\begin{details}
    \begin{tbox}
        \begin{enumerate}
    \item[(a)] $X_u$ is non-zero since $\alpha_u$ being $\ell$-primitive implies $\alpha_u \not \in \ell \mathfrak{O}$.
    \item[(b)] $X_u$ is not a multiple of the identity matrix $\id$ in $M_2(\F_\ell)$ since we assumed $\alpha_u$ generates a quadratic order.
    \item[(c)] $X_u:=M_{\alpha_u}$ for $u \in \{i,j\}$
\end{enumerate}
    \end{tbox}
\end{details}

If $d_u \equiv 1 \pmod 4$, then $\alpha_u = f_u\frac{1+u}{2}$. $X_u$ satisfies the characteristic polynomial $x^2-f_ux+\frac{f_u^2(1-d_u)}{4}$, so $X_u$ is of the form
\begin{equation} \label{matrix form pre-computation1}
    X_u = \begin{pmatrix}
   a & b \\ c & f_u-a
\end{pmatrix} \quad \text{with} \quad a^2-f_ua+bc+\frac{f_u^2(1-d_u)}{4}=0 \text{ in } \F_\ell. 
\end{equation}

\begin{details}
    \begin{tbox}
        Suppose 
\[
X = \begin{pmatrix} a & b \\ c & d \end{pmatrix} \in M_2(\mathbb{F}_\ell)
\]
We have 

\[
X^2 - f_u X + \frac{f_u^2(1 - d_u)}{4} I = 0
\]

\[
\begin{pmatrix} 
a^2 + bc & b(a+d) \\ 
c(a+d) & d^2 + bc 
\end{pmatrix}
- f_u 
\begin{pmatrix} 
a & b \\ 
c & d 
\end{pmatrix}
+ \frac{f_u^2(1 - d_u)}{4} I = 0
\]
Hence,
\[
a^2 - f_u a + bc+ \frac{f_u^2(1 - d_u)}{4} = d^2- f_ud+bc + \frac{f_u^2(1 - d_u)}{4} = 0
\]

\[
b(a+d - f_u) = c(a+d - f_u) = 0.
\]

Suppose \( a + d - f_u \neq 0 \). Then \( b = c = 0 \) as \( \mathbb{F}_\ell \) is an integral domain.

\begin{enumerate}
    \item If $a\neq d$, then \( a, d \) are roots of $x^2 - f_u x + \frac{f_u^2(1 - d_u)}{4} = 0$ (this polynomial has discriminant $f_u^2d_u>0$, so has distinct roots) , i.e., \( a + d = f_u \), which is a contradiction.
    \item If $a=d$, then $X=aI$, which is a contradiction.
\end{enumerate}
Hence, we must have $a+d=f_u$.
    \end{tbox}
\end{details}

If $d_u \not \equiv 1 \pmod 4$, then $\alpha_u=f_u u$. $X_u$ satisfies the characteristic polynomial $x^2-f_u^2d_u$, so $X_u$ is of the form
\begin{equation} \label{matrix form pre-computation2}
    X_u = \begin{pmatrix}
   a & b \\ c & -a
\end{pmatrix} \quad \text{with} \quad a^2+bc = f_u^2d_u \text{ in } \F_\ell.
\end{equation}

Let $\lambda \in \F_\ell$ be an eigenvalue of $\alpha_u$ acting on $E[\ell]$. From (\ref{matrix form pre-computation1}) and (\ref{matrix form pre-computation2}), one can check that $X_u$ takes one of the following forms:
\begin{equation} \label{matrix form}
    X_u = \begin{cases}
        \begin{pmatrix}
    f_u-ax-\lambda & f_ux-ax^2-2\lambda x \\ a & ax+\lambda
\end{pmatrix} & \text{if} \ X_u-\lambda \id \sim \omega_x \ \text{and} \ d_u \equiv 1 \pmod 4, \\
\begin{pmatrix}
    -ax-\lambda & -ax^2-2\lambda x \\ a & ax+\lambda
\end{pmatrix} & \text{if} \ X_u-\lambda \id \sim \omega_x \ \text{and} \ d_u \not\equiv 1 \pmod 4, \\
 \begin{pmatrix}
    \lambda & a \\ 0 & f_u-\lambda
\end{pmatrix} & \text{if} \ X_u-\lambda \id \sim \omega \ \text{and} \ d_u \equiv 1 \pmod 4, \\
\begin{pmatrix}
    \lambda & a \\ 0 & -\lambda 
\end{pmatrix} & \text{if} \ X_u-\lambda \id \sim \omega \ \text{and} \ d_u \not\equiv 1 \pmod 4
    \end{cases}
\end{equation}
for some $a \in \F_\ell$, where $a \neq 0$ if $\calO_u$ is not $\ell$-fundamental and $\sim$ is as in \eqref{eqivalence matrix}.

\begin{details}
\begin{tbox}
{\color{red} For $\calO_u$ not $\ell$-fundamental, the calculation in details doesn't seem to match (8.13): there is no $\lambda$ that comes out of the calculation.}
{\color{blue} When $\calO_u$ is not $\ell$-fundamental, $\lambda=f_u=0$ in $\F_\ell$ and the first two cases and the last two cases in (8.13) are combined. so there are two cases. When $\calO_u$ is $\ell$-fundamental, we have $4$ cases. So there are total $6$ cases. This is just a concise way to write them all together. In details, I went over all $6$ cases separately.}
\end{tbox}
\end{details}

\begin{details}
    \begin{tbox}
    Suppose $\calO_u$ is not $\ell$-fundamental. We want to show that 
    \begin{itemize}
    \item[1-(i)] If $X \sim \omega_x$, then
\begin{equation*}
    X =\begin{pmatrix}
    -cx & -cx^2 \\ c & cx
\end{pmatrix}
\end{equation*}
for some $c \neq 0$ in $\F_\ell$.
    \item[1-(ii)] If $X \sim \omega$, then
\begin{equation*}
    X = \begin{pmatrix}
    0 & b \\ 0 & 0
\end{pmatrix}
\end{equation*}
for some $b\neq 0$ in $\F_\ell$.
\end{itemize}

\begin{proof}
    Since $\ell \mid f_u$, i.e., $f_u=0$ in $\F_\ell$, $X=\begin{pmatrix}
        a & b \\ c & -a
    \end{pmatrix}$ for both cases: $d_u\equiv 1 \pmod 4$ and $d_u \not\equiv 1 \pmod 4$ from \eqref{matrix form pre-computation1}, \eqref{matrix form pre-computation2}.
    
1-(i) Suppose $X=P\omega_x$ for $P=\begin{pmatrix}
    u & v \\ w & z
\end{pmatrix}$ invertible.

$$\begin{pmatrix}
    a & b \\ c & -a
\end{pmatrix}=\begin{pmatrix}
    u & v \\ w & z
\end{pmatrix} \begin{pmatrix}
     1 & x  \\ 0 & 0
\end{pmatrix}=\begin{pmatrix}
    u & ux \\ w & wx
\end{pmatrix}$$

Then 
\begin{equation*}
\begin{split}
        u&=a \\
    b&=ax \\
    c&=w\\
    a &= -cx.
\end{split}
\end{equation*}
Hence, $b=-cx^2$.
So $X=\begin{pmatrix}
      - cx & -cx^2 \\ c & cx
\end{pmatrix}$.

1-(ii) Suppose $X=P\omega$ for $P=\begin{pmatrix}
    u & v \\ w & z
\end{pmatrix}$ invertible.
$$\begin{pmatrix}
    a & b \\ c & -a
\end{pmatrix}=\begin{pmatrix}
    u & v \\ w & z
\end{pmatrix} \begin{pmatrix}
     0 & 0  \\ 0 & 1
\end{pmatrix}=\begin{pmatrix}
    0 & v \\ 0 & z
\end{pmatrix}$$
Then 
\begin{equation*}
\begin{split}
        a&=0 \\
    b&=v \\
    c&=0\\
    z &= -a =0.
\end{split}
\end{equation*}
Hence, $X=\begin{pmatrix}
      0 & b \\ 0 & 0
\end{pmatrix}$.
\end{proof}

    \end{tbox}
\end{details}

\begin{details}
    \begin{tbox}

Suppose $\calO_u$ is $\ell$-fundamental. We want to show the following:

Suppose $d_u \equiv 1 \pmod 4$. $\lambda$ is a root of $x^2-f_ux+\frac{f_u^2(1-d_u)}{4} \in \F_\ell[x]$.
\begin{enumerate}
    \item[2-(i)] If $Y \sim \omega_x$, then
\begin{equation*}
    X = \begin{pmatrix}
    f_u-cx-\lambda & f_ux-cx^2-2\lambda x \\ c & cx+\lambda
\end{pmatrix}.
\end{equation*}

    \item[2-(ii)] If $Y \sim \omega$, then
\begin{equation*}
    X = \begin{pmatrix}
    \lambda & b \\ 0 & f_u-\lambda
\end{pmatrix}.
\end{equation*}
\end{enumerate}

Suppose $d_u \not \equiv 1 \pmod 4$. $\lambda$ is a root of $x^2-f^2_ud_u \in \F_\ell[x]$. 
\begin{enumerate}
    \item[2-(iii)] If $Y \sim \omega_x$, then
\begin{equation*}
    X = \begin{pmatrix}
    -cx-\lambda & -cx^2-2\lambda x \\ c & cx+\lambda
\end{pmatrix}.
\end{equation*}

    \item[2-(iv)] If $Y \sim \omega$, then
\begin{equation*}
    X = \begin{pmatrix}
    \lambda & b \\ 0 & -\lambda
\end{pmatrix}.
\end{equation*}
\end{enumerate}

\begin{proof}
     (i) Suppose $d_u \equiv 1 \pmod 4$ and $Y\sim \omega_x$.
     
        Suppose $Y=P\omega_x$ for $P=\begin{pmatrix}
    u & v \\ w & z
\end{pmatrix}$ invertible.

$$\begin{pmatrix}
    a-\lambda & b \\ c & f_u-a-\lambda
\end{pmatrix}=\begin{pmatrix}
    u & v \\ w & z
\end{pmatrix} \begin{pmatrix}
     1 & x  \\ 0 & 0
\end{pmatrix}=\begin{pmatrix}
    u & ux \\ w & wx
\end{pmatrix}$$
Then 
\begin{equation*}
\begin{split}
        u&=a-\lambda \\
    b&=(a-\lambda)x \\
    c&=w\\
    f_u-a-\lambda &= cx.
\end{split}
\end{equation*}
Hence, $a=f_u-cx-\lambda$ and $b=f_ux-cx^2-2\lambda x$.
So $X=Y+\lambda I =\begin{pmatrix}
    a & b \\ c & f_u-a
\end{pmatrix}$

  (ii) Suppose $d_u \equiv 1 \pmod 4$ and $Y \sim \omega$.
    
        Suppose $Y=P\omega_x$ for $P$ invertible.

$$\begin{pmatrix}
    a-\lambda & b \\ c & f_u-a-\lambda
\end{pmatrix}=\begin{pmatrix}
    u & v \\ w & z
\end{pmatrix} \begin{pmatrix}
     0 & 0  \\ 0 & 1
\end{pmatrix}=\begin{pmatrix}
    0 & v \\ 0 & z
\end{pmatrix}$$
Then 
\begin{equation*}
\begin{split}
        a&=\lambda \\
        c & =0
\end{split}
\end{equation*}
Hence $X=Y+\lambda I =\begin{pmatrix}
    a & b \\ c & f_u-a
\end{pmatrix}$
\end{proof}

    \end{tbox}
\end{details}

\begin{details}
    \begin{tbox}
(iii) Suppose $d_u \not\equiv 1 \pmod 4$ and $Y \sim \omega_x$. 

Suppose $Y=P\omega_x$ for $P$ invertible.

$$\begin{pmatrix}
    a-\lambda & b \\ c & -a-\lambda
\end{pmatrix}=\begin{pmatrix}
    u & v \\ w & z
\end{pmatrix} \begin{pmatrix}
     1 & x  \\ 0 & 0
\end{pmatrix}=\begin{pmatrix}
    u & ux \\ w & wx
\end{pmatrix}$$
Then 
\begin{equation*}
\begin{split}
        u&=a-\lambda \\
    b&=(a-\lambda)x \\
    c&=w\\
    -a-\lambda &= cx.
\end{split}
\end{equation*}
Hence, $a=-cx-\lambda$ and $b=-cx^2-2\lambda x$.
So $X=Y+\lambda I =\begin{pmatrix}
    a & b \\ c & -a
\end{pmatrix}$

(iv) Suppose $d_u \not\equiv 1 \pmod 4$ and $Y \sim \omega$. 

Suppose $Y=P\omega_x$ for $P$ invertible.

$$\begin{pmatrix}
    a-\lambda & b \\ c &-a-\lambda
\end{pmatrix}=\begin{pmatrix}
    u & v \\ w & z
\end{pmatrix} \begin{pmatrix}
     0 & 0  \\ 0 & 1
\end{pmatrix}=\begin{pmatrix}
    0 & v \\ 0 & z
\end{pmatrix}$$
Then 
\begin{equation*}
\begin{split}
        a&=\lambda \\
        c & =0
\end{split}
\end{equation*}
Hence $X=Y+\lambda I =\begin{pmatrix}
    a & b \\ c & -a
\end{pmatrix}$
    \end{tbox}
\end{details}

\begin{details}
    \begin{tbox}
        Recall that 
\begin{equation*}
    \alpha_u=\begin{cases}
        f_u \frac{1+u}{2} & \text{if} \ d_u \equiv 1 \pmod 4,\\
        f_u u & \text{if} \ d_u \not \equiv 1 \pmod{4}.
    \end{cases}
\end{equation*}
    \end{tbox}
\end{details}

Since $\lambda$ is an eigenvalue for $X_u$, we have that 
\begin{equation}
    Y_u=X_u-\lambda \id
\end{equation}
is non-invertible in $M_2(\F_\ell)$, and hence it generates the proper left ideal
\begin{equation}
    M_2(\F_\ell)Y_u
\end{equation}
of $M_2(\F_\ell)$ corresponding to $I(\alpha_u,\lambda)$. Moreover, if $\calO_u$ is $\ell$-fundamental and $\ell$ splits in $\calO_u$, there are non-zero distinct eigenvalues of $\lambda_1, \lambda_2 \in \F_\ell$ of $\alpha_u$ acting on $E[\ell]$ and
$X_u-\lambda_1 \id \not\sim X_u-\lambda_2 \id$.

\begin{details}
{\color{red} Why do we have the last inequivalence?}
{\color{blue} it's given in the proof of \cite{ACL+23} Proposition 4.8}
\end{details}

\begin{details}
    \begin{tbox}
        Recall that $f_u =0$ in $\F_\ell$ if and only if $\calO_u$ is not $\ell$-fundamental for $u \in \{1,2\}$ in which case all three equations become $X_iX_j+X_jX_i=0$.
    \end{tbox}
\end{details}
\begin{details}
    \begin{tbox}
    Suppose $\calO_u$ is $\ell$-fundamental.
    
    We show $Y=X-\lambda I$ is non-invertible.
    
    Suppose \( d_u \equiv 1 \pmod{4} \), 
\[
Y = \begin{pmatrix} a - \lambda & b \\ c & f_u - a - \lambda \end{pmatrix}
\]

The determinant of \( Y\) is:

\[
\det(Y) = - (a^2 - f_u a + bc) + \lambda^2 - f_u \lambda=0
\]

Since \( \lambda \) is a root of the equation:

\[
x^2 -f_u x + \frac{f_u^2(1 - d_u)}{4} = 0.
\]

Suppose \( d \not\equiv 1 \pmod{4} \)
\[
Y = \begin{pmatrix} a - \lambda & b \\ c & -a - \lambda \end{pmatrix}
\]

\[
\det(Y) = - (a^2 + bc) + \lambda^2 = 0
\]

Since \( \lambda \) is a root of the polynomial $x^2 - f_u^2 d_u$.
    \end{tbox}
\end{details}

Also, $X_i$ and $X_j$ satisfy one of the following relations:
\begin{equation} \label{matrix-equation}
    \begin{cases}
        X_iX_j+X_jX_i=0 & \text{if} \ d_i, d_j \not \equiv 1 \pmod{4},\\
        X_iX_j+X_jX_i=f_iX_j & \text{if} \ d_i \equiv 1 \pmod{4} \ \text{and} \ d_j\not\equiv 1 \pmod{4}, \\
        X_iX_j+X_jX_i=f_jX_i & \text{if} \ d_i \not\equiv 1 \pmod{4} \ \text{and} \ d_j\equiv 1 \pmod{4}.\\
    \end{cases}
\end{equation}

\begin{details}
\begin{tbox}
You can check (8.16) in hands as $\alpha_u=f_uu$ or $f_u(1+u)/2$. I didn't use maple.
\end{tbox}
\end{details}

Next, using (\ref{matrix form}) and (\ref{matrix-equation}), we compare the left $\mathfrak{O}$-ideals of the form $I(\alpha_u, \lambda)$, which correspond to an ascending isogeny or a horizontal isogeny.

\begin{details}
    \begin{tbox}
        To see $Y_u$ is non-invertible, use that $\lambda=f_i=0$ in $\F_\ell$  if $\calO_u$ is not $\ell$-fundamental. If $\calO_u$ is $\ell$-fundamental, just compute the determinant.
    \end{tbox}
\end{details}

In the calculations below, recall from the preamble to Lemma~\ref{Constructing ascending/horizontal isogeny} that $\lambda = 0$ when the order $\calO$ under consideration is not $\ell$-fundamental.

Case 1. Both $\calO_i, \calO_j$ are not $\ell$-fundamental:
\begin{equation}
    I_u=I(\alpha_u, 0)
\end{equation}
 gives the kernel of a $K_u$-ascending $\ell$-isogeny, where
\begin{equation}
    M_2(\F_\ell)X_u
\end{equation}
is the corresponding left ideal of $M_2(\F_\ell)$. We show that $I_i=I_j$ by showing that $X_i \sim X_j$.

\begin{itemize}
    \item Suppose $X_i \sim \omega_x$ and $X_j \sim \omega$. By (\ref{matrix form}),
\begin{equation}
        X_i=\begin{pmatrix}
    -ax & -ax^2 \\ a & ax
\end{pmatrix} \text{ and } X_j=\begin{pmatrix}
    0 & b \\ 0 & 0
\end{pmatrix}
\end{equation}
for some $a, b \neq0 $ in $\F_\ell$. Solving (\ref{matrix-equation}), we get $ ab=0$. This is a contradiction since $a,b \neq 0$.
    \item Suppose $X_i \sim \omega_x$ and $X_j \sim \omega_y$ for some $x\neq y$ in $\F_\ell$. By (\ref{matrix form}), 
\begin{equation}
       X_i=\begin{pmatrix}
    -ax & -ax^2 \\ a & ax
\end{pmatrix} \text{ and } X_j=\begin{pmatrix}
    -by & -by^2 \\ b & by
\end{pmatrix}
\end{equation}
for some $a, b \neq0 $ in $\F_\ell$. Solving (\ref{matrix-equation}), we get $ab(x-y)^2=0$. This is a contradiction since $a,b \neq 0$ and $x\neq y$.
\end{itemize}

 In other cases, $X_i \sim X_j$.\\

Case 2. $\calO_i$ is $\ell$-fundamental, $\ell$ is ramified in $\calO_i$, and $\calO_j$ is not $\ell$-fundamental:

Let $\lambda$ be the unique non-zero eigenvalue of $\alpha_i$ acting on $E[\ell]$. Let $Y_i= X_i-\lambda \id$. 
\begin{equation}
        I_i=I(\alpha_i, \lambda), I_j =(\alpha_j, 0)
\end{equation}
 give the kernel of a $K_i$-horizontal $\ell$-isogeny and the kernel of a $K_j$-asending $\ell$-isogeny respectively, where 
 \begin{equation}
     M_2(\F_\ell)Y_i, M_2(\F_\ell)X_j
 \end{equation}
are the corresponding left ideals of $M_2(\F_\ell)$ respectively. We show that $I_i=I_j$ by showing $Y_i \sim X_j$. 

(i) Let $d_i \equiv 1 \pmod 4$. Then $\lambda$ is a non-zero double root of 
\begin{equation}
    x^2-f_ix+\frac{f_i^2(1-d_i)}{4} \in \F_\ell[x],
\end{equation}
by the preamble to Lemma~\ref{Constructing ascending/horizontal isogeny}
so $2\lambda = f_i \neq 0$ in $\F_\ell$. In particular, $\ell \neq 2$.

\begin{details}
    \begin{tbox}
        $f_i \neq 0$ as $\calO_i$ is $\ell$-fundamental. Hence, $f_i =2\lambda \neq 0$ in $\F_\ell$. $\ell \neq 2$.
    \end{tbox}
\end{details}

\begin{itemize}
    \item Suppose $Y_i \sim \omega_x$ for some $x $ in $\F_\ell$ and $X_j \sim \omega$. By (\ref{matrix form}),
\begin{equation}
    X_i = \begin{pmatrix}
    -ax +\lambda& -ax^2 \\ a & ax+\lambda
\end{pmatrix} \text{ and } X_j = \begin{pmatrix}
    0 & b \\ 0 & 0
\end{pmatrix}
\end{equation}
for some $a, b \neq 0$ in $\F_\ell$. Solving (\ref{matrix-equation}), we get $ab=0$. This is a contradiction since $a, b \neq 0$.
    \item Suppose $Y_i \sim \omega_x$ and $X_j \sim \omega_y$ for some $x\neq y$ in $\F_\ell$. By (\ref{matrix form}),
\begin{equation}
    X_i = \begin{pmatrix}
    -ax +\lambda& -ax^2 \\ a & ax+\lambda
\end{pmatrix} \text{ and } X_j =  \begin{pmatrix}
   - by & -by^2 \\ b & by
\end{pmatrix}
\end{equation}
for some $a, b \neq 0$ in $ \F_\ell$.  Solving (\ref{matrix-equation}), we get $-ab(x-y)^2=0$. This is a contradiction since $a,b \neq 0$ and $x\neq y$.
    \item Suppose $Y_i \sim \omega$  and $X_j \sim \omega_x$ for some $x $ in $\F_\ell$. By (\ref{matrix form}),
\begin{equation}
    X_i = \begin{pmatrix}
    \lambda&  a\\ 0& \lambda
\end{pmatrix} \text{ and } X_j = \begin{pmatrix}
    -bx & -bx^2 \\ b & bx
\end{pmatrix}
\end{equation}
 for some $a, b \neq 0$ in $\F_\ell$. Solving (\ref{matrix-equation}), we get $ab=0$. This is a contradiction since $a, b \neq 0$.
\end{itemize}

In other cases, $Y_i \sim X_j$. 

(ii) Let $d_i \not\equiv 1 \pmod 4$. Then $\lambda$ is a non-zero double root of 
\begin{equation}
    x^2-f_i^2d_i \in \F_\ell[x],
\end{equation}
so $2\lambda = 0$ in $\F_\ell$ (see case (i) for details). 

\begin{itemize}
    \item Suppose $Y_i \sim \omega_x$ for some $x $ in $\F_\ell$ and $X_j \sim \omega$. By (\ref{matrix form}),
\begin{equation}
    X_i = \begin{pmatrix}
    -ax +\lambda& -ax^2 \\ a & ax+\lambda
\end{pmatrix} \text{ and } X_j = \begin{pmatrix}
    0 & b \\ 0 & 0
\end{pmatrix}
\end{equation}
for some $a, b \neq 0$ in $\F_\ell$. Solving (\ref{matrix-equation}), we get $ab=0$. This is a contradiction since $a,b\neq0$.
    \item Suppose $Y_i \sim \omega_x$ and $X_j \sim \omega_y$ for some $x\neq y$ in $\F_\ell$. By (\ref{matrix form}),
\begin{equation}
    X_i = \begin{pmatrix}
    -ax +\lambda& -ax^2 \\ a & ax+\lambda
\end{pmatrix} \text{ and } X_j =  \begin{pmatrix}
   - by & -by^2 \\ b & by
\end{pmatrix}
\end{equation}
for some $a, b \neq 0$ in $ \F_\ell$. Solving (\ref{matrix-equation}), we get $-ab(x-y)^2=0$. This is a contradiction since $a,b \neq 0$ and $x\neq y$.
\item Suppose $Y_i \sim \omega$  and $X_j \sim \omega_x$ for some $x $ in $\F_\ell$. By (\ref{matrix form}),
\begin{equation}
    X_i = \begin{pmatrix}
    \lambda&  a\\ 0& \lambda
\end{pmatrix} \text{ and } X_j = \begin{pmatrix}
    -bx & -bx^2 \\ b & bx
\end{pmatrix}
\end{equation}
 for some $a, b \neq 0$ in $\F_\ell$. Solving (\ref{matrix-equation}), we get $ab=0$. This is a contradiction since $a, b \neq 0$. 
\end{itemize}

In other cases, $Y_i \sim X_j$.\\

Case 3. $\calO_i$ is $\ell$-fundamental, $\ell$ splits in $\calO_i$, and $\calO_j$ is not $\ell$-fundamental:

Let $\lambda_1\neq \lambda_2$ be the non-zero eigenvalues of $\alpha_i$ acting on $E[\ell]$.
Let $Y_i= X_i-\lambda_1 \id$ and $Z_i=X_i-\lambda_2 \id$, where $Y_i \not \sim Z_i$.
\begin{equation}
    I_i= I(\alpha_i,\lambda_1), J_i= I(\alpha_i, \lambda_2), I_j=I(\alpha_j,0)
\end{equation}
give the kernels of distinct $K_i$-horizontal $\ell$-isogenies and the kernel of a $K_j$-asending $\ell$-isogeny respectively, where
\begin{equation}
    M_2(\F_\ell)Y_i, M_2(\F_\ell)Z_i, M_2(\F_\ell)X_j
\end{equation}
are the corresponding left ideals of $M_2(\F_\ell)$ respectively. We show that either $I_i=I_j$ or $J_i=I_j$ by showing that either $Y_i \sim X_j$ or $Z_i \sim X_j$.

(i) Let $d_i \equiv 1 \pmod 4$. Then $\lambda_1, \lambda_2$ are non-zero distinct roots of 
\begin{equation}
    x^2-f_ix+\frac{f_i^2(1-d_i)}{4} \in \F_\ell[x],
\end{equation}
so $\lambda_1+\lambda_2 = f_i \neq 0$ in $\F_\ell$. 

\begin{itemize}
    \item Suppose $Y_i \sim \omega_x$, $Z_i \sim \omega_y$ for some $x \neq y $ in $\F_\ell$. By (\ref{matrix form}), 
$$X_i = \begin{pmatrix}
    -ax +\lambda_2& -ax^2+(-\lambda_1+\lambda_2)x\\ a & ax+\lambda_1
\end{pmatrix}=\begin{pmatrix}
    -by+\lambda_1 & -by^2+(\lambda_1 - \lambda_2)y \\ b& by+\lambda_2
\end{pmatrix},$$ so  $a=b\neq 0$ and $a(x-y)=\lambda_2-\lambda_1$. 

If $X_j \sim \omega_z$ for some $z \neq x,y $ in $\F_\ell$, then by (\ref{matrix form}),
\begin{equation}
    X_j =\begin{pmatrix}
    -cz & -cz^2 \\ c & cz
\end{pmatrix}
\end{equation}
for some $c \neq 0$ in $\F_\ell$. Solving (\ref{matrix-equation}), we get $a(x - z)=\lambda_2-\lambda_1$. Hence, $a(y-z)=0$. This is a contradiction since $a\neq 0$ and $y\neq z$. 

If $X_j \sim \omega$, then by (\ref{matrix form}),
\begin{equation}
    X_j =\begin{pmatrix}
    0 & c \\ 0 & 0
\end{pmatrix}
\end{equation}
for some $c \neq 0$ in $\F_\ell$. Solving (\ref{matrix-equation}), we get $ac=0$, which is a contradiction since $a,c\neq 0$. 
\item Suppose $Y_i \sim \omega_x$, $Z_i \sim \omega$, and $X_j \sim \omega_y$ for some $x\neq y$ in $\F_\ell$. By (\ref{matrix form}),
\begin{equation}
    X_i = \begin{pmatrix}
    -ax +\lambda_2& -ax^2+(-\lambda_1+\lambda_2)x\\ a & ax+\lambda_1
\end{pmatrix}=\begin{pmatrix}
    \lambda_2 & b \\ 0& \lambda_1
\end{pmatrix}
\end{equation}
for some $a, b\in \F_\ell$ and
\begin{equation}
    X_j =\begin{pmatrix}
    -cy & -cy^2 \\ c & cy
\end{pmatrix}
\end{equation}
for some $c \in \F_\ell, c \neq 0$. We have $a=0$ and $b=(\lambda_2-\lambda_1)x$. Solving (\ref{matrix-equation}), we get $-c(\lambda_1-\lambda_2)(x-y)=0$. This is a contradiction since $c \neq 0$, $\lambda_1 \neq \lambda_2$, and $x\neq y$.
\end{itemize}

In other cases, either $Y_i \sim X_j$ or $Z_i \sim X_j$.

(ii) Let $d_i \not\equiv 1 \pmod 4$. $\lambda_1, \lambda_2$ are non-zero distinct roots of 
\begin{equation}
    x^2-f^2_id_i\in \F_\ell[x],
\end{equation}
so $\lambda_1+\lambda_2 = 0$ in $\F_\ell$. In particular, $\ell \neq 2$.  

\begin{details}
    \begin{tbox}
       If $\ell =2$, $\lambda_1=\lambda_2$.
    \end{tbox}
\end{details}

\begin{itemize}
    \item Suppose $Y_i \sim \omega_x$, $Z_i \sim \omega_y$ for some $x \neq y $ in $\F_\ell$. By (\ref{matrix form}),
\begin{equation}
    X_i = \begin{pmatrix}
    -ax -\lambda_1& -ax^2-2\lambda_1x\\ a & ax+\lambda_1
\end{pmatrix}=\begin{pmatrix}
    -by+\lambda_1 & -by^2+2 \lambda_1y \\ b& by-\lambda_1
\end{pmatrix}
\end{equation}
for some $a, b \in \F_\ell$. Hence, $a=b\neq 0$ and $a(x-y)=-2\lambda_1$. 

If $X_j \sim \omega_z$ for some $z \neq x,y $ in $\F_\ell$, then by (\ref{matrix form}),
\begin{equation}
    X_j =\begin{pmatrix}
    -cz & -cz^2 \\ c & cz
\end{pmatrix}
\end{equation}
for some $c \neq 0$ in $\F_\ell$. Solving (\ref{matrix-equation}), we get $a(x-z)=-2\lambda_1$. Hence $a(y-z)=0$. This is a contradiction since $a\neq 0$ and $y\neq z$.

If $X_j \sim \omega$, then by (\ref{matrix form}),
\begin{equation}
    X_j =\begin{pmatrix}
    0 & c \\ 0 & 0
\end{pmatrix}
\end{equation}
 for some $c \neq 0$ in $\F_\ell$. Solving (\ref{matrix-equation}), we get $ac=0$, which is a contradiction since $a,c\neq 0$.  
 \item Suppose $Y_i \sim \omega_x$, $Z_i \sim \omega$, and $X_j \sim \omega_y$ for some $x\neq y$ in $\F_\ell$. By (\ref{matrix form}),
\begin{equation}
    X_i = \begin{pmatrix}
    -ax -\lambda_1& -ax^2-2\lambda_1 x \\ a & ax+\lambda_1
\end{pmatrix}=\begin{pmatrix}
    -\lambda_1 & b \\ 0& \lambda_1
\end{pmatrix}
\end{equation}
for some $a, b\in \F_\ell$ and
\begin{equation}
    X_j =\begin{pmatrix}
    -cy & -cy^2 \\ c & cy
\end{pmatrix}
\end{equation}
for some $c \in \F_\ell, c \neq 0$. We have $a=0$ and $b=-2\lambda_1 x$. Solving (\ref{matrix-equation}), we get $-2c\lambda_1(x-y)=0$. This is a contradiction since $\ell \neq 2, \lambda_1, c\neq 0,$ and $ x\neq y$.
\end{itemize}

In other cases, either $Y_i \sim X_j$ or $Z_i \sim X_j$.\\

Case 4. Both $\calO_i, \calO_j$ are $\ell$-fundamental, $\ell$ is split in $\calO_i$, and $\ell$ is ramified in $\calO_j$:

Let $\lambda_1\neq \lambda_2$ be the non-zero eigenvalues of $\alpha_i$ acting on $E[\ell]$.
Let $Y_i= X_i-\lambda_1 \id$ and $Z_i=X_i-\lambda_2 \id$, where $Y_i \not \sim Z_i$. Also, let $\lambda$ be the unique non-zero eigenvalue of $\alpha_j$ acting on $E[\ell]$ and let $Y_j = X_j - \lambda \id$.
\begin{equation}
    I_i=I(\alpha_i, \lambda_1), J_i=I(\alpha_i, \lambda_2), I_j=I(\alpha_j,\lambda)
\end{equation}
give the kernels of distinct $K_i$-horizontal $\ell$-isogenies and the kernel of a $K_j$-horizontal $\ell$-isogeny respectively, where
\begin{equation}
    M_2(\F_\ell)Y_i, M_2(\F_\ell)Z_i, M_2(\F_\ell)Y_j
\end{equation}
are the corresponding left ideals of $M_2(\F_\ell)$ respectively. We show that either $I_i=I_j$ or $J_i=I_j$ by showing that either $Y_i \sim Y_j$ or $Z_i \sim Y_j$.

(i) Let $d_i \not\equiv 1 \pmod{4}$ and $d_j \equiv 1 \pmod{4}$. $\lambda_1, \lambda_2$ are non-zero distinct roots of
\begin{equation}
    x^2-f_i^2 d_i \in \F_\ell[x],
\end{equation}
so $\lambda_1+\lambda_2=0$ in $\F_\ell$. In particular, $\ell \neq 2$. $\lambda$ is a non-zero double root of
\begin{equation}
    x^2-f_jx + \frac{f_j^2(1-d_j)}{4} \in \F_\ell[x],
\end{equation}
so $2\lambda = f_j \neq 0$ in $\F_\ell$.

\begin{itemize}
    \item Suppose $Y_i \sim \omega_x$, $Z_i \sim \omega_y$ for some $x \neq y$ in $\F_\ell$. By (\ref{matrix form}),
\begin{equation}
    X_i = \begin{pmatrix}
    -ax -\lambda_1& -ax^2-2\lambda_1x\\ a & ax+\lambda_1
\end{pmatrix}=\begin{pmatrix}
    -by+\lambda_1 & -by^2+2 \lambda_1y \\ b& by-\lambda_1
\end{pmatrix}
\end{equation}
for some $a, b \in \F_\ell$. Hence, $a=b\neq 0$ and $a(x-y)=-2\lambda_1$. 

If $Y_j \sim \omega_z$ for some $z \neq x,y$ in $\F_\ell$, then by (\ref{matrix form}),
\begin{equation}
    X_j=\begin{pmatrix}
    -cz+\lambda & -cz^2 \\ c & cz+\lambda
\end{pmatrix}
\end{equation}
for some $c \neq 0$ in $\F_\ell$. Solving (\ref{matrix-equation}), we get $a(x-z)=-2 \lambda_1$. Hence, $a(y-z)=0$. This is a contradiction since $a \neq 0$ and $y \neq z$. 

If $Y_j \sim \omega$, then  by (\ref{matrix form}),
\begin{equation}
    X_j=\begin{pmatrix}
    \lambda & c \\ 0 & \lambda 
\end{pmatrix}
\end{equation}
for some $c \neq 0$ in $\F_\ell$. Solving (\ref{matrix-equation}), we get $ac = 0$. This is a
contradiction since $a, c \neq 0$.
\item Suppose $Y_i \sim \omega_x$, $Z_i \sim \omega$, and $Y_j \sim \omega_y$ for some $x\neq y$ in $\F_\ell$. By (\ref{matrix form}),
\begin{equation}
    X_i = \begin{pmatrix}
    -ax -\lambda_1& -ax^2-2\lambda_1x\\ a & ax+\lambda_1
\end{pmatrix}=\begin{pmatrix}
    -\lambda_1 & b \\ 0& \lambda_1
\end{pmatrix}
\end{equation}
for some $a, b \in \F_\ell$ and
\begin{equation}
    X_j =\begin{pmatrix}
    -cy+\lambda & -cy^2 \\ c & cy+\lambda
\end{pmatrix}
\end{equation}
for some $c \in \F_\ell, c \neq 0$. We have $a=0$ and $b=-2\lambda_1x$. Solving (\ref{matrix-equation}), we get $-2 \lambda_1c(x-y)=0$. This is a contradiction since $\ell \neq 2$, $\lambda_1, c \neq 0$, and $x \neq y$. 
\end{itemize}

In other cases, either $Y_i \sim Y_j$ or $Z_i \sim Y_j$.

(ii) Let $d_i \equiv 1 \pmod{4}$ and $d_j \not\equiv 1 \pmod{4}$. $\lambda_1, \lambda_2$ are non-zero distinct roots of 
\begin{equation}
    x^2-f_ix+\frac{f_i^2(1-d_i)}{4} \in \F_\ell[x],
\end{equation}
so $\lambda_1+\lambda_2 = f_i \neq 0$ in $\F_\ell$. In particular, $\ell \neq 2$. $\lambda$ is a non-zero double root of
\begin{equation}
    x^2-f_j^2d_j \in \F_\ell[x],
\end{equation}
so $2\lambda = 0$ in $\F_\ell$. This implies $\ell=2$ as $\lambda \neq 0$. Hence, this case does not occur.

\begin{details}
    \begin{tbox}
        If $\ell=2$, then the only non-zero element in $\F_\ell$ is $1$, so $\lambda_1=\lambda_2=1$, which is a contradiction as $\lambda_1\neq \lambda_2$.
    \end{tbox}
\end{details}

\begin{details}
    \begin{tbox}
        Suppose $Y_i \sim \omega_x$, $Z_i \sim \omega_y$ for some $x \neq y $ in $\F_\ell$. Then 
$$X_i = \begin{pmatrix}
    -ax +\lambda_2& -ax^2+(-\lambda_1+\lambda_2)x\\ a & ax+\lambda_1
\end{pmatrix}=\begin{pmatrix}
    -by+\lambda_1 & -by^2+(\lambda_1 - \lambda_2)y \\ b& by+\lambda_2
\end{pmatrix},$$ so $a=b\neq 0$ and $a(x-y)=\lambda_2-\lambda_1$.

If $Y_j \sim \omega_z$ for some $z \neq x, y $ in $\F_\ell$, then $X_j=\begin{pmatrix}
    -cz+\lambda & -cz^2 \\ c & cz+\lambda
\end{pmatrix}$ for some $c \neq 0$ in $\F_\ell$. Solving (\ref{matrix-equation}), we get $2(x-z)(y-z)=0$. This is a contradiction since $\ell \neq 2$ and $z \neq x, y$.

If $Y_j \sim \omega$, then $X_j= \begin{pmatrix}
    \lambda & c \\ 0 & \lambda
\end{pmatrix}$ for some $c \neq 0$ in $\F_\ell$. Solving (\ref{matrix-equation}), we get 

Suppose $Y_i \sim \omega_x$, $Z_i \sim \omega$, and $X_j \sim \omega_y$ for some $x\neq y$ in $\F_\ell$.
$$X_i = \begin{pmatrix}
    -ax +\lambda_2& -ax^2+(-\lambda_1+\lambda_2)x\\ a & ax+\lambda_1
\end{pmatrix}=\begin{pmatrix}
    \lambda_2 & b \\ 0& \lambda_1
\end{pmatrix}$$
and $X_j =\begin{pmatrix}
    -cy+\lambda & -cy^2 \\ c & cy+\lambda
\end{pmatrix}$
for some $a, b, c \in \F_\ell, c \neq 0$. We have $a=0$ and $b=(\lambda_2-\lambda_1)x$. Solving (\ref{matrix-equation}), we get $\lambda=c(y-x)$. This implies $2c(y-x)=0$ which is a contradiction since $\ell \neq 2$ and $y \neq x$. \\
    \end{tbox}
\end{details}

(iii) Let $d_i \not\equiv 1 \pmod{4}$ and $d_j \not\equiv 1 \pmod{4}$. $\lambda_1, \lambda_2$ are non-zero distinct roots of
\begin{equation}
    x^2-f^2_id_i\in \F_\ell[x],
\end{equation}
so $\lambda_1+\lambda_2 = 0$ in $\F_\ell$. In particular, $\ell \neq 2$. $\lambda$ is a non-zero double root of 
\begin{equation}
    x^2-f_j^2d_j \in \F_\ell[x],
\end{equation}
so $2\lambda = 0$ in $\F_\ell$. This implies $\ell=2$ since $\lambda \neq 0$. Hence, this case does not occur.\\

Case 5. Both $\calO_i, \calO_j$ are $\ell$-fundamental and $\ell$ is ramified in both $\calO_i$ and $\calO_j$:

Let $\lambda$ be the unique non-zero eigenvalue of $\alpha_i$ acting on $E[\ell]$ and let $Y_i= X_i-\lambda \id$. Also, let $\lambda'$ be the unique non-zero eigenvalue of $\alpha_j$ acting on $E[\ell]$ and let $Y_j = X_j - \lambda' \id$.
\begin{equation}
    I_i=I(\alpha_i, \lambda), I_j=I(\alpha_j, \lambda')
\end{equation}
give the kernel of a $K_i$-horizontal $\ell$-isogeny and the kernel of a $K_j$-horizontal $\ell$-isogeny respectively, where
\begin{equation}
    M_2(\F_\ell)Y_i, M_2(\F_\ell)Y_j
\end{equation}
are the corresponding left ideals of $M_2(\F_\ell)$ respectively. We show that $I_i=I_j$ by showing that $Y_i \sim Y_j$. 

(i) Let $d_i \not\equiv 1 \pmod{4}$ and $d_j \equiv 1 \pmod{4}$. $\lambda$ is a non-zero double root of 
\begin{equation}
    x^2-f_i^2d_i \in \F_\ell[x],
\end{equation}
so $2\lambda = 0$ in $\F_\ell$. This implies $\ell=2$ since $\lambda \neq 0$. $\lambda'$ is a non-zero double root of
\begin{equation}
    x^2-f_jx+\frac{f_j^2(1-d_j)}{4} \in \F_\ell[x],
\end{equation}
so $2\lambda'=f_j \neq 0$ in $\F_\ell$. In particular, $\ell \neq 2$. Hence, this case does not occur. 

(ii) Let $d_i \not\equiv 1 \pmod{4}$ and $d_j \not\equiv 1 \pmod{4}$. $\lambda$ is a non-zero double root of 
\begin{equation}
    x^2-f_i^2d_i \in \F_\ell[x],
\end{equation}
so $2\lambda = 0$ in $\F_\ell$. $\lambda'$ is a non-zero double root of 
\begin{equation}
    x^2-f_j^2 d_j \in \F_\ell[x],
\end{equation}
so $2\lambda'=0$ in $\F_\ell$. 

\begin{itemize}
    \item Suppose $Y_i \sim \omega_x$ for some $x $ in $\F_\ell$ and $Y_j \sim \omega$. By (\ref{matrix form}),
\begin{equation}
    X_i = \begin{pmatrix}
    -ax +\lambda& -ax^2 \\ a & ax+\lambda
\end{pmatrix} \text{ and } X_j = \begin{pmatrix}
    \lambda' & b \\ 0 & \lambda'
\end{pmatrix}
\end{equation}
for some $a, b \neq 0$ in $\F_\ell$. Solving (\ref{matrix-equation}), we get $ab=0$. This is a contradiction since $a \neq 0$ and $b \neq 0$.
    \item Suppose $Y_i \sim \omega_x$ and $Y_j \sim \omega_y$ for some $x\neq y$ in $\F_\ell$. By (\ref{matrix form}),
\begin{equation}
    X_i = \begin{pmatrix}
    -ax +\lambda& -ax^2 \\ a & ax+\lambda
\end{pmatrix} \text{ and } X_j =  \begin{pmatrix}
   - by+\lambda' & -by^2 \\ b & by+\lambda'
\end{pmatrix}
\end{equation}
for some $a, b \neq 0$ in $ \F_\ell$. Solving (\ref{matrix-equation}), we get $-ab(x-y)^2=0$. This is a contradiction since $a \neq 0, b \neq 0$, and $x\neq y$. 
\end{itemize}

In other cases, $Y_i \sim Y_j$.
\begin{details}
    \begin{tbox}
      Suppose $Y_i \sim \omega$ and $Y_j \sim \omega_x$ for some $x $ in $\F_\ell$. This case is symmetric to $Y_i \sim \omega_x$ for some $x $ in $\F_\ell$ and $Y_j \sim \omega$.
    \end{tbox}
\end{details}
\end{proof}

\begin{extra}
Here is an alternate proof of Proposition~\ref{joint-ascend-general}.

\begin{proposition}
\label{joint-ascend}
Let $E/\overline{\F}_p$ be a supersingular elliptic curve and $\ell \not= p$ be a prime. Suppose $E$ is primitively $\calO_i, \calO_j$-oriented where $\calO_i = \Z[\alpha_i]$ is an order in $\Q(i)$ and $\calO_j = \Z[\alpha_j]$ is an order in $\Q(j)$.

If $(\calO_i, \calO_j) \not= (\calO_{K_i}, \calO_{K_j})$, then there is a unique isogeny $\phi: E \rightarrow E'$ of degree $\ell$ such that $E'$ has orientations by $\calO_i'$ and $\calO_j'$.
\end{proposition}

\begin{proof}
By Theorem \ref{ss-correspondence}, there is a fractional right $\frakO_0$-ideal $I$ in $B_0$ corresponding to $\Hom(E_0,E)$ which we may assume to be integral.

We seek an $\frakO_0$-ideal $I'$ corresponding to $\Hom(E_0,E')$ so that
\begin{align}
    \End_{\frakO_0}(I') & \supseteq \calO_i' \\
    \End_{\frakO_0}(I') & \supseteq \calO_j'.
\end{align}

We seek a natural way to extend $I$ to a new lattice $I'$ so it is closed under multiplication on the left by $\calO_i'$ and $\calO_j'$ and the index $[I':I] = \ell^2$ as a $\Z$-module of rank $4$. This is achieved by considering
\begin{equation}
   I' = \begin{cases}
        I + (\alpha_i' \pm \alpha_j') I & \text{ if } \calO_i \not= \calO_{K_i}, \calO_j \not= \calO_{K_j}, \\
        I + \alpha_i' I & \text{ if } \calO_i \not= \calO_{K_i}, \calO_j = \calO_{K_j}, \\
        I + \alpha_j' I & \text{ if } \calO_i = \calO_{K_i}, \calO_j \not= \calO_{K_j}. 
    \end{cases}
\end{equation}
First note that
\begin{equation}
    I \subsetneq I' \subsetneq I/\ell.
\end{equation}
Since $[I/\ell:I] = \ell^4$ and $[I':I]$ is a square, it follows that $[I':I] = \ell^2$.

If $\alpha_j' = \alpha_j$ (and similarly if $\alpha_i' = \alpha_i$), then 
\begin{equation}
    I' = I + \alpha_i'I,
\end{equation}
and we see that the left order $\frakO'$ of $I'$ contains $\Z[\alpha_i', \alpha_j']$ as
\begin{align}
    \alpha_i' \alpha_i' & \in \Z \alpha_i' + \Z \\
\label{rosati-commute}  \alpha_j \alpha_i' & \in \Z \alpha_i' \alpha_j + \Z \alpha_j,
\end{align}
where \eqref{rosati-commute} comes from Remark~\ref{opposite-mod-4}.

Suppose now that $\alpha_i' \not= \alpha_i$ and $\alpha_j' \not= \alpha_j$. 
By the relations
\begin{align*}
  (\alpha_i' + \alpha_j')^2 & \in \Z(\alpha_i' + \alpha_j') + \Z \\
  (\alpha_i' - \alpha_j')^2 & \in \Z(\alpha_i' - \alpha_j') + \Z 
\end{align*}
we see that the left order $\frakO'$ of $I'$ contains $\Z[\alpha_i'+\alpha_j', \alpha_i'-\alpha_j'] = \Z[2 \alpha_i', 2\alpha_j']$. If $\ell \not= 2$, we deduce that $\frakO'$ contains $\Z[\alpha_i', \alpha_j']$ as $2$ is coprime to $[I':I] = \ell^2$. 

If $\ell = 2$, then
\begin{align*}
  \left( \frac{\alpha_i' + \alpha_j'}{2} \right)^2 & \in \Z \left(\frac{\alpha_i' + \alpha_j'}{2}\right) + \Z \\
  \left( \frac{\alpha_i' - \alpha_j'}{2} \right)^2 & \in \Z \left(\frac{\alpha_i' - \alpha_j'}{2}\right) + \Z,
\end{align*}
from which it follows that $\frakO'$ contains $\Z[\alpha_i', \alpha_j']$. 
\end{proof}
\end{extra}

\subsection{Volcanoes of double-oriented isogeny graphs}
From Proposition \ref{joint-ascend-general}, we obtain Theorem~\ref{Formula for sturcture 1} which is refined result on the structure of the double-oriented isogeny graph, analogous to the one for the single-oriented isogeny graph given in Proposition \ref{Onu21 Proposition 4.1}.

\begin{theorem} \label{Formula for sturcture 1}
Let $(E,\jmath)/\cong$ be a vertex in $G_{K_i, K_j}(p,\ell)$ where $\jmath$ is a primitive $\calO_i, \calO_j$-orientation on $E$. Let $f_i, f_j$ be the conductors of $\calO_i, \calO_j$ respectively and $\Delta_i, \Delta_j$ be the discriminants of $\calO_i, \calO_j$ respectively. 

    If $\ell \nmid f_i, f_j$ and $\ell$ is not inert in both of $K_i, K_j$, then the following hold.
    \begin{enumerate}
        \item There are $1-\left(\frac{d_{K_i}}{\ell}\right)\left(\frac{d_{K_j}}{\ell}\right)$ simultaneously horizontal edges from $(E,\jmath)$.
        \item There are $\left(\frac{d_{K_i}}{\ell}\right)\left(\left(\frac{d_{K_j}}{\ell}\right)+1 \right)$ edges from $(E,\jmath)$ which are mixed horizontal-descending.
        \item There are $\left(\left(\frac{d_{K_i}}{\ell}\right)+1 \right)\left(\frac{d_{K_j}}{\ell}\right)$ edges from $(E,\jmath)$ which are mixed descending-horizontal.
        \item The remaining edges from $(E,\jmath)$ are simultaneously descending.
    \end{enumerate}

    If $\ell \nmid f_i, f_j$ and $\ell$ is inert in at least one of $K_i, K_j$, then the following hold.
    \begin{enumerate}
        \item There are $\left(\frac{d_{k_i}}{\ell}\right)+1$ edges from $(E,\jmath)$ which are mixed horizontal-descending.
        \item There are $\left(\frac{d_{k_j}}{\ell}\right)+1$ edges from $(E,\jmath)$ which are mixed descending-horizontal.
        \item The remaining edges from $(E,\jmath)$ are simultaneously descending.
    \end{enumerate}

    If $\ell$ divides exactly one of $f_i, f_j$, the following hold.
    \begin{enumerate}
        \item There is exactly one edge from $(E,\jmath)$ which is mixed ascending-horizontal or horizontal-ascending edge.
        \item There are $\max\left\{\left(\frac{\Delta_i \vphantom{\Delta_j}}{\ell}\right), \left(\frac{\Delta_j}{\ell}\right)\right\}$ edges from $(E,\jmath)$ which are mixed horizontal-descending or descending-horizontal.
        \item The remaining edges from $(E,\jmath)$ are simultaneously descending.
    \end{enumerate}
    If $\ell \mid f_i, f_j$, then the following hold.
    \begin{enumerate}
        \item There is exactly one simultaneously ascending edge from $(E,\jmath)$.
        \item The remaining edges from $(E,\jmath)$ are simultaneously descending.
    \end{enumerate}
\end{theorem}

\begin{details}
    \begin{tbox}
If $\ell \nmid f_i, f_j$, then the following hold.
        \begin{enumerate}
        \item There is simultaneously horizontal edges if $\ell$ is split in both of $K_i, K_j$. 
        \item There are exactly one simultaneously horizontal edge and exactly one horizontal/descending edge from $(E,\jmath)$ if $\ell$ is split in one of $K_i, K_j$ and is ramified in the other.
        \item There is exactly one simultaneously horizontal edge from $(E,\jmath)$ if $\ell$ is ramified in both $K_i, K_j$. 
        \item The remaining edges from $(E,\jmath)$ are simultaneosly descending. In particular, all edges $(E,\jmath)$ are simultaneosly descending if $\ell$ is inert in both of $K_i, K_j$.
    \end{enumerate}
     If $\ell$ divides exactly one of $f_i, f_j$, the the following hold.
    \begin{enumerate}
        \item There is exactly one ascending/horizontal edge from $(E,\jmath)$.
        \item There is exactly one horizontal/descending edge from $(E,\jmath)$ if $\ell$ is split in $K_u$ for $u \in \{i,j\}$ such that $\ell \nmid f_u$.
        \item The remaining edges from $(E,\jmath)$ are simultaneously descending.
    \end{enumerate}
\end{tbox}
\end{details}

\begin{extra}
    We have a different formulation of the result given in Corollary \ref{Formula for sturcture 1}.

\begin{corollary} \label{Formula for sturcture 2}
    Let $(E,\jmath)$ be a primitively $\mathfrak{O}$-oriented curve in $G_{K_i, K_j}(p,\ell)$. Let $D=\discrd(\mathfrak{O})$. 
Suppose $\mathfrak{O}=\mathfrak{O}_{i,j}$.

If $\ell \nmid D$, then the following hold.
\begin{enumerate}
    \item There are $\left(\frac{\mathfrak{O}_{i,j}}{\ell}\right)+1$ simultaneously horizontal edges from $(E,\jmath)$. 
    \item The remaining edges from $(E,\jmath)$ are simultaneously descending.
\end{enumerate}

If $\ell \parallel D$, then the following hold.
\begin{enumerate}
    \item There is $\left(\frac{\mathfrak{O}_{i,j}}{\ell}\right)=1$ simultaneously horizontal edge from $(E,\jmath)$.
\end{enumerate}

If $\ell^2 \mid D$, then the following hold.
\begin{enumerate}
    \item  There are $\left(\frac{\mathfrak{O}_{i,j}}{\ell}\right)+1$ simultaneously horizontal edges from $(E,\jmath)$. 
    \item 
\end{enumerate}

If $\mathfrak{O}$ is locally Bass at $\ell$ and contain exactly one of  $\calO_{K_i}\otimes \Z_\ell$ and $$\calO_{K_i}\otimes \Z_\ell$$
\begin{enumerate}
    \item There is exactly one ascending/horizontal edge from $(E,\jmath)$.
    \item There are $\max\left\{\left(\frac{\mathfrak{O}}{\ell}\right), 0\right\}$ horizontal/descending edges from $(E,\jmath)$.
    \item The remaining edges from $(E,\jmath)$ are simultaneously descending.
\end{enumerate}

If $\mathfrak{O}$ is not locally Bass at $\ell$, then the following hold.
    \begin{enumerate}
        \item There is exactly one simultaneously ascending edge from $(E,\jmath)$.
        \item The remaining edges from $(E,\jmath)$ are simultaneously descending.
    \end{enumerate}
\end{corollary}

\end{extra}

\begin{corollary}
 The graph $G_{K_i,K_j}(p,\ell)$ is a ``volcanic ridge'', that is, it has a volcano structure except that the rim is replaced by a ridge. The edges from the ridge vertices may be mixed, otherwise the remaining edges are simultaneously descending.
\end{corollary}

\begin{remark}
    The graph $G_{K_i, K_j}(p,\ell)$ has two underlying single-oriented graphs $G_{K_i}(p,\ell)$ and $G_{K_j}(p, \ell)$. The rim of a connected component of an underlying single-oriented graph is a line in the corresponding connected component of $G_{K_i, K_j}(p,\ell)$. In case the rim forms a cycle, it corresponds to an infinite line in $G_{K_i, K_j}(p,\ell)$. The ridge of a connected component in $G_{K_i, K_j}(p,\ell)$ is the intersection of lines corresponding to the rims in the underlying single-oriented graphs.
\end{remark}

\begin{extra}

\section{Some insight/strategy we might want to state instead of actual algorithms}

\begin{remark}
    Consider two oriented graphs $G_{K}(p,\ell)$ and $G_{K_1, K_2}(p, \ell)$. 
 In the single-oriented isogeny graph $G_{K}(p,\ell)$, if a curve $E$ is primitively $\calO$-oriented by an order $\calO \subseteq K$ with conductor $f$, then 
 \begin{equation}
     \nu_\ell(f)
 \end{equation}
 represents the depth of $E$ from the rim, i.e., the number of ascensions required to reach the rim. 
        Similarly, in the double-oriented isogeny graph $G_{K_i, K_j}(p, \ell)$, if a curve $E$ is primitively $\calO_i, \calO_j$-oriented by orders $\calO_i \subseteq K_i$ and $\calO_j \subseteq K_j$ with conductors $f_i,f_j$ respectively, 
    \begin{equation}
        \nu_\ell(\lcm(f_if_j))=\max\{v_{\ell}(f_i), v_{\ell}(f_j)\}
    \end{equation}
     represents the depth of $E$ from the local roots. Notice that ascending in $G_{K_i, K_j}(p,\ell)$ is equivalent to simultaneously ascending in $G_{K_i}(p,\ell)$ and $G_{K_j}(p, \ell)$. We will reach to the rim of one of the single-oriented graphs with smaller depth first. Then we move horizontally in that graph, while we keep ascending in the other graph until we reach to the rim. Hence, to reach to the local roots in $G_{K_i, K_j}(p,\ell)$, we only need to ascend in one of the single-oriented graphs, namely the one farther from the rim.

        If we want to find a path from $E$ to a global root in $\bigcup_\ell G_{\ell}(K_i, K_j)$, we can recursively ascend to a local root in $G_{K_i,K_j}(p,\ell)$ for each $\ell$. Hence,    
\begin{equation}
    \sum_{\ell}\nu_\ell(\lcm(f_i,f_j)),
\end{equation}
        where $\ell$ runs over primes, represents the depth of $E$ from global roots. Note that each path to a global root has the length at most 
\begin{equation}
    n+\sum_{\ell}\nu_\ell(\lcm(f_i,f_j)),
\end{equation} 
where $n$ represents the number of horizontal moves between local roots. Let $\calO_{K_i}, \calO_{K_j}$ be maximal orders in $K_i, K_j$ respectively and let $\mathfrak{O}_{i,j}$ be the Bass order generated by $\calO_{K_i}, \calO_{K_j}$ in $B_0$. In the proof of Lemma \ref{number of local roots}, we showed that $e_{\ell}(\mathfrak{O}_{i,j})=2$ if and only if $\ell \mid \discrd(\mathfrak{O}_{i,j})=\frac{d_{K_i}d_{K_j}}{4}$, where $d_{K_i}, d_{K_j}$ are the fundamental discriminants of $K_i, K_j$ respectively. Hence, $n$ is the number of distinct primes divisors of  $\frac{d_{K_i}d_{K_j}}{4}$. 

    To find a path between two curves in $\bigcup_\ell G_{\ell}(K_i, K_j)$, we can find paths to all global roots from each curve and fine a collision. Recall that if there are two local roots in $G_{K_i, K_j}(p, \ell)$, then they are neighbors in $G_{K_i, K_j}(p, \ell)$, but not in $G_{K_i, K_j}(p, \ell')$ for any prime $\ell'\neq\ell$ by Corollary \ref{local roots are neigbors}. Hence, we can take the following recursive steps and detect all the global roots in $\bigcup_\ell G_{\ell}(K_i, K_j)$: let $\mathbb{P}=(\ell_k)$ be the ordered set of primes.  

    \begin{enumerate}
        \item[0] Let $(E, \jmath)$ be a curve in $\bigcup_\ell G_{\ell}(K_i, K_j)$. Store the null path of length $0$ from $E$ to itself.
        \item[\textbf{for} $\ell \in \mathbb{P}$ 
        \textbf{do}]
        \item[1] For each endpoint curve in the paths stored previously, find ascending paths to all the local roots in $G_\ell(K_i, K_j)$.
        \item[2] Extend the paths stored previously by concatenating with paths found in Step 1. Forget previously stored paths and store the extended paths.
    \end{enumerate}
    Note that after $m$-th iteration, we obtain $e_{\ell_1}(\mathfrak{O}_{i,j})\cdots e_{\ell_m}(\mathfrak{O}_{i,j})$ distinct paths from $E$. We only need to take finite number of recursions as there are finitely many primes $\ell$ such that either $\ell \mid f_if_j$ or $\ell \mid \frac{d_{K_i}d_{K_j}}{4}$, where the former implies we have ascending moves and the latter implies we have horizontal moves within $G_{K_i,K_j}(p,\ell)$. Hence, after taking recursions finite times, we obtain $e(\mathfrak{O}_{i,j})$ distinct paths from $E$ to all the global roots.
    \end{remark}

\begin{details}
    
Let $(E_0, \jmath_0)$ be the one of global roots in $\cup_\ell G_\ell(K_i, K_j)$. The above process construct an isogeny $\varphi: E_0 \rightarrow E$. Let $\mathfrak{O}_0=\End(E_0)$. Then
$$\deg \varphi =[\mathfrak{O}_0:\mathfrak{O}_0 \cap \mathfrak{O}_\varphi]$$
is the minimal distance between $(E_0, \jmath)$ and $(E, \jmath)$.

\begin{tbox}
    If $\mathfrak{O}_{i,j}$ is the Bass order generated by $\mathcal{O}_{K_i}, \mathcal{O}_{K_j}$ and $\mathfrak{O}$ is the order generated by $\mathcal{O}_i, \mathcal{O}_j$, then
\[
\begin{tikzcd}
\mathfrak{O}_0 \arrow[r, phantom, "\supseteq"] \arrow[d, phantom, "\rotatebox{90}{$\subseteq$}"'] & \mathfrak{O}_0\cap \mathfrak{O}_\varphi \arrow[d, phantom, "\rotatebox{90}{$\subseteq$}"] \\
\mathfrak{O}_{i,j} \arrow[r, phantom, "\supseteq"'] & \mathfrak{O}
\end{tikzcd}
\]
\begin{equation}
    [\mathfrak{O}_0 :\mathfrak{O}_0 \cap \mathfrak{O}_\varphi][\mathfrak{O}_0 \cap \mathfrak{O}_\varphi: \mathfrak{O}]=[\mathfrak{O}_0 : \mathfrak{O}_{i,j}][\mathfrak{O}_{i,j} : \mathfrak{O}]
\end{equation}
where
\begin{equation}
    \begin{split}
        [\mathfrak{O}_0 :\mathfrak{O}_0 \cap \mathfrak{O}_\varphi] & = \deg \varphi = n+\sum_{\ell} \nu_\ell(\lcm(f_if_j)) \\
        [\mathfrak{O}_0 \cap \mathfrak{O}_\varphi: [\mathfrak{O}_0 : \mathfrak{O}_{i,j}] & = \frac{\discrd (\mathfrak{O}_{i,j})}{\discrd(\mathfrak{O}_\varphi)}=\frac{d_{K_i}d_{K_j}}{4p}\\
        [\mathfrak{O}_{i,j} : \mathfrak{O}] & = [\calO_{K_i}\otimes \calO_{K_j}:\calO_i \otimes \mathcal{O}_j]= f_if_j \ \text{or} f_i^2f_j^2.
    \end{split}
\end{equation}
Hence, we can compute $[\mathfrak{O}_0 \cap \mathfrak{O}_\varphi: \mathfrak{O}]$.
\end{tbox}
\end{details}

\section{Proof of Theorem~\ref{main-reduce}}

Let $E/\overline{\F}_p$ be a supersingular elliptic curve.

Suppose we have used the Maximal Order Problem oracle to compute $\End(E) \cong \mathfrak{O}$ as a maximal order in $B_0$.

Using \cite[Lemma 7.2]{KV10}, we can enumerate efficiently the invertible right $\frakO$-ideals $J'$ of norm $\ell$. Taking the left order of the ideals $J'$ gives the maximal orders $\frakO' \cong \End(E')$ corresponding to endomorphism rings of the elliptic curves $E'$ which are $\ell$-isogenous to $E$.

\begin{algorithm}
\label{root-find}
\SetAlgoLined
\SetKwInOut{Input}{input}
\SetKwInOut{Output}{output}
\Input{Primes $p \not= \ell$ and a supersingular elliptic curve $E$ over $\overline{\F}_{p}$.} 
\Output{A sequence of $\ell$-isogenies from $E$ to $E_0$ or $E_1$.} 
1. Compute the isomorphism class of $\End(E) \cong \frakO$ as a maximal order in $B_0$. \\
2. Compute the elliptic curves $E$ which are $\ell$-isogenous to $E$. \\
2. Compute the maximal orders $\frakO' \cong \End(E')$ corresponding to each $E'$ using the oracle. \\
3. Let $E \leftarrow E'$ where $\End(E') \cong \frakO'$ ascends for both $\Q(i)$ and $\Q(j)$ orientations. \\
4. Repeat until one of the two roots $\D_0, \D_1$ is obtained.
\caption{Find a path to a root using an oracle to the Maximal Order Problem}
\end{algorithm}

\todo[inline]{need to bound depth from statement of problem}

Algorithm~\ref{root-find} finds a path from an arbitrary supersingular elliptic curve $E$ to one of the roots $E_0, E_1$ using a sequence of $\ell$-isogenies.

A random walk on $G(p,\ell)$ quickly reaches the uniform distribution after $O(\log N)$ steps where $N$ is the number of vertices as $G(p,\ell)$ is an expander graph \cite{expander}. The uniqueness of the ascending isogeny in Proposition~\ref{joint-ascend} implies the vertices of $G(p,\ell)$ are partitioned into up to two pieces which ascend up to each of the roots $E_0$ or  $E_1$, respectively.

If there is only one root, that is, $E_0 \cong E_1$, then the algorithm will quickly connect any supersingular elliptic curve $E$ to this root. 

Assume now that there are two roots $E_0 \not\cong E_1$. Let $p_0, p_1$ be the probability of ascending to the root $E_0, E_1$, respectively. Without loss of generality, we may assume $p_0 \ge 1/2$ by reversing the roles of $E_0$ and $E_1$.

If $p_0 \ge 2/3$, then given a supersingular elliptic curve $E$, we can perform a random walk from $E$ to $E'$ and use Algorithm~\ref{root-find} to find a path to the root $E_0$. After $n$ repetitions, the majority of the repetitions ascend to $E_0$ with high probability. Hence, given two supersingular elliptic curves, we can find a path to $E_0$, and thus to each other.

If $1/2 \le p_0 < 2/3$, then given a supersingular elliptic curve $E$, we can perform a random walk from $E$ to $E$', then use Algorithm~\ref{root-find} to find a path from $E'$ to either $E_0$ or $E_1$. Repeating this process, with high probability we ascend to both $E_0$  and $E_1$, thus connecting $E_0$ and $E_1$ by a known path, and hence any two supersingular elliptic curves.

\begin{details}
\begin{tbox}

Remark: In the arguments below, one should be able to replace $2/3$ by $1/2 + (\log p)^{-c}$ and $1/3$ by $1/2 - (\log p)^{-c}$ with a more complicated analysis.

Case $p_0 > 2/3$: There is a useful summary from 
\begin{verbatim}
https://people.seas.harvard.edu/~salil/pseudorandomness/
\end{verbatim}
about complexity classes.

\begin{definition}
{\bf BPP} is the class of languages $L$ for which there exist a probabilistic polynomial-time algorithm $A$ such that 
\begin{itemize}
  \item $x \in L \implies \text{Pr}[A(x) \text{ accepts}] \ge 2/3$
  \item $x \notin L \implies \text{Pr}[A(x) \text{ accepts}] \le 1/3$.
\end{itemize}
\end{definition}
The second case is the error probability which is bounded. The error probability of ${\bf BPP}$ algorithms can be reduced from $1/3$ to exponentially small by repetitions, taking a majority vote of outcomes. That is, we run $A(x)$ say $n$ times and the new $A'(x)$ accepts if a majority of the $A(x)$ accept. The error bound analysis uses the Chernoff bound and is explained in the reference above (let $L$ be the language consisting of the set of the elliptic curves $E$ which ascend to $E_0$; $q(n) = 1/6$ corresponds to error probability $\le 1/3$; take $p(n) = n$ for instance).

Case $p_0 \le 2/3$: Note $1/2 \le p_0 \le 2/3$ and $1/3 \le p_1 \le 1/2$.

After $n$ repetitions, the probability that $E$ ascends to $E_0$ every time is $< (2/3)^n = e^{- \log(3/2) n}$ and the probability that $E$ ascends to $E_1$ every time is $\le 1/2^n = e^{- \log(2) n}$. Hence, with high probability we ascend to both $E_0$ and $E_1$ after $n$ repetitions. 
\end{tbox}
\end{details}

\todo[inline]{Explain how the algorithm which works on the quaternion algebra side translates into navigating on the isogeny graph side.}

\section{Classical Path Finding Algorithms}

\todo[inline]{
 As discussed in Remark \ref{suitable-primitive}, we assume that isogenies and orientations have efficient representations. We also assume that the factorization of the discriminant of a given endomorphism as a maximal order in the endomorphism algebra is known. $\theta_i, \theta_j$ are primitive.

}

In this section, we discuss path-finding algorithms with two orientations.

\begin{algorithm}
\label{path to local roots}
\SetAlgoLined
\SetKwInOut{Input}{input}
\SetKwInOut{Output}{output}
\Input{A supersingular elliptic curves $E/\Fbar_{p^2}$ and a pair of endomorphisms $\theta_i, \theta_j$ of $E$ such that the associated fields $K_i \cong \Q(\theta_i)$ and $K_j \cong \Q(\theta_j)$ satisfy the condition \ref{condition1}.} 
\Output{A path from $E$ to a local root in the $K_i, K_j$-oriented supersingular $\ell$-isogeny graph $G_{K_1, K_2}(p, \ell)$.} 
Compute the conductor $f_i, f_j$ of $\Z[\theta_i], \Z[\theta_j]$ respectively.\\
Compute $s=\nu_\ell(f_i), t=\nu_\ell(f_j)$.\\
\eIf{$s\geq t$}{
$u\leftarrow i$.}{
$u \leftarrow j$.}
Find an ascending $\ell$-isogeny path $H=(j_Ej_1\cdots j_m)$ of length $m=\max\{s, t\}$ to the rim of the $K_u$-oriented $\ell$-isogeny volcano from $E$ with the $K_u$-orientation $\iota$ on $E$ induced by $\theta_u$.\\
\KwRet{$H$}
\caption{Ascending to the local roots}
\end{algorithm}

\begin{proposition} \label{path to local roots complexity}
    Let $E/\Fbar_{p^2}$ be a supersingular elliptic curve and $\ell \neq p$ be a prime. Suppose that we are given a pair of endomorphisms  $\theta_i, \theta_j \in \End(E)$ such that the associated fields $K_i \cong \Q(\theta_i)$ and $K_j\cong \Q(\theta_j)$ satisfy the condition \ref{condition1}. Let $\iota_i, \iota_j$ be the orientations on $E$ induced by $\theta_i, \theta_j$ respectively and let $\jmath=(\iota_i, \iota_j)$ be a $K_i, K_j$-orientation on $E$. Let $f_i, f_j$ be the conductors of the discriminants of $\theta_i, \theta_j$ respectively with $m=\max\{\nu_\ell(f_i), \nu_\ell(f_j)\}$. Algorithm \ref{path to local roots} produces a path of length $m$ from $(E, \jmath)$ to a local root in the $K_i, K_j$-oriented supersingular $\ell$-isogeny graph $G_{K_1, K_2}(p, \ell)$ and runs in time polynomial in $\ell, m, \log p$ and in the length of the representations of the orientations $\iota_i, \iota_j$.
\end{proposition}

\begin{proof}
   Note that $s, t$ represent the distance of $E$ from the rim of $K_i$-oriented and the rim of $K_j$-oriented $\ell$-isogeny volcanoes respectively. By Proposition \ref{joint-ascend}, we only need one of the endomorphisms $\theta_i, \theta_j$ to ascend to a local root; the one farther from the rim. Given a primitively oriented curve $(E,\iota)$, we can compute the unique ascending isogeny representing the path to the rim in time polynomial in $\ell, m, \log p$ and in the length of the representation of $\iota_i, \iota_j$ by \cite[Lemma 7.10]{MW23}.
\end{proof}

\begin{algorithm}
\label{path to global roots}
\SetAlgoLined
\SetKwInOut{Input}{input}
\SetKwInOut{Output}{output}
\Input{A supersingular elliptic curvs $E/\Fbar_{p^2}$ and a pair of endomorphisms $\theta_i, \theta_j$ of $E$ such that the associated fields $K_i \cong \Q(\theta_i)$ and $K_j \cong \Q(\theta_j)$ satisfy the condition \ref{condition1}.} 
\Output{Paths from $E$ to all global roots.} 
Compute the discriminants $\Delta_i=f_iD_{K_i}, \Delta_j=f_jD_{K_j}$ of $\theta_i, \theta_j$.\\
$\{\ell_k\}_{k=1}^n \leftarrow$ the list of distinct prime factors of $\Delta_i\Delta_j$.\\
$L \leftarrow \{(j_E)\}$.\\
$L' \leftarrow \{\}$.\\
\For{$1\leq k \leq n$}{
    \For{$H=(j_E\cdots j_{H}) \in L$}{
      Find an $\ell_k$-isogeny path $H'=(j_H\cdots j_{H'})$ of length $m_k=\max\{v_{\ell_k}(f_i), v_{\ell_{k}}(f_j)\}$ to a local root from $E(j_H)$.\\
      Store the path $HH'=(j_E \cdots j_H \cdots j_{H'})$ in $L'$.\\
      $\theta_i, \theta_j \leftarrow \theta_i/\ell_k^{v_{\ell_k}(f_i)}, \theta_j/\ell_k^{v_{\ell_k}(f_j)}$.\\
    \If{\text{$j_{H''}$ is $\ell_k$-isogeneous to $j_{H'}$ and it is a local root}}{Store the path $H''=HH'(j_{H''})=(j_E\cdots j_H \cdots j_{H'}j_{H''}) \in L'$.}
     }
     $L \leftarrow L'$\\
     $L' \leftarrow \{\}$
     }
     \KwRet{$L'$}
\caption{Ascending to the global roots}
\end{algorithm}

\begin{proposition} Let $E/\Fbar_{p^2}$ be a supersingular elliptic curve and let $\ell \neq p$ be a prime. Suppose that we are given a pair of endomorphisms $\theta_i, \theta_j \in \End(E)$ such that the associated fields $K_i \cong \Q(\theta_i)$ and $K_j\cong \Q(\theta_j)$ satisfy the condition \ref{condition1}. Let $\iota_i, \iota_j$ be the orientations on $E$ induced by $\theta_i, \theta_j$ respectively and let $\jmath=(\iota_i, \iota_j)$ be a $K_i, K_j$-orientation on $E$. Let $\Delta_i=f_iD_{K_i}, \Delta_{j}=f_j D_{K_j}$ be the discriminants of $\theta_i, \theta_j$ respectively where $f_i, f_j$ are the conductors. Let $(\ell_k)_{k=1}^n$ be the list of distinct prime factors of $\Delta_i, \Delta_j$ with $L = \max_{k=1}^n\{\ell_k\}$, $m_k=\max\{v_{\ell_k}(f_i), v_{\ell_k}(f_j)\}$, and $M=\max_{k=1}^n\{m_k\}$. Algorithm \ref{path to global roots} produces $e(\mathfrak{O}_{i,j})$ paths of length at most $n+\sum_{k=1}^n m_k$ from $(E,\jmath)$ to all global roots and runs in time polynomial in $L, M, \log p, \Delta_i\Delta_j$ and in the length of the representation of $\iota_i, \iota_j$. 
\end{proposition}
\begin{proof}

In each iteration of \textbf{for} loop, for all paths $H$ in $L$, we ascend from $E(j_H)$ to a local root in the $\ell_k$-connected component of $E(j_H)$ using Algorithm \ref{path to local roots}. Note that the ascending step only occurs if $\ell_k \mid f_if_j$. The number of local roots is given by the local embedding number $e_{\ell_k}(\mathfrak{O}_{i,j})$, which is at most $2$ by Lemma \ref{number of local roots}. If $\ell_k \mid \frac{D_{K_i}D_{K_j}}{4}$, then there are two local roots and they must be $\ell_k$-isogenous to each other. Otherwise, there is a unique local root. When we check if a neighborhood $j_{H''}$ of $j_{H'}$ is a local root as well, we use the property that an $\ell$-isogeny $\varphi: E(j_{H'}) \rightarrow E(j_{H''})$ is ascending if and only if $[\ell]^2 \mid \varphi \circ \theta \circ \hat{\varphi}$ in $\End(E(j_{H''}))$ \cite[Proposition 4.7]{ACL+23} using the division algorithm \cite[Algorithm 1]{MW23}, which takes time in polynomial in $\log p$. The runtime of each iteration is polynomial in $\ell_k, m_k, \log p$ and in the length of the representation of $\iota_i, \iota_j$ by Proposition \ref{path to local roots complexity}.

Recall that if two quaternion maximal orders have a connecting ideal of $\ell$-power reduced norm, then their localizations at any prime $\ell'\neq \ell$ are identical. Hence, $\left(\mathfrak{O}_{E({j_H})}\right)_{\ell'}=\left(\mathfrak{O}_{E(j_{H'})}\right)_{\ell'}$ for primes $\ell \neq \ell_{k}$. This implies they are not connected by an ideal of $\ell'$-power reduced norm for any $\ell' \neq \ell_k$. Hence, at the end of the $k$-th iternation, there are $e_{\ell_1}(\mathfrak{O}_{i,j})\cdots e_{\ell_k}(\mathfrak{O}_{i,j})$ elements in $L$. The global embedding number is given by
$$e(\mathfrak{O}_{i,j})=\prod_{\ell} e_\ell(\mathfrak{O}_{i,j})=\prod_{k=1}^n e_{\ell_k}(\mathfrak{O}_{i,j}).$$ Hence, Algorithm \ref{path to global roots} detects all global roots and returns paths to them.

   It's a standard result that the number of prime factors of an integer $m>2$ is bounded above by$$\sigma(m)=\frac{\log(m)}{\log\log(m)} \big(1 + O\big(1/\log\log(m)\big)\big).$$
\href{https://mathoverflow.net/questions/23867/bound-on-the-number-of-prime-factors-of-logarithmically-rough-numbers?_gl=1*1ksogtu*_ga*NDM3MDE2MDUyLjE3MjY1NTAyNDY.*_ga_S812YQPLT2*MTczMTYxNDQzNS4xMjUuMS4xNzMxNjE0NDM3LjAuMC4w}{link}
It follows that the number of iterations is $n< \sigma(\Delta_i \Delta_j)$ and $e(\mathfrak{O}_{i,j})<2^{\sigma(D_{K_i}D_{K_j}/4)} \cong D_{K_i}D_{K_j}$. Hence, the overall runtime of Algorithm \ref{path to global roots} is polynomial in $L, M, \log p, \Delta_i\Delta_j$ and in the length of the representation of $\iota_i, \iota_j$. 
\end{proof}

\begin{details}
    \begin{tbox}
        $e(\mathfrak{O}_{i,j})<2^{\sigma(D_{K_i} D_{K_j})}<10^{\sigma(D_{K_i} D_{K_j})}\cong D_{K_i}D_{K_j}$
    \end{tbox}
\end{details}

\begin{details}
    \begin{tbox}
        Notations for \cite{BE92}:
        \begin{enumerate}
            \item $\Phi_L:$ A definite quaternion algebra over $\Q$ with discriminant $L$.
            \item $o(d)$ is the maximal order of $\Q(\sqrt{-d})$.
            \item $h(d)$ is the class number of $\Q(\sqrt{-d}).$
            \item $e(o(d), \calO)$ is the number of (optimal) embeddings of $o(d)$ into $\calO$. Not given the definition (double check). Proposition 8 states that $e(o(d),\calO) \leq a(\calO)h(o(d))$.
            \item For an order $\calO \subseteq \Phi_L$, $a(\calO):=\#\aut(\calO)$.
        \end{enumerate}

        \href{https://math.stackexchange.com/questions/4291124/an-upper-bound-for-class-numbers-of-imaginary-quadratic-fields}{An upper bound for the class number:} $h_k<\frac{d}{2}$.

        Analytic class number formula: $h_D= O(\left|D\right|^{1/2}\log \left|D\right|)$ \href{Computing Hilbert Class Polynomials}{Lemma 1}.

        Also, see \href{https://studenttheses.uu.nl/bitstream/handle/20.500.12932/41776/Master%27s%20thesis%20Corijn%20Rudrum.pdf?sequence=1&isAllowed=y}{Elliptic curves and lower bounds for class numbers}.ch
    \end{tbox}
\end{details}

\begin{details}
    \begin{tbox}
            What if we go over prime factors of $\disc(\mathfrak{O})$ in size $O(\log p)$? Choose $B = O(\log p)$ and runs over $\ell_k < B$. For primes bigger than $B$, they are big so there can't be many of them and they are very close to the rim (assume $\Delta_i, \Delta_j< p$ or something. So $\disc(\mathfrak{O})< p^2$)
    \end{tbox}
\end{details}
\begin{details}
\begin{tbox}
    \begin{remark}
    Consider the case where ideal class group is cylic. \href{https://math.stackexchange.com/questions/75762/on-the-class-group-of-an-imaginary-quadratic-number-field}{link}
\end{remark}
\end{tbox}

\end{details}

\section{follow-up study}
\begin{enumerate}
    \item Show that local roots, when $ij+ji = s\neq 0 \in \Z$, form a cycle.
    \item Implement the algorithms for the probabilistic polynomial time reductions between endomorphism ring problem and one endomorphism problem (single-orientation).
\end{enumerate}

\end{extra}

\bibliography{lattices}{}
\bibliographystyle{plain}

\end{document}